\documentclass[11pt,reqno]{amsart}
\usepackage[utf8]{inputenc}
\usepackage{amssymb,a4wide,esint,enumitem,plain}
\newtheorem{theorem}{Theorem}[section]
\newtheorem{lemma}[theorem]{Lemma}
\newtheorem{proposition}[theorem]{Proposition}
\newtheorem{claim}[theorem]{Claim}
\theoremstyle{definition}
\newtheorem{remark}[theorem]{Remark}
\numberwithin{equation}{section}
\def\psl#1#2{\left(#1,#2\right)}
\def\psh#1#2{\left(#1,#2\right)_{\dot H^1}}
\def\pshb#1#2{\left(#1,#2\right)_{\dot H^1_\ell}}

\def\pshbbk#1#2{\left(#1,#2\right)_{\dot H^1_{{\boldsymbol{\ell}}_k}}}
\def\pshbbun#1#2{\left(#1,#2\right)_{\dot H^1_{{\boldsymbol{\ell}}_1}}}
\begin{document}
\title[Soliton collisions for the critical wave equation]{Inelasticity of soliton collisions for the $5$D energy critical wave equation}
\author[Y.~Martel]{Yvan Martel}
\address{CMLS, \'Ecole polytechnique, CNRS, Universit\'e Paris-Saclay, 91128 Palaiseau Cedex, France}
\email{yvan.martel@polytechnique.edu}
\author[F.~Merle]{Frank Merle}
\address{AGM, Universit\'e de Cergy Pontoise and Institut des Hautes \'Etudes Scientifiques, CNRS, 95302 Cergy-Pontoise, France}
\email{merle@math.u-cergy.fr}
\begin{abstract}
For the focusing energy critical wave equation in 5D, we construct a solution showing the inelastic nature of the collision of two solitons for any choice of sign, speed, scaling and translation parameters, except the special case of two solitons of same scaling and opposite signs.
Beyond its own interest as one of the first rigorous studies of the collision of solitons for a non-integrable model, the case of the quartic gKdV equation being partially treated in~\cite{MMcol1,MMalmost,MMimrn},
this result can be seen as part of a wider program aiming at establishing the soliton resolution conjecture for the critical wave equation. This conjecture has already been established in the 3D radial case in~\cite{DKM4} and in the general case in~3,~4 and~5D along a sequence of times in~\cite{DJKM1}.

Compared with the construction of an asymptotic two-soliton in~\cite{MMwave1}, the study of the nature of the collision requires a more refined approximate solution of the two-soliton problem and a precise determination of its space asymptotics.
To prove inelasticity, these asymptotics are combined with the method of channels of energy from~\cite{DKM4,KLS}. 
\end{abstract}
\maketitle
\section{Introduction}
\subsection{Main result}
We consider the focusing energy critical nonlinear wave equation in 5D
\begin{equation}\label{wave}
\partial_t^2 u - \Delta u - |u|^{\frac 4{3}} u = 0, \quad (t,x)\in {\mathbb{R}}\times {\mathbb{R}}^5.
\end{equation}
Recall that the Cauchy problem for equation~\eqref{wave} is locally well-posed in the energy space $\dot H^1\times L^2$, using suitable Strichartz estimates. 
See \emph{e.g.}~\cite{KM} and references therein.
Note that equation~\eqref{wave} is invariant by the $\dot H^1$ scaling:
if $u$ is solution of~\eqref{wave}, then for any $\lambda>0$, $u_\lambda$ defined by 
\[
u_\lambda(t,x)=\frac{1}{\lambda^{\frac 32}}u\left(\frac{t}{\lambda},\frac{x}{\lambda}\right),\quad
\|u_\lambda\|_{\dot H^1}=\|u\|_{\dot H^1},
\]
is also solution of~\eqref{wave}.
The energy 
$E(u(t),\partial_t u(t))$ and the momentum $M(u(t),\partial_t u(t))$ of an $\dot H^1\times L^2$ solution are conserved, where
\[
E(u,v) = \frac 12 \int v^2 + \frac 12 \int |\nabla u|^2 - \frac {3}{10} \int |u|^{\frac {10}{3}},
\quad
M(u,v)=\int v\nabla u.
\]
Recall also that the function $W$ defined by
\begin{equation}\label{defW}
W(x) = \left( 1+ \frac {|x|^2}{15}\right)^{-\frac{3}2},\quad 
\Delta W + W^{\frac 73}=0 , \quad x\in {\mathbb{R}}^5,
\end{equation}
is a stationary solution of~\eqref{wave}, called here \emph{ground state}, or \emph{soliton}.
By scaling, translation invariances and change of sign, we obtain a family of stationary solutions of~\eqref{wave} defined by $W_{\lambda,x_0,\pm}(x)=\pm \lambda^{-\frac 32}W\left(\lambda^{-1}(x-x_0)\right)$, where $\lambda>0$ and $x_0\in {\mathbb{R}}^5$.

Using the Lorentz transformation, we obtain \emph{traveling waves}.
For $\boldsymbol{\ell}\in {\mathbb{R}}^5$, with $|\boldsymbol{\ell}|< 1$, let
\begin{equation}\label{defWbb}
W_{{\boldsymbol{\ell}}}(x)=W\left(\left(\frac{1}{\sqrt{1-|\boldsymbol{\ell}|^2}}-1\right) \frac{\boldsymbol{\ell}(\boldsymbol{\ell}\cdot x)}{|\boldsymbol\ell|^2} +x\right).
\end{equation}
Then, the functions $w_{{\boldsymbol{\ell}},\pm}(t,x)=\pm W_{\boldsymbol{\ell}}(x-{\boldsymbol{\ell}} t)$, as well as rescaled and translated versions of $ w_{{\boldsymbol{\ell}},\pm}$, are solutions of~\eqref{wave}.
While the ground state $W$ is the unique, up to scaling invariance and sign change, radial stationary solution of~\eqref{wave}, there also exist non-radial solutions 
$Q\in \dot H^1({\mathbb{R}}^5)$ of the elliptic equation $\Delta Q+|Q|^{\frac 43}Q=0$ on ${\mathbb{R}}^5$; see~\cite{Di, dPMPP} for explicit constructions. However, no classification result is known for such solutions.
\smallbreak

The present paper adresses in the context of the wave equation~\eqref{wave} the classical question of the elastic or inelastic nature of the collision of traveling waves. Recall that such questions were first investigated by early numerical simulations~\cite{FPU,KZ} on some nonlinear models, and then mathematically studied by integrability (see \emph{e.g.}~\cite{KZ, LAX1, HIROTA, WT, Mi,DZ}), using the inverse scattering transform.
In such \emph{integrable cases}, the collision of any number of solitons is elastic, meaning that neither the number of solitons, nor their speeds, are changed by the collision. For models perturbative to integrable models, few results are known (see \emph{e.g.}~\cite{Pe,Mu}) and it is generally observed that elasticity is lost.

For nonlinear equations that are not close to any known integrable model, the collision problem is widely open. 
To the authors' knowledge, it was studied rigorously only for the
quartic gKdV equation on the line
\[
\partial_t u + \partial_x \left(\partial_x^2 u + u^4 \right)=0, \quad (t,x)\in {\mathbb{R}}^2,
\]
following Open Problem~4 in \S11 of~\cite{Mi}.
For two solitons with speeds $0<c_2<c_1$, the authors of the present paper have adressed the collision problem for the quartic gKdV equation in the following two asymptotic situations:
\begin{itemize}
\item[{(a)}] Solitons of very different speeds: $ 0< {c_2}\ll {c_1}$. See~\cite{MMcol1}
\item[{(b)}] Solitons with almost equal speeds: $0<1- {c_2}/{c_1}\ll 1$. See~\cite{MMalmost, MMimrn}.
\end{itemize}
Under condition (a) or (b), it is proved that in contrast with the integrable cases, the collision is always inelastic.
In~\cite{MMcol1,MMalmost}, the explicit computation of an approximate two-soliton solution for all $(t,x)\in {\mathbb{R}}^2$ describes globally the colllision and shows the presence of a non-trivial residual term after the collision. Moreover, as a consequence of the conservation of mass and energy, it is proved that the speeds and the sizes of the solitons are slightly altered by the interaction. In~\cite{MMimrn}, the strategy is different and could in principe cover the whole range of parameters $0<c_2<c_1$, though for technical reasons, the result is restricted to the case $1-c_2/c_1<1/4$. Indeed, an approximate solution of the two-soliton problem is computed only for large time, so that the solitons are decoupled regardless their respective speeds. Then, the defect due to the collision is propagated to any further time by special monotonicity properties of the gKdV equation. The present paper is partly inspired by this approach, replacing such monotonicity properties by the finite speed of propagation 
and the method of channels of energy introduced in~\cite{DKM4}.

Experimental and numerical results on collision are available for various physical contexts and nonlinear models, see \emph{e.g.}~\cite{Mi, SHIH, CGHHS, HHGY, BPS,LS}. It seems that inelasticity is found in all non-integrable models studied, which supports the general belief that the existence of pure multi-solitons is tightly related to integrability. We refer the reader to the more extended discussions in~\cite{Mi,CGHHS,MMcol1}.
\smallbreak

In this paper, we prove the existence of a solution of~\eqref{wave} which shows the inelastic nature of the collision of any two solitons, except the special case of same scaling and opposite signs.

\begin{theorem}\label{th.1}
For $k\in \{1,2\}$, let $\lambda_k^\infty>0,$ ${\mathbf y}_k^\infty\in {\mathbb{R}}^5$, $\epsilon_k\in \{\pm1\}$, ${\boldsymbol{\ell}}_k\in {\mathbb{R}}^5$ with $|{\boldsymbol{\ell}}_k|<1$,
and
\[
W_k^\infty(t,x)= \frac {\epsilon_k}{(\lambda_k^\infty)^{\frac 32}} W_{{\boldsymbol{\ell}}_k} \left( \frac{x - {\boldsymbol{\ell}}_k t -{\mathbf y}^\infty_k}{\lambda_k^\infty}\right).
\]
Assume that ${\boldsymbol{\ell}}_1\neq {\boldsymbol{\ell}}_2$ and
\begin{equation}\label{c.th.1}
\epsilon_1= \epsilon_2\quad \hbox{or}\quad \lambda_1^\infty\neq \lambda_2^\infty.
\end{equation}
Then, there exists a solution $u$ of~\eqref{wave} in the energy space such that 
\begin{enumerate}[label=\emph{(\roman*)}]
\item \emph{Two-soliton as $t\to+\infty$}
\[
\lim_{t\to +\infty} \left\| \nabla_{t,x} u(t) - \nabla_{t,x} \left(W_1^\infty(t)+W_2^\infty(t)\right)\right\|_{L^2} = 0.
\]
\item \emph{Dispersion as $t\to-\infty$.} There exists $C>0$ such that, for all $A>0$ large enough,
\begin{equation}\label{eth:1}
\liminf_{t\to -\infty} \|\nabla u(t)\|_{L^2(|x|>|t|+A)} \geq { C A^{-\frac 52}}.
\end{equation}
\end{enumerate}
\end{theorem}

The solution constructed in Theorem~\ref{th.1} is a two-soliton asymptotically as $t\to +\infty$ and it does not necessarily exist for all $t\in {\mathbb{R}}$. 
However, by finite speed of propagation and small data Cauchy theory, it is straightforward to justify that it can be extended uniquely as a solution of~\eqref{wave} for all $t\in {\mathbb{R}}$ in the region $|x|>|t|+A$, provided that $A$ is large enough.
Thus, the limit in~\eqref{eth:1} makes sense (see~\S\ref{sec:5.1} for details).
Since the estimate~\eqref{eth:1} gives an explicit lower bound on the loss of energy as dispersion as $t\to -\infty$, the solution $u$ is not a two-soliton asymptotically as $t\to -\infty$ and the collision is inelastic.
Note that the two-soliton could have any global behavior, like dislocation of the solitons and dispersion, blow-up or a different multi-soliton plus radiation, but the property that we obtain is universal and independent of the behavior on compact sets.
Note also that the only case left open by Theorem~\ref{th.1} corresponds, up to scaling and Lorentz invariance (and up to irrelevant translations),
to the \emph{dipole} case, \emph{i.e.} $\lim_{t\to +\infty}\|u(t)-W_{\boldsymbol{\ell}}(x-\boldsymbol{\ell} t)
+ W_{\boldsymbol{-\ell}}(x+\boldsymbol{\ell} t)\|_{\dot H^1}= 0$,
for some $\boldsymbol{\ell} \in \mathbb{R}^5$, $|\boldsymbol{\ell}|<1$.
We expect that a similar dispersion phenomenon takes place but possibly at lower order due to cancellation of the tail asymptotics by symmetry.
 
In the case of $K$ solitons with $K\geq 3$, existence of an asymptotic multi-soliton at $+\infty$ still holds for collinear speeds from~\cite{MMwave1}. Applying the same strategy, inelasticity is proved under a simple explicit non-vanishing condition which generalizes~\eqref{c.th.1}. See details in \S\ref{sec:10}.
\smallbreak

The interest of this work is twofold.
A main motivation is to continue the authors' program on the collision of solitons for non-integrable equations.
It is the first non-integrable model for which we are able to prove inelasticity without restriction on the relative sizes or speeds of the solitons except the dipole case $\epsilon_1=-\epsilon_2$ and $\lambda_1^\infty=\lambda_2^\infty$.
We also study the nature of soliton collisions because of its importance in the context of the soliton resolution conjecture for equation~\eqref{wave}. A particular case of this conjecture says that any global and bounded solution of~\eqref{wave} in the energy space should decompose as $t\to +\infty$ as a finite sum of solitons plus a dispersive part. This conjecture was proved in~\cite{DKM1,DKM4} for the~$3$D radial case.
In~\cite{DKM6,DJKM1}, the above version of the soliton resolution conjecture was proved in the non-radial case for a sequence of times $t_n \to +\infty$ in~3,~4 and~5D.
We also refer to previous results of classification in~\cite{DMwave,NS1,KNS1,KNS2} and to constructions of special solutions in~\cite{KST,JJblowup,JJ,JL}.
We expect that, beyond its own interest, the full understanding of the collision problem will be a key to the proof of the soliton resolution conjecture for the whole sequence of time.
\subsection{Outline of the proof}\label{simpl}
The strategy of the proof is to construct a refined approximate solution of the two-soliton problem that displays an explicit dispersive radial part at the leading order and then to propagate the dispersion for any negative time at the exterior of large cones by finite speed of propagation and the method of channels of energy.
\smallbreak

First, we construct a \emph{refined approximate solution} to the two-soliton problem for large $t>0$ of the form
$\vec {\mathbf W} = \vec W_1+\vec W_2+\vec v_1+\vec v_2,$
where $\vec W_1$ and $\vec W_2$ are two solitons with time dependent scaling and translation parameters, and
$\vec v_1$, $\vec v_2$ are correction terms improving the simpler approximate solution used in~\cite{MMwave1}.
These correction terms of size $t^{-2}$ in the energy space are solutions of non-homogeneous wave equations whose source terms are the main order of the nonlinear interactions of size $t^{-3}$ between the two solitons.
In this way, $\vec {\mathbf W}$ is an approximate solution of the two soliton problem at order $t^{-4}$.
Such refined approximate solutions were introduced in several other situations related to blow up or soliton interactions, see \emph{e.g.}~\cite{Mizumachi,MeRa,RR,MMcol1,MMalmost,MMimrn,JJblowup,MRnls,GLPR}.
In the case of the gKdV equation~\cite{MMimrn}, since solitons decay exponentially in space, the method of separation of variables applies and correction terms have simple expressions in terms of solutions of elliptic problems. In the present paper, this method would lead to correction terms not belonging to the energy space
(see \emph{e.g.}~\cite{JJblowup}). Since the strategy is based on a close examination of the asymptotics of the approximate solution, using cut-off to balance artificial growth cannot be successful. This is the reason why we define $\vec v_1$, $\vec v_2$ as solutions of linear evolution problems with source terms. Now, $\vec v_1$ and $\vec v_2$ are much less explicit but they belong to the energy space and their asymptotics contain the desired information. Because of the specific forms of the source terms, their equations cannot be reduced to radial ones by the Lorentz transformation.
\smallbreak

The next step is to compute the \emph{space asymptotics of the radial part of the approximate solution.} The main asymptotic part of $\vec v_1$, $\vec v_2$ is explicit but it turns out not to channel any energy (as a soliton). In view of the formula of the fundamental solution of the wave equation in~$5$D, it is not clear how to obtain manageable expressions for the next orders of $\vec v_1$ and $\vec v_2$. Our strategy is to compute only the radial part of their asymptotics using spherical means and reduction to a $1$D problem. The computation reveals an explicit dispersive tail for the radial parts of $\vec v_1$ and $\vec v_2$ for large positive times. It is remarkable that understanding only the radial component of the approximate solution is sufficient to treat all cases of two solitons except the dipole.
\smallbreak

Finally, we \emph{propagate the dispersion by the method of channels of energy,} which is a refined characterization of dispersion for wave type equations introduced by Duyckaerts, Kenig and Merle in~\cite{DKM4}.
We check that under the non-vanishing condition~\eqref{c.th.1}, summing the dispersive tails of $\vec v_1$ and $\vec v_2$, 
the radial part of the approximate solution has itself a non-zero dispersive tail for large positive times.
As in~\cite{MMwave1} and several other works related to the construction of multi-solitons (see references in~\S\ref{sec:5}), the two-soliton is constructed by compactness using the approximate solution $\vec{\mathbf W}$.
We also prove that the non-zero dispersive tail of the approximate solution is greater than the error terms so that it is still visible in the two-soliton.
The method of channels of energy (see~\cite{DKM4,DKM6,KLLS,KLS}) then allows us to propagate the dispersion for any negative time at the exterior of large cones. 
Moreover, from Theorem~2 of~\cite{DKM10}, the solution behaves asymptotically as $t\to -\infty$ as a non-zero solution of the linear wave equation in the region $|x|>|t|+A$ for $A$ large.
\smallbreak

We expect that our method can solve the same problem for 
odd space dimensions larger than~$5$. The method should also extend to other wave type equations.

\subsection{Notation}\label{s:not}
The canonical basis of ${\mathbb{R}}^5$ is denoted by $(\mathbf{e}_1,\ldots,\mathbf{e}_5)$.
We denote for real-valued functions
\[\psl v {\tilde v} =\int v \tilde v,\quad \|v\|_{L^2}^2 = \int |v|^2,\quad \psh v {\tilde v} =\int \nabla v \cdot \nabla {\tilde v} ,
\quad \| v\|_{\dot H^1}^2=\int |\nabla v|^2.\]
For 
\[
\vec v = \left(\begin{array}{c}v \\z\end{array}\right),
\ \vec {\tilde v} = \left(\begin{array}{c}\tilde v \\ \tilde z\end{array}\right),\]
set
\[ \psl {\vec v} {\vec {\tilde v}} = \psl v {\tilde v} + \psl z{\tilde z},\
 \left(\vec v,\vec {\tilde v}\right)_{\dot H^1\times L^2} = (v,\tilde v)_{\dot H^1} + \psl z{\tilde z}.
\]
We denote by $d\omega$ the Lebesgue measure on the sphere, and by
\[\fint_{|y-x|=r} v(y) d\omega(y)=\frac 3{8\pi^2}r^{-4}\int_{|y-x|=r} v(y) d\omega(y)\]
the average of a function $v$ over the sphere of ${\mathbb{R}}^5$ of center $x\in {\mathbb{R}}^5$ and radius $r>0$.

Set $\langle x \rangle = (1+|x|^2)^{\frac 12}$.
Let
\begin{equation}
\label{aL}
\Lambda = \frac 32 + x \cdot \nabla ,
\quad \widetilde \Lambda = \frac 5 2 + x \cdot \nabla ,\quad
\widetilde\Lambda \nabla=\nabla\Lambda,\quad 
\vec \Lambda = \left(\begin{array}{c}\widetilde \Lambda \\ \Lambda\end{array}\right).
\end{equation}
When $x_1$ is seen as a specific coordinate, denote
\[
\overline x = (x_2,\ldots, x_5),
\quad \overline \nabla v = (\partial_{x_2} v, \ldots, \partial_{x_5} v),
\quad \overline \Delta v = \sum_{j=2}^5 \partial_{x_j}^2 v.
\]
For $-1< \ell<1$, set 
\[ \pshb v{\tilde v} =(1-\ell^2)\int (\partial_{x_1} v) (\partial_{x_1}\tilde v)+ \int \overline \nabla v \cdot \overline \nabla \tilde v, 
\quad \|v\|_{\dot H^1_\ell}^2= \pshb vv,\]
\[
x_\ell = \left( \frac{x_1-\ell t}{\sqrt{1-\ell^2}},\overline x\right),\quad
A_\ell = \partial_t + \ell \partial_{x_1},\quad
B_\ell = \partial_t^2 - \ell^2\partial_{x_1}^2 = A_\ell^2 - 2 \ell \partial_{x_1} A_\ell,
\]
\[
\Lambda_\ell = \frac 32 + (x-\ell t \mathbf{e}_1) \cdot \nabla,\quad
\Delta_\ell = (1-\ell^2) \partial_{x_1}^2 + \overline \Delta.
\]
For $\gamma>0$ small to be fixed later, set
\begin{equation}\label{phia}
\varphi_\gamma (x) = (1+|x|^2)^{- \gamma}
\end{equation}

We recall standard Sobolev and H\"older inequalities
\begin{equation}\label{sobolev}
\|u\|_{L^{10/3}}\lesssim \|u\|_{\dot H^1},\quad \|u\|_{L^{10}} \lesssim \|u\|_{\dot H^2},
\end{equation}
\begin{equation}\label{holder1}
\int |u| |v| |w|^{\frac 43} \lesssim \|u\|_{L^{10/3}}\|v\|_{L^{10/3}}\|w\|_{L^{10/3}}^{4/3}\lesssim
\|u\|_{\dot H^1}\|v\|_{\dot H^1}\|w\|_{\dot H^1}^{4/3},
\end{equation}
\begin{equation}\label{lor}
\|uv\|_{L^{10/7}}\lesssim \|u\|_{L^{10/3}}\|u\|_{L^{5/2}},\quad 
\|uvw\|_{L^{10/7}}\lesssim \|u\|_{L^{10/3}}\|v\|_{L^{10/3}}\|w\|_{L^{10}}.
\end{equation}
\subsection*{Acknowledgements} 
This work was partially supported by ERC 291214 BLOWDISOL.
This material is partly based upon work supported by the National Science Foundation under Grant No. 0932078 000 while Y.M. was in residence at the Mathematical Sciences Research Institute in Berkeley, California, during the Fall 2015 semester.
\section{Preliminaries}
We gather in this section preliminary results on the linearized operator around $W$, on the linear homogeneous and non-homogeneous wave equations in $5$D, on the method of channels of energy and on the Cauchy problem for~\eqref{wave}.
\subsection{Linearized operator around the soliton}\label{sec21}
Let
\[
 L = -\Delta - \frac 73 W^{\frac 43} ,\quad 
(L v,v) = \int \left(|\nabla v|^2 - \frac 73 W^{\frac 43} v^2\right),
\]
\[ H = \left(\begin{array}{cc} L & 0 \\0 & {\rm Id}\end{array}\right),\quad 
(H \vec v,\vec v) = (L v,v)+ \|z\|_{L^2}^2 \quad \text{for} \quad\vec v = \left(\begin{array}{c}v \\z\end{array}\right).
\]
For $\vec v$ small in the energy space, we recall the expansion of the energy
\[
E(W+v,z) 
=E(W,0)+\frac 12(H \vec v,\vec v) + O(\|v\|_{\dot H^1}^3).
\]
\begin{lemma}[Properties of $L$]\label{le:Q}\leavevmode
\begin{enumerate}[label=\emph{(\roman*)}]
\item \emph{Spectrum.} The operator $L$ on $L^2$ with domain $H^2$ is a self-adjoint operator with essential spectrum $[0,+\infty)$, no positive eigenvalue and only one negative eigenvalue $-\lambda_0$, with a smooth radial positive eigenfunction $Y \in \mathcal S({\mathbb{R}}^5)$.
Moreover,
\[
L (\Lambda W) = L (\partial_{x_j} W) =0, \quad \hbox{for any $j=1,\ldots,5$.}
\]
\item \emph{Coercivity results.}
There exists $\mu>0$ such that, for $\gamma>0$ small enough, for all $v \in \dot H^1$,
\[
({Lv},v)\geq \mu \| v\|_{\dot H^1}^2 -\frac 1{\mu} \left\{ \psh {v}{\Lambda W}^2 + |\psh v{\nabla W}|^2+\psh v{W}^2\right\},
\]
 \[
({Lv},v)\geq \mu \| v\|_{\dot H^1}^2 -\frac 1{\mu} \left\{ \psh {v}{\Lambda W}^2 + |\psh v{\nabla W}|^2+\psl v{Y}^2\right\},
\]
\[
 \int \left(|\nabla v|^2 \varphi_\gamma^2- \frac 73 W^{\frac 43}v^2 \right)
 \geq \mu \int |\nabla v|^2 \varphi_\gamma^2 
-\frac 1{\mu} \left\{ \psh v{\Lambda W}^2 +|\psh v{\nabla W}|^2+\psl v{Y}^2\right\}.
\]
\item \emph{Inversion of $L$.}
Let $F\in \dot H^{-1}$ be such that $\psl F{\Lambda W}= |\psl F{\nabla W}|=0.$
Then, there exists a unique $V\in \dot H^1$ such that $\psh V{\Lambda W} = |\psh V{\nabla W}| =0$ and $LV =F$.
Moreover, if $F$ is of class $\mathcal C^p$, $p\geq 1$ and satisfies, for some $0<\delta<1$, for all $\alpha\in {\mathbb{N}}^5$, $|\alpha|\leq p$, for all $x \in {\mathbb{R}}^5$,
\[
 |\partial^{\alpha} F(x)|\lesssim \langle x\rangle^{-5-\delta},\]
then $V$ is of class $\mathcal C^{p+1}$ and satisfies, for all $\alpha'\in {\mathbb{N}}^5$, $2\leq |\alpha'|\leq p+1$,
for all $x \in {\mathbb{R}}^5$,
\[
|V(x)|\lesssim \langle x \rangle^{-3},\quad 
|\nabla V(x)|\lesssim \langle x \rangle^{-4},\quad |\partial^{\alpha'} V(x)|\lesssim \langle x \rangle^{-5}.\]
\end{enumerate}
\end{lemma}
\begin{proof}
The spectral properties of $L$ in (i) are standard and easily checked. The coercivity properties (ii) are given respectively in~\cite{OR},~\cite{DMwave} and~\cite{MMwave1}.
To prove (iii), we first define
\[
\mathcal Y^\perp=\{v \in \dot H^1,\ \psh v{\Lambda W} = |\psh v{\nabla W}| = \psh v W=0 \},
\]
\[
\mathcal Y_0^\perp=\{v \in \dot H^1,\ \psh v{\Lambda W} = |\psh v {\nabla W}| = 0 \}.
\]
Denote $M V : = V + \Delta^{-1} (\frac 73 W^{\frac 43} V) $, so that $- \Delta M = L$.
For $f\in \dot H^1$, by~\eqref{holder1}, we have
$
|(W^{\frac 43} V,f)| 
\lesssim \|V\|_{\dot H^1}\|f\|_{\dot H^1}$.
It follows that $M$ is continuous in $\dot H^1$.
We check that the image of $\mathcal Y^\perp$ by $M$ is included in $\mathcal Y^\perp$.
Indeed, for any $V\in \dot H^1$, $\psh {M V}{\Lambda W}=- \psl {\Delta M V} {\Lambda W} = \psl V{L \Lambda W}=0$,
 similarly $\psh {M V}{\nabla W}=0$, and $\psh{M V}{W}=- \psl{\Delta M V}{W} = \psl{V}{LW}=
-\frac 43 \psh VW$ since $LW = \frac 43 \Delta W$ from $\Delta W + W^{\frac 73}=0$.
Moreover, $M$ is coercive in $\mathcal Y^\perp$ from~(ii), since for $\mu>0$,
for all $V\in \mathcal Y^\perp$, $\psh{M V}{V} = (L V,V)\geq \mu \|V\|_{\dot H^1}^2$.
Thus, for any $f\in \mathcal Y^\perp$, there exists a unique $V\in \mathcal Y^\perp$ such that $M V = f$.

Let now $f \in \mathcal Y_0^\perp$ and set $f = f^\perp + a W$, where $a$ is such that $f^\perp\in \mathcal Y^\perp$. Let $V^\perp\in \mathcal Y^\perp$ be such that $M V^\perp = f^\perp$.
Note that by $\Delta W + W^{\frac 73}=0$, one has $M W = W + \Delta^{-1}(\frac 73 W^{\frac 43} V) = -\frac 43 W$.
Let $V = V^\perp - \frac 34 a W$. Then $V\in \mathcal Y^\perp_0$ and $M V= f$, in particular, $LV = -\Delta f$.
To conclude, note that setting $F=-\Delta f$, the assumptions on $F$ are equivalent to $f\in \mathcal Y_0^\perp$.

Now, we prove decay properties of $V$ assuming further that $F$ is of class $\mathcal C^p$, for $p\geq 1$ and satisfies for some $0<\delta<1$, for all $\alpha\in {\mathbb{N}}^5$, $|\alpha|\leq p$, for all $x \in {\mathbb{R}}^5$,
$|\partial^{\alpha} F(x)|\lesssim \langle x\rangle^{-5-\delta}$.
Write $-\Delta V = F+\frac 73 W^{\frac 43} V$. 
First, recall that by the explicit expression of the fundamental solution $\frac 1{8\pi^2} |x|^{-3}$ of 
the Laplace equation in ${\mathbb{R}}^5$ (see \emph{e.g.} \S2.2 of~\cite{Ev}),
the unique (in the class of functions going to $0$ at $\infty$) solution $U$ of $-\Delta U=F$ in ${\mathbb{R}}^5$ is of class $\mathcal C^{p+1}$ and satisfies, for all $\alpha'\in {\mathbb{N}}^5$, $2\leq |\alpha'|\leq p+1$, for all $x \in {\mathbb{R}}^5$,
\[
|U(x)|\lesssim \langle x \rangle^{-3},\quad 
|\nabla U(x)|\lesssim \langle x \rangle^{-4},\quad |\partial^{\alpha'} U(x)|\lesssim \langle x \rangle^{-5}.\]
Second, since $F\in L^2$ and $W^{\frac 43} V\in L^2$ (by the Hardy inequality), we have $V \in \dot H^2\subset L^{10}$. 
Let
\[
w(x) = c \int W^{\frac 43}(x-y) V(x-y) |y|^{-3} dy
\]
be solution of $-\Delta w = \frac 73 W^{\frac 43} V$.
Then, by Holder inequality,
\[
|w(x)| \lesssim \|V\|_{L^{10}} \left(\int \langle x-y\rangle^{-\frac {40}{9}} |y|^{-\frac {10}3} dy \right)^{\frac 9{10}}.
\]
Since
\[
\int_{|y|< \frac 12 \langle x\rangle} \langle x-y\rangle^{-\frac {40}{9}} |y|^{-\frac {10}3} dy\lesssim 
\langle x\rangle^{-\frac {40}{9}}\int_{|y|< \frac 12 \langle x\rangle} |y|^{-\frac {10}3} dy\lesssim \langle x\rangle^{-\frac {25}{9}},
\]
and
\[
\int_{|y|> \frac 12 \langle x\rangle} \langle x-y\rangle^{-\frac {40}{9}} |y|^{-\frac {10}3} dy\lesssim 
\langle x \rangle^{-\frac 83} \int \langle x -y \rangle^{-\frac {40}9}|y|^{-\frac {2}3} dy \lesssim \langle x \rangle^{-\frac 83},
\]
this gives $|V(x)|\lesssim \langle x\rangle^{-\frac {12}5}$ and thus $W^{\frac 43}(x) |V(x)|\lesssim \langle x\rangle^{-\frac {32}5}$.
We bootstrap this estimate to we find the desired estimates on $V$ and $\nabla V$.
For estimates on $\partial^{\alpha'}V$ for $|\alpha'|\geq 2$, we write $-\Delta (\partial^{\alpha'}V) =\partial^{\alpha'} (\frac 73 W^{\frac 43} V+F )$ and proceed similarly by induction on $|\alpha'|\leq p+1$.
\end{proof}
For $-1<\ell<1$, let
\[
W_{\ell }(x) = W\left(\frac {x_1}{\sqrt{1-\ell^2}}, \overline x\right),\quad 
(1-\ell^2) \partial_{x_1}^2 W_{\ell} + \overline \Delta W_{\ell} + W_\ell^{\frac 73}=0,
\]
so that $u(t,x) = W_{\ell }\left(x_1 - \ell t,\bar x\right)$ is a solution of~\eqref{wave}. 
Note that
\[
E(W_\ell,-\ell \partial_{x_1} W_\ell)-\ell^2\int|\partial_{x_1}W_\ell|^2
 = (1-\ell^2)^{\frac 12} E(W,0).
\]
Let 
\begin{align*}
& 
L_{\ell} = -(1-\ell^2) \partial_{x_1}^2 -\overline \Delta - \frac 73 W_\ell^{\frac 43} ,\\
&(L_\ell v,v) = \int \left( (1-\ell^2) |\partial_{x_1} v|^2+|\overline \nabla v|^2 -\frac 73 W_\ell^{\frac 43} v^2\right)
,\\
& H_\ell = \left(\begin{array}{cc} -\Delta -\frac 73 W_\ell^{\frac 43} & -\ell \partial_{x_1} \\ \ell \partial_{x_1}& {\rm Id}\end{array}\right),\quad
(H_\ell \vec v,\vec v) = 
(L_\ell v,v) + \| \ell \partial_{x_1} v + z\|_{L^2}^2.
\end{align*}
Let
\[
J=\left(\begin{array}{cc}0 & 1 \\-1 & 0\end{array}\right).
\]
The following functions appear when studying the properties of the operators $H_\ell$ and $H_\ell J$
\begin{align*}
\vec Z_\ell^\Lambda = \left(\begin{array}{c} \Lambda W_\ell \\ - \ell \partial_{x_1}\Lambda W_\ell\end{array}\right),\quad
\vec Z_\ell^{\nabla_j} = \left(\begin{array}{c} \partial_{x_j} W_\ell \\ - \ell \partial_{x_1}\partial_{x_j} W_\ell\end{array}\right),\quad
\vec Z_\ell^W = \left(\begin{array}{c} W_\ell \\ - \ell \partial_{x_1}W_\ell \end{array}\right),
\end{align*}
\begin{align*}
Y_{\ell}(x) = Y\left(\frac {x_1}{\sqrt{1-\ell^2}},\overline x\right),
\quad 
\vec Z_\ell^\pm = \left(\begin{array}{c} \left(\ell \partial_{x_1} Y_\ell 
\pm \frac {\sqrt{\lambda_0}}{\sqrt{1-\ell^2}} Y_\ell\right) e^{\pm \frac {\ell \sqrt{\lambda_0}}{\sqrt{1-\ell^2}}x_1} \\ 
Y_\ell e^{\pm \frac {\ell \sqrt{\lambda_0}}{\sqrt{1-\ell^2}}x_1 } \end{array}\right).
\end{align*}
We recall the following result from~\cite{CMkg} and~\cite{MMwave1}.
\begin{lemma}\label{surZZ} Let $-1<\ell<1$.
\begin{enumerate}[label=\emph{(\roman*)}]
\item\emph{Properties of $L_\ell$.}
\[
 L_\ell (\Lambda W_\ell) = L_\ell (\partial_{x_j} W_\ell)=0,\quad L_\ell Y_\ell = -\lambda_0 Y_\ell,\quad 
L_\ell W_\ell = - \frac 43 W_\ell^{\frac 73}.
\]
\item\emph{Properties of $H_\ell$ and $H_\ell J$.}
\begin{align*}
&H_\ell \vec Z_\ell^\Lambda= H_\ell \vec Z_\ell^{\nabla_j}=0,\quad
 H_\ell \vec Z_\ell^W= - \frac 43 \left(\begin{array}{c} W_\ell^{\frac 73} \\ 0 \end{array}\right),
\\
&\psl {H_\ell \vec Z_\ell^W}{\vec Z_\ell^W} = -\frac 43 \int W_\ell^{\frac {10}3},
\quad
 - H_\ell J (\vec Z_\ell^\pm) = \pm \sqrt{\lambda_0} (1-\ell^2)^{\frac 12} \vec Z_\ell^\pm,
 \\
&\left(\vec Z_\ell^\Lambda,\vec Z_\ell^W\right)_{\dot H^1\times L^2}=
\left(\vec Z_\ell^{\nabla_j},\vec Z_\ell^W\right)_{\dot H^1\times L^2}=0,
\quad \psl { \vec Z_\ell^\Lambda}{\vec Z_\ell^\pm}=
\psl { \vec Z_\ell^{\nabla_j}}{\vec Z_\ell^\pm}=0.
\end{align*}
\item\emph{Coercivity.}
There exists $\mu>0$ such that, for all $\vec v \in \dot H^1\times L^2$,
\[
(H_\ell \vec v,\vec v) \geq \mu \|\vec v\|_{\dot H^1\times L^2}^2 
- \frac 1{\mu}\left\{ \pshb{v}{\Lambda W_\ell}^2 +|\pshb {v}{\nabla W_\ell}|^2
+ \psl {\vec v}{\vec Z_\ell^{+}}^2 + \psl {\vec v}{\vec Z_\ell^{-}}^2\right\},
\]
\begin{multline*}
\int \left(|\nabla v|^2 \varphi_\gamma^2 -\frac 73 W_\ell^{\frac 43} v^2 + z^2 \varphi_\gamma^2
+ 2 \ell (\partial_{x_1} v) z \varphi_\gamma^2 \right)
\\ \geq\mu\int\left(|\nabla v|^2+z^2\right)\varphi_\gamma^2 
 - \frac 1{\mu }\left\{ \psh{v}{\Lambda W_\ell}^2+|\psh {v}{\nabla W_\ell}|^2
+ \psl {\vec v}{\vec Z_\ell^{+}}^2 + \psl {\vec v}{\vec Z_\ell^{-}}^2\right\}.
\end{multline*}
\end{enumerate}
\end{lemma}
We extend the above notation to any ${\boldsymbol{\ell}} \in {\mathbb{R}}^5$ such that 
$|\boldsymbol{\ell}|< 1$. The function $W_{{\boldsymbol{\ell}}}$ defined in~\eqref{defWbb} satisfies the equation
$
\Delta W_{\boldsymbol{\ell}}-({\boldsymbol{\ell}}\cdot\nabla)^2 W_{\boldsymbol{\ell}}+W_{\boldsymbol{\ell}}^{7/3}=0$.
Let
\[ L_{{\boldsymbol{\ell}}} = -\Delta -{\boldsymbol{\ell}}\cdot\nabla({\boldsymbol{\ell}}\cdot\nabla) - \frac 73 W_{\boldsymbol{\ell}}^{\frac 43} ,\quad H_{\boldsymbol{\ell}} = \left(\begin{array}{cc} -\Delta - \frac 73 W_{\boldsymbol{\ell}}^{\frac 43} & -{\boldsymbol{\ell}}\cdot \nabla \\ {\boldsymbol{\ell}} \cdot \nabla & {\rm Id}\end{array}\right),\] 
\begin{align*}
\vec Z_{{\boldsymbol{\ell}}}^\Lambda = \left(\begin{array}{c} \Lambda W_{{\boldsymbol{\ell}}} \\ - {{\boldsymbol{\ell}}}\cdot \nabla(\Lambda W_{{\boldsymbol{\ell}}})\end{array}\right),\quad
\vec Z_{{\boldsymbol{\ell}}}^{\nabla_j} = \left(\begin{array}{c} \partial_{x_j} W_{{\boldsymbol{\ell}}} \\ - {{\boldsymbol{\ell}}}\cdot \nabla (\partial_{x_j} W_{{\boldsymbol{\ell}}})\end{array}\right),\quad
\vec Z_{{\boldsymbol{\ell}}}^W = \left(\begin{array}{c} W_{{\boldsymbol{\ell}}} \\ - {{\boldsymbol{\ell}}}\cdot \nabla W_{{\boldsymbol{\ell}}} \end{array}\right),
\end{align*}
\[
Y_{{\boldsymbol{\ell}}}=Y\left(\left(\frac{1}{\sqrt{1-|\boldsymbol{\ell}|^2}}-1\right) \frac{\boldsymbol{\ell}(\boldsymbol{\ell}\cdot x)}{|\boldsymbol\ell|^2} +x\right),
\ 
 \vec Z_{{\boldsymbol{\ell}}}^\pm = \left(\begin{array}{c} \left({{\boldsymbol{\ell}}} \cdot \nabla Y_{{\boldsymbol{\ell}}} 
\pm \frac {\sqrt{\lambda_0}}{\sqrt{1-|{\boldsymbol{\ell}}|^2}} Y_{{\boldsymbol{\ell}}}\right) e^{\pm \frac { \sqrt{\lambda_0}}{\sqrt{1-|{{\boldsymbol{\ell}}}|^2}} ({\boldsymbol{\ell}} \cdot x)} \\ 
Y_{{\boldsymbol{\ell}}} e^{\pm \frac { \sqrt{\lambda_0}}{\sqrt{1-|{{\boldsymbol{\ell}}}|^2}} ({\boldsymbol{\ell}} \cdot x) } \end{array}\right).
\]
\subsection{On the linear homogeneous and non-homogeneous wave equations}\label{secc:2.2}
For $g,h\in \mathcal S({\mathbb{R}}^5)$, it is well-known (see, \emph{e.g.}~\cite{Ev} \S2.4) that 
 the solution $z$ of the homogeneous wave equation in ${\mathbb{R}}\times {\mathbb{R}}^5$
\[\left\{\begin{aligned}
&\partial_t^2 z - \Delta z = 0 \quad \hbox{on ${\mathbb{R}}\times {\mathbb{R}}^5$},\\
&z_{|t=0}=g,\quad \partial_t z_{|t=0} = h \quad \hbox{on ${\mathbb{R}}^5$}
\end{aligned}\right.\]
writes
\begin{multline}
z=\cos(t\sqrt{-\Delta}) g + \frac{\sin(t\sqrt{-\Delta})}{\sqrt{-\Delta}} h\\= \frac 13 \left[ \left(\frac \partial{\partial t}\right)\left( \frac 1t\frac \partial {\partial t}\right) \left( t^3 
\fint_{|y-x|=t} g(y) d\omega(y)\right)+
\left( \frac 1t\frac \partial {\partial t}\right) \left( t^3 
\fint_{|y-x|=t} h(y) d\omega(y)\right)\right].\label{g.Lin}
\end{multline}
For $f\in \mathcal S({\mathbb{R}}\times {\mathbb{R}}^5)$, we define 
\[
v(t) = - \int_t^{+\infty} \frac{\sin((s'-t)\sqrt{-\Delta})}{\sqrt{-\Delta}} f(s') ds'
= \int_0^\infty \frac{\sin(s\sqrt{-\Delta})}{\sqrt{-\Delta}} f(s+t) ds
\]
the unique solution of the non-homogeneous wave equation 
\[\partial_t^2 v - \Delta v = f \quad \hbox{on ${\mathbb{R}}\times {\mathbb{R}}^5$}\]
which converges to $0$ in the energy norm as $t\to +\infty$.
 From~\eqref{g.Lin}, one has
\begin{multline*}
\frac{\sin(s\sqrt{-\Delta})}{\sqrt{-\Delta}} f(s+t) = \frac 13 \left[\left( \frac 1\sigma \frac {\partial }{\partial \sigma}\right) \left(\sigma^3 \fint_{|y-x|=\sigma} f(\zeta,y) d\omega(y) \right) \right]_{\sigma=s,\zeta=t+s}
\frac{\sin(s\sqrt{-\Delta})}{\sqrt{-\Delta}} f(s+t)\\ = 
\frac 1{3 s} 
\left[ \frac {\partial }{\partial s} \left(s^3 \fint_{|y-x|=s} f(t+s,y) d\omega(y) \right) 
- s^3 \fint_{|y-x|=s} \partial_tf(t+s,y) d\omega(y)\right] .
\end{multline*}
Thus, integrating by parts in the variable $s$ and then changing variable,
\begin{align}
v(t,x)& = \frac 13 \int_0^{+\infty} \frac 1s 
\left[ \frac {\partial }{\partial s} \left(s^3 \fint_{|y-x|=s} f(t+s,y) d\omega(y) \right) 
- s^3 \fint_{|y-x|=s} \partial_tf(t+s,y) d\omega(y)\right] ds\nonumber\\
& = \frac 13 \int_0^{+\infty} \fint_{|y|=s} \left[ sf(t+s,x+y) - s^2 \partial_t f(t+s,x+y) \right] d\omega(y) ds.
\label{duhamel2}\end{align}
Now, we prove estimates on $v$ assuming bounds on $f$, $A_\ell f$ and $A_\ell^2 f$.
\begin{lemma}[Bounds for the non-homogeneous wave equation]\label{le:NH}
Let $-1< \ell <1$, $q \geq 2$ and $p>2$. Let $f$ be a smooth function such that, for all $(t,x)\in (1,+\infty)\times {\mathbb{R}}^5$,
\begin{equation}\label{a:F}
	 |A_\ell^m f(t,x)|\lesssim t^{-(q+m)} \langle x_\ell \rangle^{-p} \quad \hbox{for $m=0,1,2$.}
\end{equation}
Let $v$ be given by~\eqref{duhamel2}.
Then, for all $(t,x)\in (1,+\infty)\times {\mathbb{R}}^5$,
\begin{itemize}
\item if $q=2$ and $2< p<5$,
\begin{align*}
|v(t,x)| &\lesssim (t+\langle x_\ell \rangle)^{-2} \langle x_\ell \rangle^{- (p-2)} \log\left( 2+\frac{\langle x_\ell\rangle}t\right),\\
|\nabla v(t,x)| &\lesssim (t+\langle x_\ell \rangle)^{-1} t^{-1} \langle x_\ell \rangle^{- (p-1)},
\end{align*}
\item if $q> 2$ and $2< p <5$, 
\begin{align*}
|v(t,x)|& \lesssim (t+\langle x_\ell \rangle)^{-2} t^{-(q-2)}\langle x_\ell \rangle^{- (p-2)} ,\\
|\nabla v(t,x)|& \lesssim (t+\langle x_\ell \rangle)^{-1} t^{-(q-1)}\langle x_\ell \rangle^{- (p-1)} ,
\end{align*}
\item if $q\geq 2$ and $p=5$,
\begin{align*}
|v(t,x)|\lesssim (t+\langle x_\ell \rangle)^{-2} t^{-(q-2)}\langle x_\ell\rangle^{-3} \log(1+\langle x_\ell\rangle),\\
|\nabla v(t,x)|\lesssim (t+\langle x_\ell \rangle)^{-1} t^{-(q-1)}\langle x_\ell\rangle^{-4} \log(1+\langle x_\ell\rangle),
\end{align*}
\item if $q\geq 2$ and $p>5$,
\begin{align*}
|v(t,x)|&\lesssim (t+\langle x_\ell \rangle)^{-2} t^{-(q-2)}\langle x_\ell\rangle^{-3},\\
|\nabla v(t,x)|&\lesssim (t+\langle x_\ell \rangle)^{-1} t^{-(q-1)}\langle x_\ell\rangle^{-4} .
\end{align*}
\end{itemize}
\end{lemma}
\begin{remark}\label{rkp10}
Note the particular space-time decay properties of $v$: \emph{e.g.} in the case $q>2$ and $2<p<5$, we have 
$|v|\lesssim t^{-q}\langle x_\ell\rangle^{-(p-2)}$ and $|v|\lesssim t^{-(q-2)}\langle x_\ell\rangle^{-p}$ on $(1,+\infty)\times {\mathbb{R}}^5$.
\end{remark}
\begin{proof}[Proof of Lemma~\ref{le:NH}] We set ${\boldsymbol{\ell}} = \ell \mathbf{e}_1$.
First, we claim that for $(t,x)\in (1,+\infty)\times {\mathbb{R}}^5$,
\begin{equation}\label{mai1}
|v(t,x)|\lesssim J(t,x-{\boldsymbol{\ell}} t)\quad\hbox{where}\quad J(t,a)=\int|y|^{-3} (t+|y|)^{-q} \langle a + y-{\boldsymbol{\ell}} |y|\rangle^{-p} dy
\end{equation}
\begin{equation}\label{mai2}
|\nabla v(t,x)|\lesssim K(t,x-{\boldsymbol{\ell}} t)\quad\hbox{where}\quad K(t,a)=\int|y|^{-4} (t+|y|)^{-q} \langle a + y-{\boldsymbol{\ell}} |y|\rangle^{-p} dy.
\end{equation}
\emph{Proof of~\eqref{mai1}.}
From~\eqref{duhamel2}, $v$ writes
\begin{align*}
v(t,x) & = \frac {1}{8\pi^2} \int_0^{+\infty}\int_{|y|=s} \left[ sf(t+s,x+y) - s^2 \partial_t f(t+s,x+y) \right] d\omega(y) s^{-4} ds\\
& =\frac 1{8\pi^2} \int \left[ |y|^{-3} f(t+|y|,x+y) - |y|^{-2} \partial_t f(t+|y|,x+y)\right] dy.
\end{align*}
Set $g(t,x,y)=f(t+|y|,x+y)$.
Since
\[
\partial_{y_1} g(t,x,y) = \frac{y_1}{|y|} \partial_t f(t+|y|,x+y)+ \partial_{x_1} f(t+|y|,x+y)
\]
and, using the definition of $A_\ell$, $\partial_t f = A_\ell f - \ell \partial_{x_1} f$, we obtain
\[
\left(1-\ell \frac{y_1}{|y|}\right) \partial_t f(t+|y|,x+y) = A_\ell f(t+|y|,x+y) - \ell \partial_{y_1} g(t,x,y).
\]
Thus, integrating by parts,
\begin{multline*}
8\pi^2 v(t,x) = \int \left[ |y|^{-3} f(t+|y|,x+y) + \frac { \ell \partial_{y_1} g(t,x,y) -A_\ell f(t+|y|,x+y) }{|y|(|y|-\ell y_1)} \right] dy\\
=\int \left[ k(y) f(t+|y|,x+y) - h(y) A_\ell f(t+|y|,x+y)\right] dy,
\end{multline*}
where $h(y)= \frac 1{|y|(|y|-\ell y_1)}$ and $k(y)= |y|^{-3} -\ell \partial_{y_1}h(y)$.
We note that $ \frac 1{|y|(|y|-\ell y_1)} \lesssim |y|^{-2}$ and $\left|\partial_{y_1} \left(\frac 1{|y|(|y|-\ell y_1)}\right)\right|\lesssim |y|^{-3}$ on ${\mathbb{R}}^5$. Thus, using~\eqref{a:F},
\begin{multline*}
|v(t,x)| \lesssim \int |y|^{-3} |f(t+|y|,x+y)| dy + \int |y|^{-2} |A_\ell f(t+|y|,x+y)| dy\\
\lesssim \int \left[ |y|^{-3} (t+|y|)^{-q}+ |y|^{-2} (t+|y|)^{-(q+1)}\right] \langle x+y-{\boldsymbol{\ell}} (t+|y|)\rangle^{-p} dy 
\lesssim J(x-{\boldsymbol{\ell}} t).
\end{multline*}
\smallbreak

\emph{Proof of~\eqref{mai2}.}
For $j=1,\ldots,5$, we have
\[
8\pi^2 \partial_{x_j} v(t,x) 
= \int \left[ k(y) \partial_{x_j}f(t+|y|,x+y) - h(y) (\partial_{x_j}A_\ell f)(t+|y|,x+y)\right] dy,
\]
As before, 
\begin{align*}
\partial_{y_j} g(t,x,y) &= \frac{y_j}{|y|} \partial_t f(t+|y|,x+y)+ \partial_{x_j} f(t+|y|,x+y)\\
& = \frac{y_j}{|y|} A_\ell f(t+|y|,x+y)-\ell \frac{y_j}{|y|}\partial_{x_1} f(t+|y|,x+y)+ \partial_{x_j} f(t+|y|,x+y).
\end{align*}
In particular, for $j=1$, we obtain
\[
\partial_{x_1} f(t+|y|,x+y) = |y| \left(|y|-\ell {y_1} \right)^{-1} \left( \partial_{y_1} g(t,x,y) - \frac{y_1}{|y|} A_\ell f(t+|y|,x+y)\right)
\]
and next, for all $j=1,\ldots,5$,
\[
\partial_{x_j} f(t+|y|,x+y) = \partial_{y_j} g(t,x,y) +\frac {\ell y_j}{|y|-\ell y_1} \partial_{y_1} g(t,x,y)
-\frac{y_j}{|y|-\ell y_1} A_\ell f(t+|y|,x+y).
\]
Integrating by parts, we obtain
\begin{multline*}
\int k(y) \partial_{x_j}f(t+|y|,x+y) dy = -\int \partial_{y_j} k(y) f(t+|y|,x+y) dy\\
- \ell \int \partial_{y_1} \left(\frac{k(y)y_j}{|y|-\ell y_1}\right) f (t+|y|,x+y) dy
- \int \frac{k(y) y_j}{|y|-\ell y_1} A_\ell f(t+|y|,x+y) dy.
\end{multline*}
Note that $|\partial_{y_j}k(y)| \lesssim |y|^{-4}$, $\left|\partial_{y_1} \left(\frac{k(y)y_j}{|y|-\ell y_1}\right)\right|\lesssim |y|^{-4}$ and
$\left|\frac{k(y) y_j}{|y|-\ell y_1}\right|\lesssim |y|^{-3}$. 
Thus, using~\eqref{a:F},
\begin{multline*}
\left| \int k(y) \partial_{x_j}f(t+|y|,x+y) dy\right| \lesssim \int |y|^{-4} |f(t+|y|,x+y)| dy + \int |y|^{-3} |A_\ell f(t+|y|,x+y)| dy\\
\lesssim \int \left[ |y|^{-4} (t+|y|)^{-q}+ |y|^{-3} (t+|y|)^{-(q+1)}\right] \langle x+y-{\boldsymbol{\ell}} (t+|y|)\rangle^{-p} dy 
\lesssim K(x-{\boldsymbol{\ell}} t).
\end{multline*}
Proceeding similarly, we have
\begin{multline*}
\int h(y) (\partial_{x_j}A_\ell f)(t+|y|,x+y) dy
 = -\int \partial_{y_j} h(y) A_\ell f(t+|y|,x+y) dy\\
 - \ell \int \partial_{y_1} \left(\frac{h(y)y_j}{|y|-\ell y_1}\right) A_\ell f (t+|y|,x+y) dy
 - \int \frac{h(y) y_j}{|y|-\ell y_1} A_\ell^2 f(t+|y|,x+y) dy.
\end{multline*}
Note that $|\partial_{x_j}h(y)| \lesssim |y|^{-3}$, $\left|\partial_{y_1} \left(\frac{h(y)y_j}{|y|-\ell y_1}\right)\right|\lesssim |y|^{-3}$ and
$\left|\frac{h(y) y_j}{|y|-\ell y_1}\right|\lesssim |y|^{-2}$. 
Thus, using~\eqref{a:F},
\begin{multline*}
\left|\int h(y) (\partial_{x_j}A_\ell f)(t+|y|,x+y) dy\right|\\ \lesssim \int |y|^{-3} |A_\ell f(t+|y|,x+y)| dy + \int |y|^{-2} |A_\ell^2 f(t+|y|,x+y)| dy\lesssim K(x-{\boldsymbol{\ell}} t).
\end{multline*}
\smallbreak

Now, we estimate $J(t,a)$. We split $J$ as follows
\[
J= \int_{|y|<\frac14 {\langle a\rangle} } + \int_{\frac14 {\langle a\rangle}<|y|<\frac 2{1-\ell} \langle a\rangle}+\int_{ |y|>\frac 2{1-\ell} \langle a\rangle} =J_1+J_2+J_3.
\]
First, we observe that if $|y|<\frac14 {\langle a\rangle}$ then $|y-{\boldsymbol{\ell}} |y|| < 2 |y| <\frac 12 \langle a\rangle$ and thus
$\langle a + y-{\boldsymbol{\ell}} |y|\rangle \gtrsim \langle a\rangle$.
It follows that
\[
J_1\lesssim \langle a\rangle^{-p} \int_{|y|<\frac14 {\langle a\rangle} } |y|^{-3} (t+|y|)^{-q} dy
\lesssim t^{-(q-2)}\langle a\rangle^{-p} \int_{|z|<\frac1{4t} {\langle a\rangle} } |z|^{-3} (1+|z|)^{-q} dz.
\]
For all $b>0$, we have
\[
\int_{|z|<b} |z|^{-3} (1+|z|)^{-q} dz
\lesssim \int_0^b r (1+r)^{-q} dr\lesssim
\left\{\begin{aligned}
& b^2 (1+b^2)^{-1} \log (2+b) \quad \hbox{if $q=2$,}\\
& b^2 (1+b^2)^{-1} \quad \hbox{if $q>2$.}
\end{aligned}\right.
\]
Thus, we have, for any $p>2$,
\[
J_1\lesssim
\left\{\begin{aligned}
& (t+\langle a\rangle)^{-2} \langle a\rangle^{-(p-2)} \log \left(2+\frac {\langle a\rangle}{t}\right) \quad \hbox{if $q=2$,}\\
& (t+\langle a\rangle)^{-2} t^{-(q-2)}\langle a\rangle^{-(p-2)} \quad \hbox{if $q>2$.}
\end{aligned}\right.
\]
Second, we observe that
\[
J_2\lesssim (t+\langle a\rangle)^{-q} \langle a\rangle^{-3} \int_{\frac14 {\langle a\rangle}<|y|<\frac 2{1-\ell} \langle a\rangle} \langle a + y-{\boldsymbol{\ell}} |y|\rangle^{-p} dy.
\]
We change variable $z=\varphi(y)=a+y-{\boldsymbol{\ell}} |y|$. Since $|D\varphi(y)| = 1-\ell \frac {y_1}{|y|}\geq 1-\ell$
and
\[
|y|<\frac 2{1-\ell} \langle a\rangle \quad \hbox{implies} \quad |z|\leq \langle a\rangle+\frac{2(1+\ell)}{1-\ell} \langle a\rangle \leq \frac{4}{1-\ell} \langle a\rangle,
\]
we obtain, for any $q\geq 2$,
\[
J_2\lesssim (t+\langle a\rangle)^{-q} \langle a\rangle^{-3} \int_{|z|\leq \frac{4}{1-\ell} \langle a\rangle} \langle z\rangle^{-p} dz
\lesssim \left\{
\begin{aligned}
& (t+\langle a\rangle)^{-q} \langle a \rangle^{-(p-2)} \quad \hbox{if $p<5$,}\\
& (t+\langle a\rangle)^{-q} \langle a\rangle^{-3} \log(1+\langle a \rangle) \quad \hbox{if $p=5$,}\\
& (t+\langle a\rangle)^{-q} \langle a\rangle^{-3} \quad \hbox{if $p>5$.}\\
\end{aligned}\right.
\]
Third, for $|y|>\frac 2{1-\ell} \langle a\rangle$, we have
$|a +y-{\boldsymbol{\ell}} |y||\geq |y| -(\ell |y|+|a|)\geq (1-\ell) |y| - |a|\geq \frac 12 (1-\ell)|y|$, and so, for any $q\geq 2$, $p>2$,
\[
J_3\lesssim \int_{ |y|>\frac 2{1-\ell} \langle a\rangle} |y|^{-3-p}(t+|y|)^{-q} dy 
\lesssim (t+\langle a \rangle)^{-q}\langle a\rangle^{-(p-2)}.
\]
The estimates on $v(t,x)$ follow from gathering the above estimates on $J_1$, $J_2$ and $J_3$.

\smallbreak

Finally, we estimate $K(t,a)$. We split $K$ as follows
\[
K= \int_{|y|<\frac14 {\langle a\rangle} } + \int_{\frac14 {\langle a\rangle}<|y|<\frac 2{1-\ell} \langle a\rangle}+\int_{ |y|>\frac 2{1-\ell} \langle a\rangle}
=K_1+K_2+K_3.
\]
First, as before,
\[
K_1\lesssim \langle a\rangle^{-p} \int_{|y|<\frac14 {\langle a\rangle} } |y|^{-4} (t+|y|)^{-q} dy
\lesssim t^{-(q-1)}\langle a\rangle^{-p} \int_{|z|<\frac1{4t} {\langle a\rangle} } |z|^{-4} (1+|z|)^{-q} dz.
\]
For all $b>0$, $q\geq 2$, we have
$\int_{|z|<b} |z|^{-4} (1+|z|)^{-q} dz
\lesssim \int_0^b (1+r)^{-q} dr\lesssim
 b (1+b)^{-1}$.
Thus, we have, for any $q\geq 2$, $p>2$,
$
K_1\lesssim
 (t+\langle a\rangle)^{-1} t^{-(q-1)}\langle a\rangle^{-(p-1)}.
$
Second, we observe that
\[
K_2\lesssim (t+\langle a\rangle)^{-q} \langle a\rangle^{-4} \int_{\frac14 {\langle a\rangle}<|y|<\frac 2{1-\ell} \langle a\rangle} \langle a + y-{\boldsymbol{\ell}} |y|\rangle^{-p} dy,
\]
and thus, proceeding as before for $J_2$, we obtain, for any $q\geq 2$,
\[
K_2\lesssim (t+\langle a\rangle)^{-q} \langle a\rangle^{-4} \int_{|z|\leq \frac{4}{1-\ell} \langle a\rangle} \langle z\rangle^{-p} dz
\lesssim \left\{
\begin{aligned}
& (t+\langle a\rangle)^{-q} \langle a \rangle^{-(p-1)} \quad \hbox{if $p<5$,}\\
& (t+\langle a\rangle)^{-q} \langle a\rangle^{-4} \log(1+\langle a \rangle) \quad \hbox{if $p=5$,}\\
& (t+\langle a\rangle)^{-q} \langle a\rangle^{-4} \quad \hbox{if $p>5$.}\\
\end{aligned}\right.
\]
Third, for any $q\geq 2$, $p>2$,
\[
K_3\lesssim \int_{ |y|>\frac 2{1-\ell} \langle a\rangle} |y|^{-4-p}(t+|y|)^{-q} dy 
\lesssim (t+\langle a \rangle)^{-q}\langle a\rangle^{-(p-1)}.
\]
The estimates on $\nabla v(t,x)$ follow from gathering the above estimates on $K_1$, $K_2$ and $K_3$.
\end{proof}
\subsection{Spherical means and reduction to $1$D}
We recall a standard property of spherical means of general solutions of the linear wave equation
(see \emph{e.g.}~\cite{Ev}, \S2.2 and \S2.4), both in the homogeneous and non-homogeneous cases.
\begin{lemma}\label{le:spherical}
Let $u_{\rm L}(t,x)$ be solution of the {\rm $5$D} linear wave equation.
Then, the radial function
\[
U_{\rm L}(t,x) = \fint_{|y|=|x|} u_{\rm L}(t,y) d\omega(y)
\]
also satisfies the {\rm $5$D} linear wave equation.
For $f\in \mathcal S({\mathbb{R}}\times {\mathbb{R}}^5)$, let $v(t,x)$ be given by~\eqref{duhamel2} and
\[
V(t,x) = \fint_{|y|=|x|} v(t,x) d\omega(x),\quad 
F(t,x)=\fint_{|y|=|x|} f(t,x) d\omega(x).
\]
Then,
\begin{equation}\label{def:V}
V(t)
= \int_0^\infty \frac{\sin(s\sqrt{-\Delta})}{\sqrt{-\Delta}} F(s+t) ds.
\end{equation}
\end{lemma}
\begin{remark}
Note that for $U(r)= \fint_{|y|=r} u(t,x) d\omega(x)$, the following hold
\begin{equation}\label{UuL2}\begin{aligned}
&\int u^2 = \frac{8\pi^2}3 \int_0^{+\infty} U^2(r) r^4 dr,\quad
 \int |\nabla u|^2 \geq \frac{8\pi^2}3\int_0^{+\infty} (\partial_r U)^2(r) r^4 dr.
\end{aligned}\end{equation}
\end{remark}
Now, we recall a standard reduction of radial $5$D to $1$D in the non-homogeneous case.
\begin{lemma}[Reduction to 1D]\label{le:red1D}
Let $F\in \mathcal S({\mathbb{R}}\times {\mathbb{R}}^5)$ be a radial function and $V$ be the radial function defined by~\eqref{def:V}.
Let
\[
h(t,r) = r^2 \partial_r F(t,r)+3r F(t,r)
\quad\hbox{and}\quad
\phi(t,r) = r^2 \partial_r V(t,r)+3r V(t,r).
\]
Then,
\[
\phi(t,r)= \frac 12 \int_0^{+\infty}\int_{|r-\sigma|}^{r+\sigma} h(t+\sigma,a) dad\sigma.
\]
\end{lemma}
\begin{remark}
 Note that $\phi$ satisfies a non-homogeneous wave equation with zero Dirichlet conditions at $r=0$.
 We refer to computations in \S2.4 of~\cite{Ev}.
\end{remark}
\subsection{Channels of energy}
We recall a result on channels of energy for the linear radial wave equation in 5D
from~\cite{KLS} (see also~\cite{DKM1} and
\cite{KLLS} for any odd space dimension).
\begin{proposition}[\cite{KLS}, Proposition~4.1]\label{pr:ch}
There exists a constant $C>0$ such that any radial energy solution $U_{\rm L}$ of the {\rm $5$D} linear wave equation
\[
\left\{ \begin{aligned}
&\partial_t^2 U_{\rm L} - \Delta U_{\rm L} = 0, \quad (t,x)\in [0,\infty)\times {\mathbb{R}}^5,\\
& {U_{\rm L}}_{|t=0} = U_0\in \dot H^1,\quad 
{\partial_t U_{\rm L}}_{|t=0} = U_1\in L^2,
\end{aligned}\right.
\]
satisfies, for any $R>0$, either
\[
\liminf_{t\to -\infty} \int_{|x|>|t|+R} |\partial_t U_{\rm L}(t,x)|^2+|\nabla U_{\rm L}(t,x)|^2 dx 
\geq C\| \pi^\perp_R (U_0,U_1) \|^2_{(\dot H^1\times L^2) (|x|>R)}
\]
or
\[
\liminf_{t\to +\infty} \int_{|x|>|t|+R} |\partial_t U_{\rm L}(t,x)|^2+|\nabla U_{\rm L}(t,x)|^2 dx 
\geq C \| \pi^\perp_R (U_0,U_1) \|^2_{(\dot H^1\times L^2) (|x|>R)}
\]
where $\pi^{\perp}_R(U_0,U_1)$ denotes the orthogonal projection of $(U_0,U_1)^\mathsf{T}$ onto the complement of the plane
\[
 {\rm span} \left\{ (|x|^{-3},0)^\mathsf{T},(0,|x|^{-3})^\mathsf{T} \right\}
\]
in $(\dot H^1\times L^2) (|x|>R)$.
\end{proposition}
\begin{remark}\label{rk:proj}
Part of the proof of Proposition~4.1 in~\cite{KLS} relies on reduction to $1$D and on the fact that for a radial function $f$ on ${\mathbb{R}}^5$, and $R>0$,
the function $\tilde f(x) = f(x) - \frac{R^3} {|x|^3} f(R)$ is the orthogonal projection perpendicular to $|x|^{-3}$ in $\dot H^1(|x|>R)$ and so
\begin{multline*}
\|\pi^{\perp}_R(f,0)\|^2_{(\dot H^1\times L^2) (|x|>R)}
=\|\tilde f\|^2_{\dot H^1(|x|>R)}
=\int_R^{+\infty} (\tilde f'(r))^2 r^4 dr \\= \int_R^{+\infty} (f'(r))^2 r^4 dr - 3 R^3f^2(R) 
=\int_R^{+\infty} \left(r^2 f'(r)+3rf(r)\right)^2dr.
\end{multline*}
\end{remark}
\begin{remark}\label{rk:ch}
It follows from the proof of Proposition~4.1 in~\cite{KLS} that there exists $c>0$ such that if 
\[
\limsup_{t\to +\infty} \int_{|x|>|t|+R} |\partial_t U_{\rm L}(t,x)|^2+|\nabla U_{\rm L}(t,x)|^2 dx \leq c \| \pi^\perp_R (U_0,U_1) \|^2_{(\dot H^1\times L^2) (|x|>R)}, 
\]
then, for some $C>0$,
\[
\liminf_{t\to -\infty} \int_{|x|>|t|+R} |\partial_t U_{\rm L}(t,x)|^2 dx 
\geq C\| \pi^\perp_R (U_0,U_1) \|^2_{(\dot H^1\times L^2) (|x|>R)}
\]
and
\[
\liminf_{t\to -\infty} \int_{|x|>|t|+R} |\nabla U_{\rm L}(t,x)|^2 dx 
\geq C\| \pi^\perp_R (U_0,U_1) \|^2_{(\dot H^1\times L^2) (|x|>R)}.
\]
\end{remark}
\subsection{Lorentz transform}\label{clLor}
For ${\boldsymbol{\beta}}\in {\mathbb{R}}^5$ with $|{\boldsymbol{\beta}}|<1$, the Lorentz transform of parameter ${\boldsymbol{\beta}}$ of a function $u(t,x)$ is defined by
\[
u_{\boldsymbol{\beta}} (t,x) = u\left(\frac{t-{\boldsymbol{\beta}} x}{\sqrt{1-|{\boldsymbol{\beta}}|^2}},x-x_{\boldsymbol{\beta}}+\frac{x_{\boldsymbol{\beta}}-{\boldsymbol{\beta}} t}{\sqrt{1-|{\boldsymbol{\beta}}|^2}}\right),\quad
x_{\boldsymbol{\beta}} = \frac{{\boldsymbol{\beta}}}{|{\boldsymbol{\beta}}|}\left( x\cdot \frac{{\boldsymbol{\beta}}}{|{\boldsymbol{\beta}}|}\right).
\]
In particular, it ${\boldsymbol{\beta}}=\beta \mathbf{e}_1$, then the Lorentz transform of $u(t,x)$ is given simply by
\[
u_\beta (t,x) = u\left(\frac{t-\beta x}{\sqrt{1-\beta^2}}, \frac{x-\beta t}{\sqrt{1-\beta^2}},\bar x\right).\]
Let $-1<\ell<1$ and $-1<\beta<1$ and set 
$
\tilde \ell = \frac{\ell + \beta}{1+\ell \beta}.
$
Then the soliton $w_\ell(t,x)= W_{\ell}(x-\ell \mathbf{e}_1 t)$ is transformed into the soliton 
$ w_{\tilde \ell}(t,x)= W_{\tilde\ell}(x-\tilde\ell \mathbf{e}_1 t)$ by the Lorentz transform of parameter $\beta \mathbf{e}_1$.
Moreover, if $-1<\ell_1<\ell_2<1$, then $-1<\tilde \ell_1<\tilde \ell_2<1$.
Indeed, the Lorentz transform of $w_\ell(t,x)= W_{\ell}(x-\ell \mathbf{e}_1 t)$ of parameter $\beta \mathbf{e}_1$ writes
\begin{multline*}
\tilde w(t,x)=w_\ell\left(\frac{t-\beta x}{\sqrt{1-\beta^2}}, \frac{x-\beta t}{\sqrt{1-\beta^2}},\bar x\right)
=W\left(\frac 1{1-\ell^2}\left(\frac{x_1-\beta t}{\sqrt{1-\beta^2}}-\ell \left( \frac{t-\beta x_1}{\sqrt{1-\beta^2}}\right)\right),\bar x\right)\\
= W \left(\frac{1+\beta \ell}{\sqrt{(1-\beta^2)(1-\ell^2)}} \left( x_1-\frac{\beta+\ell}{1+\beta\ell}t\right),\bar x\right)
=w_{\tilde \ell}(t,x).
\end{multline*}
For the second statement, we note that for fixed $\beta\in (-1,1)$, $\frac{d \tilde \ell}{d\ell} =\frac{1-\beta^2}{(1+\ell \beta)^2}>0$,
and $\tilde \ell \to -1$ as $\ell\to -1$ and $\tilde \ell\to 1$ as $\ell \to 1$.
\subsection{On the nonlinear wave equation}\label{sec:2.6}
We recall from Lemma~2.1 and Theorem~2.7 of~\cite{KM} (see also references therein) the following fact concerning small solutions of equation~\eqref{wave}.
\begin{proposition}[Cauchy problem for small data in ${\mathbb{R}}\times{\mathbb{R}}^5$]\label{pr:CP} There exists $\delta_0>0$ such that for any $(u_0,u_1)^\mathsf{T}\in \dot H^1\times L^2$ with
$\|(u_0,u_1)\|_{\dot H^1\times L^2} \le \delta_0$, the unique global solution $\vec u=(u,\partial_t u)^\mathsf{T}\in \mathcal C({\mathbb{R}},\dot H^1 \times L^2)$ of~\eqref{wave} with initial data $(u_0,u_1)$ satisfies $\sup_{t\in {\mathbb{R}}^2} \|\vec u(t)\|_{\dot H^1\times L^2} \lesssim \delta_0$.
Moreover, if $\vec u_{\rm L}=(u_{\rm L},\partial_t u_{\rm L})^\mathsf{T}$ is the global solution of the linear wave equation $\partial_t^2 u_{\rm L} - \Delta u_{\rm L}=0$ with initial data $(u_0,u_1)\in \dot H^1\times L^2$, then
\[
\sup_{t\in {\mathbb{R}}}\|\vec u (t)-\vec u_{\rm L}(t)\|_{\dot H^1 \times L^2} \lesssim \|(u_0,u_1)\|_{\dot H^1\times L^2}^{\frac 73}.
\]
\end{proposition}
\section{Non-homogeneous linearized problem related to soliton interaction}
Following the sketch of proof given in \S\ref{simpl}, we study a non-homogeneous linearized wave equation related to the interaction of two solitons. We prove the existence of an approximate solution to this problem and then prove sharp asymptotic properties. 
\subsection{Approximate solution to a non-homogeneous linearized equation}
Let $-1< \ell <1$ and $F$, $G$ be defined by
\begin{equation}\label{def:FG} 
F = W^{\frac 43} + \kappa_\ell \Lambda W , \quad 
G = \ell(1-\ell^2)^{-\frac 12} \kappa_\ell\partial_{x_1} \Lambda W ,\quad \kappa_{\ell} = - (1-\ell^2) \frac{(W^{\frac 43},\Lambda W)}{\|\Lambda W\|_{L^2}^2}>0 .
 \end{equation}
Set
\[
w_\ell (t,x) = W ( x_\ell ),\quad F_\ell (t,x) = F ( {x_\ell} ),\quad
G_\ell (t,x) = G ( {x_\ell} ),\quad x_\ell = \left( \frac{x_1-\ell t}{\sqrt{1-\ell^2}},\overline x\right).
\]
\begin{lemma}\label{pr:AS1}
There exists a smooth function $v_\ell$ such that, for all $0<\delta<1$ and all $t\geq 1$, 
\begin{equation}\label{e:n40bis}
\|(v_\ell,\partial_t v_\ell)(t)\|_{\dot H^1\times L^2}\lesssim t^{-2},\quad
\|v_\ell(t)\|_{L^2}\lesssim t^{-\frac 32+\delta},\quad
\left\| \mathcal E_\ell(t) \right\|_{L^2}\lesssim t^{-4+\delta},
\end{equation}
where 
\[
\mathcal E_\ell=
\partial_t^2 v_\ell - \Delta v_\ell - \frac 73 w_\ell^{\frac 43} v_\ell 
- f_\ell - g_\ell,\quad f_\ell = t^{-3} F_\ell,\quad g_\ell =t^{-2} G_\ell.
\]
Moreover, for all $m\geq 0$, $|\alpha|= 1$, $|\alpha'|\geq 2$, $t\geq 1$, $x\in {\mathbb{R}}^5$,
\begin{equation}\label{e:n40}\begin{aligned}
|A_\ell^m v_\ell(t,x)| &\lesssim (t+\langle x_\ell\rangle)^{-1} t^{-(1+m)} \langle x_\ell\rangle^{-2+\delta},\\
|A_\ell^m\partial^\alpha v_\ell(t,x)| &\lesssim (t+\langle x_\ell\rangle)^{-1}t^{-(1+m)} \langle x_\ell\rangle^{-3+\delta},\\
|A_\ell^m\partial^{\alpha'} v_\ell(t,x)| &\lesssim (t+\langle x_\ell\rangle)^{-1}t^{-(1+m)} \langle x_\ell\rangle^{-4+\delta},
\end{aligned}
\end{equation}
and
\begin{equation}\label{pres1}\begin{aligned}
& |A_\ell^m \mathcal E_{\ell}(t,x)|\lesssim t^{-(4+m)+\delta} \langle x_\ell\rangle^{-3},\quad
|A_\ell^m\partial^{\alpha} \mathcal E_{\ell}(t,x)|\lesssim t^{-(4+m)+\delta} \langle x_\ell \rangle^{-4}
,\\
& |A_\ell^m\partial^{\alpha'} \mathcal E_{\ell}(t,x)|\lesssim t^{-(4+m)+\delta} \langle x_\ell \rangle^{-5}.
\end{aligned}\end{equation}
\end{lemma} 
\begin{remark}\label{rk:contrast}
In contrast with the strategy used in~\cite{MMimrn} for the gKdV equation, we do not construct an approximate solution $v_\ell$ of the equation
$\mathcal E_\ell=0$ simply by separation of variables. Indeed, the decay properties in space of such approximate solution would not be sufficient for our needs.
Rather, we solve alternatively the linear wave equation $\partial_t^2 v - \Delta v= K_1$, and the elliptic equation $L v = K_2$, for various functions $K_1$ and $K_2$.
For the linear wave equation we use the estimates of Lemma~\ref{le:NH}, see also Remark~\ref{rkp10}.
For the elliptic equation, we use Lemma~\ref{le:Q}.
Because of the existence of a non-trivial kernel for the operator $L$, specific relations on $F$ and $G$ are needed. To state them precisely, we introduce
\[
D_0=x_1 \Lambda W,\quad D_j=x_1 \partial_{x_j} W,\quad 
LD_0=-2\partial_{x_1} \Lambda W,\quad LD_j = -2 \partial_{x_1}\partial_{x_j} W.
\]
We note that the following relations hold, for $j=1,\ldots,5$, 
\begin{equation}\label{o:FG}\begin{aligned}
& ( G,\Lambda W)= (G,\partial_{x_j} W)=\langle G, D_{j}\rangle=0,
\quad (F,\partial_{x_j}W) = 0\\&
( F,\partial_{x_1} W) = 2 \ell(1-\ell^2)^{-\frac 12} ( G, D_1).
\end{aligned}\end{equation}
 Indeed, first, for $j=1,\ldots,5$,
\[
(\partial_{x_1} \Lambda W,\Lambda W)=(\partial_{x_1} W,\partial_{x_1} \Lambda W)=
(\partial_{x_1} \Lambda W,\partial_{x_j} W)=(\partial_{x_1}\Lambda W,D_j)=0.
\]
Moreover, for $j =1,\ldots,5$, $(F,\partial_{x_j}W) = 0$.
Now,
we compute $(F,\Lambda W)$ and $( G, D_0)$
\[
(F,\Lambda W) = (W^{\frac 43},\Lambda W) + \kappa_\ell \|\Lambda W\|_{L^2}^2 ,
\]
\[
2(G,D_0)= 2\kappa_\ell \ell (1-\ell^2)^{-\frac 12}( \partial_{x_1} \Lambda W,x_1\Lambda W)= - \kappa_\ell \ell (1-\ell^2)^{-\frac 12} \|\Lambda W\|_{L^2}^2.
\]
Thus, the condition 
$(F,\Lambda W) = 2 \ell (1-\ell^2)^{-\frac 12} ( G, D_0)$ is equivalent to
\[
(W^{\frac 43},\Lambda W)= - \kappa_\ell \|\Lambda W\|_{L^2}^2 (1+\ell^2 (1-\ell^2)^{-1}) = - \frac{\kappa_\ell}{1-\ell^2} \|\Lambda W\|_{L^2}^2 
\]
which is indeed the definition of $\kappa_\ell$.
We define the operator $\mathcal L_\ell$ by
\[
{\mathcal L}_\ell G_\ell = -(1-\ell^2)\partial_{x_1}^2 G_\ell -\bar \Delta G_\ell-\frac 73 w_\ell^{\frac 43} G_\ell
\quad\hbox{so that}\quad
{\mathcal L}_\ell G_\ell = (LG)\left( x_\ell\right).
\]
\end{remark}
\begin{proof}[Proof of Lemma~\ref{pr:AS1}]
\textbf{Approximate solution at order $t^{-2}$.} 
First, following \S\ref{secc:2.2}, we set
\[ 
v_1(t)=\int_0^\infty \frac{\sin(s\sqrt{-\Delta})}{\sqrt{-\Delta}} g_\ell(s+t) ds.
\]
Note that for $m\geq 0$, $A_\ell^m v_1$ satisfies
\[
(A_\ell^m v_1)(t)=\int_0^\infty \frac{\sin(s\sqrt{-\Delta})}{\sqrt{-\Delta}} (A_\ell^m g_\ell)(s+t) ds.
\]
Since $|A_\ell^m g_\ell| \lesssim t^{-(2+m)} \langle x_\ell\rangle^{-4}$, it follows from Lemma~\ref{le:NH} that, for $m\geq 1$,
\begin{equation}\label{ff2}
|v_1(t,x)|\lesssim (t+ \langle x_\ell\rangle)^{-2} \langle x_\ell\rangle^{-2+\delta},\quad
|A_\ell^m v_1(t,x)|\lesssim (t+ \langle x_\ell\rangle)^{-2} t^{-m} \langle x_\ell\rangle^{-2} .
\end{equation}
Here, and in the rest of the proof, $0<\delta<1$ is arbitrary.
For $m\geq 0$, $|\alpha|\geq 1$, we have $
|\partial^{\alpha} A_\ell^m g_\ell| \lesssim t^{-(2+m)} \langle x_\ell\rangle^{-5}$,
and thus by Lemma~\ref{le:NH}, for $m\geq 0$, $|\alpha|=1$, $|\alpha'|\geq 2$,
\begin{equation}\label{ff2b} 
|\partial^{\alpha} A_\ell^m v_1|\lesssim (t+ \langle x_\ell\rangle)^{-2}t^{-m} \langle x_\ell\rangle^{-3+\delta},\quad
|\partial^{\alpha'} A_\ell^m v_1|\lesssim (t+ \langle x_\ell\rangle)^{-1}t^{-(m+1)} \langle x_\ell\rangle^{-4+\delta}.
\end{equation}
\smallbreak

Now, let
\[
R_1= \frac 73 v_1 w_\ell^{\frac 43} -a_1w_\ell^2-b_1\cdot\nabla (w_\ell^2),
\]
where, for $j=1,\ldots,5$,
\[
a_1 =\frac {(\frac 73 v_1 w_\ell^{\frac 43} , \Lambda_\ell w_\ell)}{( w_\ell^2, \Lambda_\ell w_\ell)},
\quad b_{1,j}= \frac {( \frac 73v_1 w_\ell^{\frac 43} , \partial_{x_j} w_\ell)}{( \partial_{x_j}( w_\ell^2), \partial_{x_j} w_\ell)}
\]
so that
$
( R_1,\Lambda_\ell w_\ell)=0$, $( R_1,\nabla w_\ell)=0$.
Note that by~\eqref{ff2}, $|a_1(t)|+|b_1(t)|\lesssim t^{-2}$. Next,
since 
\begin{align*}
\dot a_1=
\frac {( \frac 73 A_\ell v_1, w_\ell^{\frac 43} \Lambda_\ell w_\ell)}{( w_\ell^2, \Lambda_\ell w_\ell)} ,
\end{align*}
by~\eqref{ff2}, we have
$|\dot a_1|\lesssim t^{-3}$ and similarly, $|\dot b_1|\lesssim t^{-3}$. More generally,
$\left| \frac{d^m a_1}{dt^m}\right|+ \left| \frac{d^m b_1}{dt^m}\right|\lesssim t^{-2-m}$.
Thus, using~\eqref{ff2} and~\eqref{ff2b}, we obtain, for all $\alpha\in {\mathbb{N}}^5$,
$|\partial^\alpha R_1|\lesssim t^{-2}\langle x_\ell\rangle^{-6+\delta}$. Moreover, by direct computations,
\begin{align*}
A_\ell^m R_1 & = \frac 73 (A_\ell^m v_1) w_\ell^{\frac 43} -\frac {d^ma_1}{dt^m} w_\ell^2-\frac {d^mb_1}{dt^m} \cdot \nabla (w_\ell^2).
\end{align*}
Using~\eqref{ff2} and~\eqref{ff2b} again, we obtain, for all $m\geq 0$, $\alpha\in {\mathbb{N}}^5$,
$
|\partial^\alpha A_\ell^m R_1|\lesssim t^{-2-m}\langle x_\ell\rangle^{-6+\delta}$.
From (iii) of Lemma~\ref{le:Q}, there exists $v_2$ solution of
${\mathcal L}_\ell v_2 = R_1$,
 satisfying, for $|\alpha|= 1$, $|\alpha'|\geq 2$,
\[
( v_2,\Lambda_\ell w_\ell) =0,\quad ( v_2,\nabla w_\ell) =0,
\]
and
\[
|v_2|\lesssim t^{-2} \langle x_\ell\rangle^{-3},\quad 
|\partial^\alpha v_2|\lesssim t^{-2} \langle x_\ell\rangle^{-4},
\quad 
|\partial^{\alpha'} v_2|\lesssim t^{-2} \langle x_\ell\rangle^{-5}.
\]
Next, we see that, for all $m\geq 1$,
\[
0=\frac {d^m}{dt^m}( v_2,\Lambda_\ell w_\ell)=( A_\ell^m v_2,\Lambda_\ell w_\ell),\]
and similarly, $ ( A_\ell^m v_2,\nabla W(x_\ell))=0$.
We also observe that for all $m\geq 0$,
\[
 \mathcal L_\ell(A_\ell^m v_2)=A_\ell^m \left( {\mathcal L}_\ell v_2 \right)
 = A_\ell^m R_1.
\]
Thus, by (iii) of Lemma~\ref{le:Q}, for all $m\geq 0$, $|\alpha|= 1$, $|\alpha'|\geq 2$,
\begin{equation}\label{ff3}
|A_\ell^m v_2|\lesssim t^{-2-m}\langle x_\ell \rangle^{-3},\quad
|\partial^\alpha A_\ell^m v_2|\lesssim t^{-2-m}\langle x_\ell\rangle^{-4}
\quad
|\partial^{\alpha'} A_\ell^m v_2|\lesssim t^{-2-m}\langle x_\ell\rangle^{-5}.
\end{equation}
\smallbreak
 
We see that $v_1+v_2$ satisfies
\begin{equation}\label{gg1}
\left( \partial_t^2 - \Delta - \frac 73 w_\ell^{\frac 43} \right) (v_1+v_2)
= g_\ell + B_\ell v_2 - a_1 w_\ell^2-b_1\cdot \nabla (w_\ell^2) ,
\end{equation}
which also rewrites (since $\partial_t^2-\Delta=B_\ell-\Delta_\ell$),
\begin{equation}\label{gg2}
{\mathcal L}_\ell (v_1+v_2)
= g_\ell - B_\ell v_1 - a_1 w_\ell^2-b_1\cdot \nabla w_\ell^2.
\end{equation}
\smallbreak
 
\textbf{Estimates at order $t^{-3}$.}
First, from~\eqref{gg2}, using the orthogonality relations $( G,\Lambda W)=0$ and $( G,\nabla W)=0$, we claim that, for any $m\geq 0$,
\begin{equation}\label{gg3}
\left| \frac{d^m a_1}{dt^m}\right|+ \left| \frac{d^m b_1}{dt^m}\right|\lesssim t^{-3-m},
\end{equation}
improving by a factor $t^{-1}$ the previous estimates on $a_1$ and $b_1$.
\smallbreak
 
Proof of~\eqref{gg3}. For $m\geq 0$, by direct computations from~\eqref{gg2}, we have
\begin{equation}\label{gg4} {\mathcal L}_\ell \left( A_\ell^m (v_1+v_2)\right)
=A_\ell^m g_\ell - B_\ell( A_\ell^m v_1) - \frac {d^m a_1}{dt^m} w_\ell^2
-\frac {d^m b_1}{dt^m} \cdot \nabla (w_\ell^2)
\end{equation}
We project this estimate on $\Lambda w_\ell$ and $\partial_{x_j} w_\ell$.
First, note that $A_\ell^m g_\ell = \frac{d^m}{dt}(t^{-2}) G_\ell$, and since
$( G,\Lambda W)=0$ and $( G,\partial_{x_j} W)=0$,
one has
\[
( A_\ell^m g_\ell ,\Lambda_\ell w_\ell)=0,\quad ( A_\ell^m g_\ell ,\nabla w_\ell)=0.
\]
 Next, since $L\Lambda W=0$ and $L\nabla W=0$, one has
$
{\mathcal L}_\ell (\Lambda_\ell w_\ell) = 0$ and ${\mathcal L}_\ell (\nabla w_\ell) = 0$.
Thus
\[
( {\mathcal L}_\ell \left( A_\ell^m (v_1+v_2)\right),\Lambda_\ell w_\ell)=0,\quad
( {\mathcal L}_\ell \left( A_\ell^m (v_1+v_2)\right),\nabla w_\ell)=0.
\]
Now, we estimate $( B_\ell A_\ell^m v_1,\Lambda_\ell w_\ell)$.
Since $B_\ell = A_\ell^2-2 \ell \partial_{x_1} A_\ell$,
\[
( B_\ell A_\ell^m v_1,\Lambda_\ell w_\ell)
= ( A_\ell^{m+2} v_1,\Lambda_\ell w_\ell)+2\ell \langle A_\ell^{m+1}v_1,\partial_{x_1}\Lambda w_\ell).
\]
By~\eqref{ff2},
\[
|( A_\ell^{m+2} v_1,\Lambda_\ell w_\ell)|\lesssim t^{-m-4+\delta},\quad
|( A_\ell^{m+1} v_1,\partial_{x_1}\Lambda w_\ell))|\lesssim t^{-m-3}.
\]
Thus, $|( B_\ell A_\ell^m v_1,\Lambda w_\ell)|\lesssim t^{-m-3}$, and similarly,
$|( B_\ell A_\ell^m v_1,\nabla w_\ell)|\lesssim t^{-m-3}$. Projecting~\eqref{gg4} on $\Lambda w_\ell$ and
$\nabla w_\ell$ and gathering the above estimates, we find~\eqref{gg3}.
\smallbreak
 
Again by $\partial_t^2-\Delta=B_\ell-\Delta_\ell$ and $B_\ell = A_\ell^2-2\ell \partial_{x_1} A_\ell$, we rewrite~\eqref{gg1} as follows
\begin{equation}\label{3.25b}
B_\ell (v_1+v_2) + {\mathcal L}_\ell (v_1+v_2) = g_\ell+\mathcal E_1 + R_2
\end{equation}
 where
\[
\mathcal E_1 = A_\ell^2 v_2,\quad
R_2 = -2 \ell \partial_{x_1} A_\ell v_2 - a_1 w_\ell^2 - b_1\cdot \nabla( w_\ell^2 ).
\]
From~\eqref{ff3}, we have, for $m\geq 0$, $|\alpha|= 1$, $|\alpha'|\geq 2$,
\begin{equation}\label{hh1}
|A_\ell^m \mathcal E_1|\lesssim t^{-4-m} \langle x_\ell\rangle^{-3},\quad
|\partial^\alpha A_\ell^m \mathcal E_1|\lesssim t^{-4-m} \langle x_\ell\rangle^{-4}
\quad
|\partial^{\alpha '}A_\ell^m \mathcal E_1|\lesssim t^{-4-m} \langle x_\ell\rangle^{-5}.
\end{equation}
In particular, since $\partial_t =A_\ell -\ell \partial_{x_1}$,
\begin{align*}& |A_\ell^m \partial_t \mathcal E_1|\lesssim t^{-5-m} \langle x_\ell\rangle^{-3}+t^{-4-m} \langle x_\ell\rangle^{-4},\\
& |A_\ell^m \partial_t\partial^{\alpha} \mathcal E_1|\lesssim t^{-5-m} \langle x_\ell\rangle^{-4}+t^{-4-m} \langle x_\ell\rangle^{-5}, 
 \quad |A_\ell^m \partial_t\partial^{\alpha'} \mathcal E_1|\lesssim t^{-4-m} \langle x_\ell\rangle^{-5}.
\end{align*}
Similarly, from~\eqref{ff3} and~\eqref{gg3}, for $m\geq 0$, $|\alpha|\geq 1$, we have
\begin{equation}\label{ii0}
|A_\ell^m R_2|\lesssim t^{-3-m} \langle x_\ell\rangle^{-4},\quad
|\partial^\alpha A_\ell^mR_2|\lesssim t^{-3-m} \langle x_\ell\rangle^{-5}.
\end{equation}
\smallbreak
 
Now, we claim, for $j=1,\ldots,5$, $m\geq 0$,
\begin{equation}\label{ii1}\begin{aligned}
&\frac{d^m}{dt^m}( R_2,\Lambda_\ell w_\ell) = -\frac{d^{m+1}}{dt^{m+1}} (t^{-2}) \langle G, D_0\rangle + O(t^{-m-4+\delta}),\\ 
&\frac{d^m}{dt^m}( R_2,\partial_{x_j} w_\ell) = O(t^{-m-4+\delta}).
\end{aligned}\end{equation}
Proof of~\eqref{ii1}.
From~\eqref{3.25b}, we have
\begin{equation}\label{p1p}
A_\ell^2(v_1+v_2)-2 \ell \partial_{x_1}(A_\ell(v_1+v_2))+ {\mathcal L}_\ell (v_1+v_2) = g_\ell+\mathcal E_1 + R_2,
\end{equation}
\begin{equation}\label{p2p}
A_\ell^3(v_1+v_2)-2 \ell \partial_{x_1}(A_\ell^2(v_1+v_2))+ {\mathcal L}_\ell (A_\ell(v_1+v_2)) = A_\ell g_\ell
+A_\ell \mathcal E_1 + A_\ell R_2.
\end{equation}
First, we project~\eqref{p1p} on $\Lambda_\ell w_\ell$. We have $( g_\ell,\Lambda_\ell w_\ell)=0$, and by~\eqref{hh1},
$|( \mathcal E_1,\Lambda_\ell w_\ell)|\lesssim t^{-4}$. By~\eqref{ff2} and~\eqref{ff3} (using also that $|\Lambda W(x)|\lesssim \langle x\rangle^{-3}$),
\[
|(A_\ell^2 (v_1+v_2) ,\Lambda_\ell w_\ell)|\lesssim t^{-4+\delta}.
\]
Therefore, we have obtained
\begin{equation}\label{jj1}
\left| ( R_2 ,\Lambda_\ell w_\ell) 
+2 \ell ( \partial_{x_1}(A_\ell(v_1+v_2)),\Lambda_\ell w_\ell) \right|\lesssim t^{-4+\delta}.
\end{equation}
In a similar way (using $|\partial_{x_j} W(x)|\lesssim \langle x\rangle^{-4}$), we find for $j=1,\dots,5$,
\[
\left| ( R_2 ,\partial_{x_j} w_\ell) 
+2 \ell ( \partial_{x_1}(A_\ell(v_1+v_2)),\partial_{x_j} w_\ell) \right|\lesssim t^{-4}.
\]

Now, we compute $( \partial_{x_1}(A_\ell(v_1+v_2)),\Lambda_\ell w_\ell) $
and $( \partial_{x_1}(A_\ell(v_1+v_2)),\partial_{x_j} w_\ell) $
from~\eqref{p2p}. 
Recall that we set $D_0=x_1\Lambda W$ and $LD_0 = -2 \partial_{x_1} \Lambda W$.
We cannot project equation~\eqref{p2p} directly on $D_0(x_\ell)$ because we only know $|A_\ell^3 v_2|\lesssim t^{-5} \langle x_\ell \rangle^{-3}$, $|D_0|\lesssim \langle x \rangle^{-2}$ and $\langle x\rangle^{-5}\not\in L^1({\mathbb{R}}^5)$. Thus, we consider $\tilde\chi(x)=\tilde\chi(|x|)$ a smooth cut-off function such that
\[
\tilde\chi\equiv 1 \hbox{ on $[-1,1]$},\quad
\tilde\chi\equiv 0 \hbox{ on $[-2,2]^c$},\quad
0\leq \tilde\chi\leq 1 \hbox{ on ${\mathbb{R}}$},
\]
and we set 
\[
D_\ell(t,x)=
 D_0\left( {x_\ell} \right) ,\quad 
\tilde D_\ell(t,x)=D_\ell(t,x) \tilde\chi\left( \frac {x_\ell}{t^{10}}\right).
\]
We project~\eqref{p2p} on $\tilde D_\ell$. By~\eqref{ff2}, $|( A_\ell^3 v_1,\tilde D_\ell)|\lesssim t^{-4+\delta}$,
and by~\eqref{ff3}, $|( A_\ell^3 v_2,\tilde D_\ell)|\lesssim t^{-5}\log t\lesssim t^{-5+\delta}$.
Also, by~\eqref{ff2}-\eqref{ff3}, $|(\partial_{x_1}A_\ell^2(v_1+v_2),\tilde D_\ell)|\lesssim t^{-4+\delta}$.
Next,
\[
\langle {\mathcal L}_\ell (A_\ell(v_1+v_2)), \tilde D_\ell\rangle
= \langle A_\ell (v_1+v_2),{\mathcal L}_\ell D_\ell\rangle
-\langle A_\ell (v_1+v_2) , {\mathcal L}_\ell \left[\left(1-\tilde\chi\left( \frac {x_\ell}{t^{10}}\right) \right)D_\ell\right]\rangle.
\]
Note that $ ( A_\ell (v_1+v_2),{\mathcal L}_\ell D_\ell) = ( A_\ell (v_1+v_2),(LD_0)(x_\ell)) =-2 (1-\ell^2)^{\frac 12}( A_\ell (v_1+v_2),\partial_{x_1} \Lambda_\ell w_\ell) $.
Next,
\begin{align*}
{\mathcal L}_\ell \left[\left(1-\tilde\chi\left( \frac {x_\ell}{t^{10}}\right) \right)D_\ell\right] 
&= \left(1-\tilde\chi\left( \frac {x_\ell}{t^{10}}\right) \right) {\mathcal L}_\ell D_\ell-t^{-20} \Delta_\ell \tilde\chi\left( \frac {x_\ell}{t^{10}}\right) D_\ell
\\
&\quad -2 t^{-10} (1-\ell^2) \partial_{x_1} \tilde\chi\left( \frac {x_\ell}{t^{10}}\right) (\partial_{x_1} D_\ell)
-2 t^{-10} \bar\nabla \tilde\chi\left( \frac {x_\ell}{t^{10}}\right) \cdot \bar \nabla D_\ell\quad 
\end{align*}
so that
\[
\left|{\mathcal L}_\ell \left[\left(1-\tilde\chi\left( \frac {x_\ell}{t^{10}}\right) \right)D_\ell\right]\right|\lesssim |x_\ell|^{-4} \mathbf{1}_{|x_\ell|>t^{10}} 
+ t^{-10} |x_\ell|^{-2} \mathbf{1}_{t^{10}<|x_\ell|<2t^{10}}.
\]
Thus,
\[
\left|( A_\ell (v_1+v_2) , {\mathcal L}_\ell \left[\left(1-\tilde\chi\left( \frac {x_\ell}{t^{10}}\right) \right)D_\ell\right])\right|\lesssim t^{-10}.
\]
Next, since $(A_\ell g_\ell)(t,x) = -2 t^{-3} G_\ell$,
we have
\begin{align*}
( A_\ell g_\ell,\tilde D_\ell) = 
-2 t^{-3} ( G_\ell,D_\ell)+O(t^{-4})
 = - 2(1-\ell^2)^{\frac 12} t^{-3} \langle G, D_0\rangle + O(t^{-4}).
 \end{align*}
Moreover, by~\eqref{hh1}-\eqref{ii0},
\[
|\langle A_\ell \mathcal E_1,\tilde D_\ell\rangle|\lesssim t^{-5} \int_{|x_\ell|<2 t^{10}} \langle x_\ell\rangle^{-3} \langle x_\ell\rangle^{-2} dx
\lesssim t^{-5+\delta},
\]
and
\[
|\langle A_\ell R_2,\tilde D_\ell\rangle|\lesssim t^{-4} \int_{|x_\ell|<2 t^{10}} \langle x_\ell\rangle^{-4} \langle x_\ell\rangle^{-2} dx
\lesssim t^{-4}.
\]
Thus, the projection of~\eqref{p2p} on $\tilde D_\ell$ gives
\[
-2 (\partial_{x_1}(A_\ell(v_1+v_2)),\Lambda_\ell w_\ell)
= 2 ( A_\ell (v_1+v_2),\partial_{x_1} (\Lambda_\ell w_\ell))
=2 t^{-3} ( G, D_0) + O(t^{-4+\delta}).
\]
Inserted in~\eqref{jj1}, it gives
\[
( R_2,\Lambda_\ell w_\ell) =2 \ell t^{-3} ( G, D_0) + O(t^{-4+\delta}).
\]
To obtain the estimate on $( R_2,\partial_{x_j} w_\ell)$, for $j=1,\ldots 5$, we compute $( A_\ell(v_1+v_2),\partial_{x_1}\partial_{x_j} w_\ell)$. We use $D_j(x) = x_1 \partial_{x_j} W(x)$, so that 
$LD_j = -2 \partial_{x_1} \partial_{x_j} W$.
We proceed as before, projecting~\eqref{p2p} on $D_j(x_\ell)$. The computations are similar and easier because of the better decay properties of $D_j$ (the cut-off function $\tilde \chi$ is no longer needed).
The proof of~\eqref{ii1} for $m\geq 1$ is similar and it is omitted.
\smallbreak
 
\textbf{Approximate solution at order $t^{-3}$.}
Set
\[ 
v_3(t)=\int_0^\infty \frac{\sin(s\sqrt{-\Delta})}{\sqrt{-\Delta}} (f_\ell-R_2)(s+t) ds.
\]
We see that
 $f_\ell$ satisfies~\eqref{a:F} with $q=3$ and $p=3$. 
From~\eqref{ii0}, the function $R_2$ satisfies~\eqref{a:F} with $q=3$ and $p=4$.
Thus, from Lemma~\ref{le:NH}, for $m\geq 0$,
\begin{equation}\label{77} 
|A_\ell^m v_3 |\lesssim (t+ \langle x_\ell\rangle)^{-2} t^{-1-m} \langle x_\ell\rangle^{-1} .
 \end{equation}
 Moreover, from~\eqref{ii0} and Lemma~\ref{le:NH}, for $|\alpha|= 1$, $|\alpha'|\geq 2$, $m\geq 0$,
\begin{equation}\label{777} 
|A_\ell^m \partial^\alpha v_3|\lesssim (t+ \langle x_\ell\rangle)^{-2} t^{-1-m} \langle x_\ell\rangle^{-2},\quad
|A_\ell^m \partial^{\alpha'} v_3 |\lesssim (t+ \langle x_\ell\rangle)^{-2} t^{-1-m} \langle x_\ell\rangle^{-3+\delta}.
 \end{equation}
By construction, $v_3$ verifies
\begin{equation}\label{diese}
B_\ell v_3+{\mathcal L}_\ell v_3 = f_\ell-R_2-\frac 73 w_\ell^{\frac 43} v_3
\end{equation}
and $v_1+v_2+v_3$ satisfies
\[
\left(\partial_t^2 -\Delta-\frac 73 w_\ell^{\frac 43} \right) (v_1+v_2+v_3)
=g_\ell+f_\ell+\mathcal E_1-\frac 73 w_\ell^{\frac 43} v_3.
\]
\smallbreak

As before, set, for $j=1,\ldots,5$,
\[
a_3 =\frac {( \frac 73 v_3 w_\ell^{\frac 43} , \Lambda_\ell w_\ell)}{( w_\ell^2, \Lambda w_\ell)},
\quad b_{3,j} =\frac {(\frac 73 v_3 w_\ell^{\frac 43}, \partial_{x_j} w_\ell)}{( \partial_{x_j}( w_\ell^2) , \partial_{x_j} w_\ell)}
\]
Let
\[
R_3= \frac 73 v_3 w_\ell^{\frac 43}-a_3w_\ell^2-b_3\cdot\nabla (w_\ell^2),
\]
so that $R_3$ satisfies $(R_3,\Lambda_\ell w_\ell)=0$ and $( R_3,\nabla w_\ell)=0$.
By the decay properties of $v_3$, 
$|A_\ell^mR_3|\lesssim t^{-3-m+\delta}\langle x_\ell\rangle^{-5-\delta}$ and 
$|A^m \partial^\alpha R_3|\lesssim t^{-3-m}\langle x_\ell\rangle^{-6}$ for all $\alpha\in {\mathbb{N}}^5$ with $|\alpha|\geq 1$.
Thus, from (iii) of Lemma~\ref{le:Q}, there exists $v_4$ solution of
$
{\mathcal L}_\ell v_4 = R_3$, satisfying, for $|\alpha|=1$, $|\alpha'|\geq 2$ and $m\geq 0$,
\[
( v_4,\Lambda_\ell w_\ell) =0,\quad ( v_4,\nabla w_\ell) =0,
\]
\[
|A_\ell^m v_4|\lesssim t^{-3-m+\delta}\langle x_\ell \rangle^{-3},\quad
|\partial^\alpha A_\ell^m v_4|\lesssim t^{-3-m+\delta}\langle x_\ell\rangle^{-4},
\quad
|\partial^{\alpha'} A_\ell^m v_4|\lesssim t^{-3-m+\delta}\langle x_\ell\rangle^{-5}.
\]
In particular, $\mathcal E_2 = B_\ell v_4 = A_\ell^2 v_4 - 2 \ell \partial_{x_1} A_\ell v_4$ satisfies, for all $m\geq 0$,
\[
 |A_\ell^m\mathcal E_2|\lesssim t^{-4-m+\delta} \langle x_\ell\rangle^{-3},\quad
 |A_\ell^m \partial^\alpha \mathcal E_2|\lesssim t^{-4-m+\delta} \langle x_\ell \rangle^{-4},\quad
 |A_\ell^m \partial^{\alpha'} \mathcal E_2|\lesssim t^{-4-m+\delta} \langle x_\ell \rangle^{-5}.
\]
Let $v_\ell = v_1+v_2+v_3+v_4$ 
satisfy
\[
\left(\partial_t^2 -\Delta-\frac 73 w_\ell^{\frac 43}\right) v_\ell
=g_\ell+f_\ell+\mathcal E_\ell,
\]
where
$
\mathcal E_\ell= \mathcal E_1+\mathcal E_2+\mathcal E_3
$
and $\mathcal E_3 = -a_3 w_\ell^2- b_3\cdot \nabla (w_\ell^2)$.
\smallbreak
 
\textbf{Estimates at order $t^{-4}$.} We claim that, for $m\geq 0$,
\begin{equation}\label{mm1}
\left|\frac {d^m a_3}{dt^m} \right| +\left|\frac {d^mb_3}{dt^m} \right|\lesssim t^{-4-m+\delta}.
\end{equation}
Proof of~\eqref{mm1}. By the estimate on $v_3$, we only have 
$|a_3(t)|+|b_3(t)|\lesssim t^{-3+\delta}$.
We project~\eqref{diese} on $\Lambda_\ell w_\ell$.
As before, $( {\mathcal L}_\ell v_3,\Lambda_\ell w_\ell) = 0$. Next, recall that 
$B_\ell v_3 = A_\ell^2 v_3-2 \ell \partial_{x_1} A_\ell v_3$ and, using~\eqref{77}-\eqref{777}, we have
\[
|( A_\ell^2 v_3,\Lambda w_\ell ) |\lesssim t^{-4+\delta},\quad
|( \partial_{x_1} A_\ell v_3,\Lambda w_\ell ) |=|( A_\ell v_3,\partial_{x_1} \Lambda_\ell w_\ell ) |\lesssim t^{-4+\delta}.
\]
Now, using~\eqref{ii1} and then~\eqref{o:FG}, we compute
\begin{align*}
( f_\ell-R_2,\Lambda_\ell w_\ell) & = t^{-3} (1-\ell^2)^{\frac 12} (F,\Lambda W) - 2\ell t^{-3} (G,D_0) +O(t^{-4+\delta})
=O(t^{-4+\delta}).
\end{align*}
This is enough to obtain the estimate~\eqref{mm1} on $a_3$. 
The estimate on $b_3$ is proved similarly using~\eqref{ii1}.
For $m\geq 1$, apply $A_\ell^m$ to~\eqref{diese}, so that
\[
B_\ell A_\ell^m v_3 + {\mathcal L}_\ell(A_\ell^m v_3) = A_\ell^m f_\ell - A_\ell^m R_2 - \frac 73 A_\ell^m (w_\ell^{\frac 43} v_3),
\]
 then project on $\Lambda_\ell w_\ell$ (or $\nabla w_\ell$) and use~\eqref{ii1} to find~\eqref{mm1}.

As a consequence of~\eqref{mm1}, we obtain, for $|\alpha|=1$, $|\alpha'|\geq 2$ and $m\geq 0$,
\[
|A_\ell^m \mathcal E_3|\lesssim t^{-4-m+\delta} \langle x_\ell\rangle^{-6},\quad
|A_\ell^m \partial^\alpha \mathcal E_3|+|A_\ell^m \partial^{\alpha'} \mathcal E_3|\lesssim t^{-4-m+\delta} \langle x_\ell\rangle^{-7}.
\]
\smallbreak

\textbf{Conclusion.} Gathering the estimates on $v_1$, $v_2$, $v_3$ and $v_4$, we obtain
\begin{align*}
|A_\ell^m v_\ell|& \lesssim 
(t+ \langle x_\ell\rangle)^{-2} t^{-m} \langle x_\ell\rangle^{-2+\delta} +
t^{-(2+m)}\langle x_\ell \rangle^{-3}+
(t+ \langle x_\ell\rangle)^{-2}t^{-(1+m)} \langle x_\ell\rangle^{-1} ,
\end{align*}
so that
\begin{align*}
|A_\ell^m v_\ell | \lesssim t^{-(2+m)} \langle x_\ell\rangle^{-2+\delta},\quad 
|A_\ell^m v_\ell | \lesssim t^{-(1+m)} \langle x_\ell\rangle^{-3+\delta}.
\end{align*}
Moreover, for $|\alpha|=1$, 
\[
|A_\ell^m \partial^\alpha v_\ell| \lesssim t^{-(2+m)} \langle x_\ell \rangle^{-3+\delta},\quad
|A_\ell^m\partial^\alpha v_\ell| \lesssim t^{-(1+m)} \langle x_\ell\rangle^{-4+\delta},
\]
and for 
$|\alpha'|\geq 2$,
\[
|A_\ell^m \partial^{\alpha'} v_\ell| \lesssim t^{-(2+m)} \langle x_\ell \rangle^{-4+\delta},\quad
|A_\ell^m \partial^{\alpha'} v_\ell| \lesssim t^{-(1+m)} \langle x_\ell \rangle^{-5+\delta}.
\]
Note that time estimates on $v_\ell$ are easily obtained from~\eqref{e:n40} using $\partial_t v = A_\ell v_\ell - \ell \partial_{x_1} v_\ell$. For example, we have
\begin{equation}\label{rk33}
|A_\ell^m \partial_t v_\ell|\lesssim |A_\ell^{m+1}v_\ell|+|A_\ell^m\partial_{x_1}v_\ell|\lesssim t^{-(2+m)}\langle x_\ell\rangle^{-3+\delta}.
\end{equation}
Gathering~the estimates on $\mathcal E_1$, $\mathcal E_2$ and $\mathcal E_3$, we find, for $|\alpha|=1$, $|\alpha'|\geq 2$ and $m\geq 0$,
\[
|A_\ell^m\mathcal E_{\ell}|\lesssim t^{-4-m+\delta} \langle x_\ell\rangle^{-3},\quad
|A_\ell^m\partial^{\alpha} \mathcal E_{\ell}|\lesssim t^{-4-m+\delta} \langle x_\ell \rangle^{-4},\quad
|A_\ell^m\partial^{\alpha'} \mathcal E_{\ell}|\lesssim t^{-4-m+\delta} \langle x_\ell \rangle^{-5}.
\]
In particular, this proves $\|\mathcal E_{\ell}\|_{L^2} \lesssim t^{-4+\delta}$, which is the estimate of $\mathcal E_\ell$ in~\eqref{e:n40bis}.
\end{proof}
\subsection{Asymptotics of solutions of non-homogeneous problems}
To obtain explicitely the main order of the asymptotics of the radial part of the approximate solution $v_\ell$ constructed in Lemma~\ref{pr:AS1}, we consider a simplified problem as $|x|\to +\infty$. 
For $-1< \ell <1$ and $t>0$, let
\[
f^{\sharp}_\ell(t,x)= t^{-3} \langle x_\ell\rangle^{-3},\quad g^{\sharp}_\ell= \ell t^{-2} \partial_{x_1} \langle x_\ell\rangle^{-3}\]
and
\begin{equation}\label{eq:v33}
v^{\sharp}_\ell(t)=\int_0^\infty \frac{\sin(s\sqrt{-\Delta})}{\sqrt{-\Delta}} \left( f^{\sharp}_\ell + g^{\sharp}_\ell\right) (t+s)ds.\end{equation}
\begin{lemma}[Asymptotics for a non-homogeneous wave problem]\label{le:B2} 
For any $0<\delta<1$, 
for all $m\geq 0$, $|\alpha|= 1$, $|\alpha'|\geq 2$, $t>1$, $x\in {\mathbb{R}}^5$,
\begin{equation}\label{asymp2}\begin{aligned}
|A_\ell^m v^{\sharp}_\ell(t,x)| &\lesssim (t+\langle x_\ell\rangle)^{-1} t^{-(1+m)} \langle x_\ell\rangle^{-2+\delta},\\ 
|A_\ell^m\partial^\alpha v^{\sharp}_\ell(t,x)| &\lesssim (t+\langle x_\ell\rangle)^{-1}t^{-(1+m)} \langle x_\ell\rangle^{-3+\delta},\\
|A_\ell^m\partial^{\alpha'} v^{\sharp}_\ell(t,x)| &\lesssim (t+\langle x_\ell\rangle)^{-1}t^{-(1+m)} \langle x_\ell\rangle^{-4+\delta}.
\end{aligned}
\end{equation}
Moreover, for $r>0$, $t>1$, let
\begin{equation}\label{def:w}
\phi_\ell(t,r) = r^2 \partial_r V^{\sharp}_\ell(t,r) +3rV^{\sharp}_\ell(t,r),\quad
V^{\sharp}_\ell(t,r)=\fint_{|x|=r} v^{\sharp}_\ell(t,x)d\omega(x) .
\end{equation}
Then, for all $1\ll t<r^{\frac {11}{12}}$, 
\begin{equation}\label{asymp} 
\phi_\ell(t,r)=
(1-\ell^2)^{\frac12} r^{-3} + O(r^{-1}t^{-\frac 94}).
\end{equation}
\end{lemma}
The first part of Lemma~\ref{le:B2} is a consequence of Lemma~\ref{le:NH} and the decay properties of the functions $f^{\sharp}_\ell$ and $g^{\sharp}_\ell$. The second part is proved in Appendix~\ref{app:3}.
Note that we do not determine the asymptotic behavior of $v_\ell^\sharp$, but only the one of its spherical means.
\smallskip

Now, we check that the asymptotics of $v_\ell$ defined in Lemma~\ref{pr:AS1} and 
of $v^{\sharp}_\ell$ defined in Lemma~\ref{le:B2} coincide at the main order, up to a multiplicative constant.
\begin{lemma}[Comparison of asymptotics]\label{le:3.5}
Let $-1< \ell <1$.
For all $0<\delta<1$, $t>1$, $x\in {\mathbb{R}}^5$,
\begin{equation}\label{e:n42}\begin{aligned}
\left|v_\ell(t,x)+\frac 32(15)^{\frac 32} \kappa_\ell v^{\sharp}_\ell(t,x)\right| 
&\lesssim t^{-2} \langle x_\ell\rangle^{-3+\delta},
\\ 
\left|\nabla v_\ell(t,x)+\frac 32(15)^{\frac 32} \kappa_\ell \nabla v^{\sharp}_\ell(t,x)\right| 
&\lesssim t^{-2} \langle x_\ell\rangle^{-4+\delta}.
\end{aligned}\end{equation}
\end{lemma}
\begin{proof} By the properties of $v_\ell$ in Lemma~\ref{pr:AS1}, we have
\[ 
v_\ell(t)+\frac 32(15)^{\frac 32} \kappa_\ell v^{\sharp}_\ell(t)=\int_0^\infty \frac{\sin(s\sqrt{-\Delta})}{\sqrt{-\Delta}} \left(\mathcal E_\ell + {\mathcal E}_\ell^{\rm I} + { {\mathcal E}}^{\rm II}_\ell\right)(s+t) ds
\]
where $\mathcal E_\ell$ is defined in Lemma~\ref{pr:AS1} and
\[
{\mathcal E}^{\rm I}_\ell= f_\ell+g_\ell +\frac 32(15)^{\frac 32}\kappa_\ell\left(f_\ell^\sharp+g_\ell^\sharp\right)
\quad \hbox{and}\quad
 { {\mathcal E}}^{\rm II}_\ell=\frac 73 w_\ell^{\frac 43} v_\ell.
\]
First, we see from~\eqref{pres1} that $\mathcal E_\ell$ satisfies~\eqref{a:F} with $q=4-\delta$ and $p=3$. Therefore, applying Lemma~\ref{le:NH},
\[
\left| \int_0^\infty \frac{\sin(s\sqrt{-\Delta})}{\sqrt{-\Delta}} \mathcal E_\ell (s+t) ds\right|
\lesssim t^{-2+\delta} \langle x_\ell\rangle^{-1} (t+\langle x_\ell\rangle)^{-2}\lesssim t^{-2} \langle x_\ell\rangle^{-3+\delta}.
\]
Second, we observe from the explicit expression of $W$
\begin{align*}
 \Lambda W(x) = -\frac 32 (15)^{\frac 32} \langle x \rangle^{-3} + O(\langle x\rangle^{-5}) ,\quad \partial_{x_1}(\Lambda W)(x) = -\frac 32 (15)^{\frac 32} \partial_{x_1}(\langle x \rangle^{-3}) + O(\langle x\rangle^{-6}).
\end{align*}
In particular, by the definitions of $f_\ell$ and $g_\ell$, we have
\[
|A_\ell^m { {\mathcal E}^{\rm I}_\ell} |\lesssim t^{-3-m} \langle x_\ell \rangle^{-4}+t^{-2-m} \langle x_\ell \rangle^{-6}.\]
Applying Lemma~\ref{le:NH} with $q=3$, $p=4$ and $q=2$, $p=6$, we
obtain
\[
\left| \int_0^\infty \frac{\sin(s\sqrt{-\Delta})}{\sqrt{-\Delta}} { {\mathcal E}}^{\rm I}_\ell (s+t) ds\right|
\lesssim (t+\langle x_\ell\rangle)^{-2} \left(t^{-1} \langle x_\ell\rangle^{-2} + \langle x_\ell\rangle^{-3} \right)
 \lesssim t^{-2} \langle x_\ell\rangle^{-3}.
\]
Third, from~\eqref{e:n40} and the properties of $w_\ell$, $ { {\mathcal E}}^{\rm II}_\ell$ satisfies~\eqref{a:F} with $q=2$, $p=6-\delta$.
Thus, by Lemma~\ref{le:NH},
\[
\left| \int_0^\infty \frac{\sin(s\sqrt{-\Delta})}{\sqrt{-\Delta}} { {\mathcal E}}^{\rm II}_\ell (s+t) ds\right|
\lesssim t^{-2} \langle x_\ell\rangle^{-3}.
\]
This proves the estimate on $v_\ell+\frac 32(15)^{\frac 32} \kappa_\ell v_\ell^{\sharp}$. The estimate for the gradient is similar.
\end{proof}
\section{Refined approximate solution for the two-soliton problem}\label{sec:4}
For $k=1,2$, let
\[
\lambda_k^\infty>0,\quad \mathbf{y}_k^\infty\in {\mathbb{R}}^5,\quad \epsilon_k=\pm1,\quad {\boldsymbol{\ell}}_k=\ell_k {\mathbf e}_1,
\quad \hbox{where $-1\leq \ell_1<\ell_2<1$}.\]
Indeed, by rotation invariance and the Lorentz transformation, we restrict ourselves without loss of generality to the case where ${\boldsymbol{\ell}}_k=\ell_k \,{\bf e}_1$, with $\ell_k\in (-1,1)$ for $k=1,2$. See more details in \S5 of~\cite{MMwave1}.
\smallbreak

Let $C_0\gg1$ and $T_0\gg 1$ to be fixed and $I\subset [T_0,+\infty)$ be an interval of ${\mathbb{R}}$. For $k=1,2$,
we consider $\mathcal C^1$ functions $\lambda_k>0$, $\mathbf{y}_k\in {\mathbb{R}}^5$ defined on $I$.
We assume that these functions satisfy, for all $t\in I$,
\begin{equation}\label{BS0}
|\lambda_k(t)-\lambda^\infty_k|+|\mathbf{y}_k(t)-\mathbf{y}_k^\infty| \leq C_0 t^{-1},\quad
 | \dot {\lambda}_k |+ | \dot {\mathbf y}_k |\leq C_0 t^{-2}.
\end{equation}
For $\vec G= (G,H)$, define
\[(\theta_k G)(t,x) = \frac {\epsilon_k}{\lambda_k^{\frac 32}(t)} G\left(\frac {x -{\boldsymbol{\ell}}_k t - \mathbf{y}_k(t)}{\lambda_k(t)}\right),
\quad
\vec \theta_k \vec G = \left(\begin{array}{c}\theta_k G \\[.2cm] \displaystyle \frac {\theta_k} {\lambda_k } H \end{array}\right),\quad
\vec {\tilde \theta}_k \vec G = \left(\begin{array}{c}\displaystyle \frac {\theta_k}{\lambda_k} G\\[.4cm] \theta_k H \end{array}\right).
\]
In particular, set
\begin{equation}\label{defWk}
W_k = \theta_k W_{{\boldsymbol{\ell}}_k} ,\quad X_k=-{\boldsymbol{\ell}}_k\cdot \nabla W_k,\quad 
\vec W_k = \left(\begin{array}{c} W_k \\[.2cm] \displaystyle X_k \end{array}\right).
\end{equation}
Define also
\[
(\theta_k^\infty G)(t,x) = \frac {\epsilon_k}{(\lambda_k^\infty)^{\frac 32}} G\left(\frac {x -{\boldsymbol{\ell}}_k t -\mathbf{y}_k^\infty}{\lambda_k^\infty}\right),
\quad 
\vec \theta_k^\infty \vec G 
= \left(
\begin{array}{c}\theta_k^\infty G \\[.2cm] 
\displaystyle \frac {\theta_k^\infty} {\lambda_k^\infty} H \end{array}\right),\]
\[
 W_k^\infty =\theta_k^\infty W_{{\boldsymbol{\ell}}_k}, \quad X_k^\infty=-{\boldsymbol{\ell}}_k\cdot \nabla W_k^\infty,\quad 
\vec W_k^\infty =\left(\begin{array}{c} W_k^\infty \\[.2cm] \displaystyle X_k^\infty \end{array}\right) .
\]

\subsection{Main interaction terms}\label{pr:st}
Expanding the nonlinearity $|u|^{\frac 43}u$ at $u=W_1+W_2$, we identify the two main order interaction terms of the form $t^{-3} \sum_k c_k |W_k|^{\frac 43}$. The remaining error term is of size $t^{-4}$.
\begin{lemma}\label{le:int}
For $k,k'=1,2$, $k\neq k'$, let
\[
{\boldsymbol{\sigma}}_{k,k'} = 
\left(\frac 1{\sqrt{1-|{\boldsymbol{\ell}}_{k'}|^2}} - 1\right) \frac{{\boldsymbol{\ell}}_{k'} ({\boldsymbol{\ell}}_{k'}\cdot({\boldsymbol{\ell}}_k-{\boldsymbol{\ell}}_{k'}))}{|{\boldsymbol{\ell}}_{k'}|^2}
+{\boldsymbol{\ell}}_k-{\boldsymbol{\ell}}_{k'},
\]
and
$
c_{k} = \frac 73(15)^{\frac 32} {\epsilon_{k'}(\lambda_{k'}^\infty)^{\frac 32} }{|{\boldsymbol{\sigma}}_{k,k'}|^{-3}}.
$
Then,
\[
\left| \sum_k W_k\right|^{\frac 43} \left(\sum_k W_k\right)- \sum_{k} |W_k|^{\frac 43} W_k =
t^{-3}\sum_{k} c_k |W_k|^{\frac 43}+{\mathbf R}_{\Sigma },
\]
where, for all $t\in I$,
\begin{equation}\label{NLIN} 
\left\| {\mathbf R}_{\Sigma }\right\|_{H^1} \lesssim t^{-4}.
\end{equation}
\end{lemma}
\begin{remark}\label{rk:simple}
Lemma~\ref{le:int} holds for general ${\boldsymbol{\ell}}_k$, ${\boldsymbol{\ell}}_{k'}$. For the special case $ {\boldsymbol{\ell}}_k=\ell_k {\mathbf e}_1$ and $ {\boldsymbol{\ell}}_{k'}=\ell_{k'} {\mathbf e}_1$, we obtain 
\begin{equation}\label{e:simple2}
{\boldsymbol{\sigma}}_{1,2} =\frac {\ell_1-\ell_2}{(1-\ell_2^2)^{\frac 12}} {\mathbf e_1},\quad
{\boldsymbol{\sigma}}_{2,1} =\frac {\ell_2-\ell_1}{(1-\ell_1^2)^{\frac 12}} {\mathbf e_1},
\end{equation}
and 
\begin{equation}\label{e:simple}
 c_{1} = \frac 73(15)^{\frac 32} \epsilon_2 (\lambda_2^\infty)^{\frac 32}\frac{(1-\ell_2^2)^{\frac 32}}{|\ell_1-\ell_2|^3},\quad
 c_{2} = \frac 73(15)^{\frac 32} \epsilon_1 (\lambda_1^\infty)^{\frac 32}\frac{(1-\ell_1^2)^{\frac 32}}{|\ell_1-\ell_2|^3}.
\end{equation}
\end{remark}
\begin{proof}[Proof of Lemma~\ref{le:int}]
Let $\sigma= \frac 1{10} |{\boldsymbol{\ell}}_1-{\boldsymbol{\ell}}_2|$ and 
$
B_k(t)=\{x,\ |x-{\boldsymbol{\ell}}_k t|\leq \sigma t\}$, $B(t)=\cup_k B_k(t)$. We prove the $L^2$ estimate of ${\mathbf R}_\Sigma$, the proof of the $\dot H^1$ estimate is similar.
First, we claim 
\[
\|W_k^{\frac 73}\|_{L^2(B^c_k)} \lesssim t^{-4},\quad
\|W_k^{\frac 43}\|_{L^2(B^c_k)} \lesssim t^{-1}.
\]Indeed, for all $x\not\in B_k$, 
\[
|W_k|^{\frac 73}\lesssim \langle x-{\boldsymbol{\ell}}_k t \rangle^{-7} \lesssim t^{-4} \langle x-{\boldsymbol{\ell}}_k t \rangle^{-3},\quad
|W_k|^{\frac 43}\lesssim \langle x-{\boldsymbol{\ell}}_k t \rangle^{-4} \lesssim t^{-1} \langle x-{\boldsymbol{\ell}}_k t \rangle^{-3},
\]
and $x\mapsto\langle x \rangle^{-3} \in L^2({\mathbb{R}}^5)$.

Second, we claim, for $k,k'=1,2$, $k'\neq k$,
\begin{equation}
\label{step2}
\left\| \left| W_1+W_2\right|^{\frac 43} (W_1+W_2)- |W_k|^{\frac 43} W_k- \frac 73 |W_k|^{\frac 43} W_{k'} \right\|_{L^2(B_k)}
\lesssim t^{-\frac 92}.
\end{equation}
Indeed, for $x\in B_k$ and $k'\neq k$,
$
|W_k|^{\frac 13} |W_{k'}|^2 \lesssim t^{-6} \langle x-{\boldsymbol{\ell}}_k t \rangle^{-1}, 
$
$|W_{k'}|^{\frac 73}\lesssim t^{-7}$,
and so
\[
\left\| |W_k|^{\frac 13} |W_{k'}|^2 \right\|_{L^2(B_k)}+ \left\| |W_{k'}|^{\frac 73} \right\|_{L^2(B_k)} \lesssim t^{-\frac 92}.
\]
Thus,~\eqref{step2} follows from the Taylor expansion
\begin{multline*}
 \left| W_k + W_{k'}\right|^{\frac 43} (W_k + W_{k'}) 
 = \left| W_k \right|^{\frac 43} W_k \left| 1+ \frac {W_{k'}}{W_k} \right|^{\frac 43} 
\left( 1+ \frac {W_{k'}}{W_k} \right)\\ \quad 
 = |W_k|^{\frac 43} W_k + \frac 73 \left| W_k \right|^{\frac 43} W_{k'}
+O\left(|W_k|^{\frac 13} |W_{k'}|^2\right)+O\left(|W_{k'}|^{\frac 73}\right).
\end{multline*}

Third, we claim, $k'\neq k$,
\begin{equation}
\label{step3}
\left\| |W_k|^{\frac 43} \left(W_{k'}(t,x)-W_{k'}(t,{\boldsymbol{\ell}}_k t) \right) \right\|_{L^2(B_k)} \lesssim t^{-4}.
\end{equation}
Indeed, for $x\in B_k$,
\[\left|W_{k'}(t,x)-W_{k'}(t,{\boldsymbol{\ell}}_k t) \right| \lesssim \sup_{B_k} |\nabla W_{k'}(t)| \cdot |x-{\boldsymbol{\ell}}_k t|
\lesssim t^{-4} |x-{\boldsymbol{\ell}}_k t|,
\]
and so,
\[
|W_k|^{\frac 43} \left|W_{k'}(t,x)-W_{k'}(t,{\boldsymbol{\ell}}_k t) \right| \lesssim
t^{-4} \langle x-{\boldsymbol{\ell}}_k t\rangle^{-3},
\]
which implies~\eqref{step3}.

Last, note from the explicit expression~\eqref{defW} of $W$ the following asymptotics for $|x|\gg 1$,
$|W(x) - 15^{\frac 32} |x|^{-3}| \lesssim |x|^{-5}$.
Thus, using the assumptions of the parameters~\eqref{BS0} and the definition of $W_{k'}$ from~\eqref{defWbb} and~\eqref{defWk}, we have
\begin{multline*}
W_{k'}(t,{\boldsymbol{\ell}}_kt) = \frac {\epsilon_{k'}}{\lambda_{k'}^{\frac 32}(t)}
W_{\ell_{k'}}\left(\frac{({\boldsymbol{\ell}}_k-{\boldsymbol{\ell}}_{k'})t-\mathbf y_{k'}(t)}{\lambda_{k'}(t)}\right) = \frac {\epsilon_{k'}}{(\lambda_{k'}^\infty)^{\frac 32}}
W_{\ell_{k'}}\left(\frac{({\boldsymbol{\ell}}_k-{\boldsymbol{\ell}}_{k'})t}{\lambda_{k'}^\infty}\right) +O(t^{-4})\\
= \frac {\epsilon_{k'}}{(\lambda_{k'}^\infty)^{\frac 32}}
W \left(\frac{{\boldsymbol{\sigma}}_{k,k'}t}{\lambda_{k'}^\infty}\right) +O(t^{-4})
= {15^{\frac 32}\epsilon_{k'}}{(\lambda_{k'}^\infty)^{\frac 32}}
|{\boldsymbol{\sigma}}_{k,k'}|^{-3} t^{-3} +O(t^{-4}).
\end{multline*}
 Gathering these estimates, we find~\eqref{NLIN}.
\end{proof}
 \subsection{The approximate solution $\vec{\mathbf W}$}\label{sec42}
To remove the main interaction terms $c_1 t^{-3} |W_1|^{\frac 43}$ and
$c_2 t^{-3} |W_2|^{\frac 43}$ computed in Lemma~\ref{le:int}, we define suitably rescaled versions of the function $v_{\ell_k}$ given by Lemma~\ref{pr:AS1}.
Let
\begin{align}
v_k(t,x) & = \frac 1{ \lambda_k^3} v_{\ell_k}\left( \frac t{ \lambda_k} , \frac {x-\mathbf{y}_k}{ \lambda_k}\right),
\label{def.vk}\\
z_k(t,x) & = \frac 1{ \lambda_k^4} (\partial_tv_{\ell_k})\left( \frac t{ \lambda_k} , \frac {x-\mathbf{y}_k}{ \lambda_k}\right)
+ \frac{\kappa_{\ell_k}\epsilon_k}{2 \lambda_k^{\frac 12}t^2}\Lambda_k W_k(t,x)
\label{def.wk} \end{align}
and
\[
\vec v_k = \left(\begin{array}{c} v_k 
\\[.2cm]z_k \end{array}\right),\quad 
\kappa_{\ell_k} = - (1-\ell_k^2) \frac{(W^{\frac 43},\Lambda W)}{\|\Lambda W\|_{L^2}^2},\quad
a_k = - \frac {c_k\kappa_{\ell_k}\epsilon_k}2.
\]
Set
\begin{equation}\label{def:Wapp}
\vec{\mathbf W} = \left(\begin{array}{c} {\mathbf W}
\\[.2cm] {\mathbf X} \end{array}\right) =
\sum_{k=1,2} \left( \vec W_k + c_k\vec v_k \right) .
 \end{equation}
\begin{lemma}\label{le:43}
Assume~\eqref{BS0}.
Then, the function $\vec{\mathbf W}$ satisfies on $I\times {\mathbb{R}}^5$ 
\begin{equation}\label{syst_WX}\left\{\begin{aligned}
\partial_t{\mathbf W} & = {\mathbf X} - {\rm Mod}_{{\mathbf W}} - {\mathbf R}_{{\mathbf W}}\\
\partial_t	{\mathbf X} & 
	= \Delta {\mathbf W} +|{\mathbf W}|^{\frac 43} {\mathbf W} - {\rm Mod}_{{\mathbf X}}- {\mathbf R}_{{\mathbf X}}
\end{aligned}\right.\end{equation} 
where $\Lambda_k = \frac 32 + (x -{\boldsymbol{\ell}}_k t -\mathbf{y}_k)\cdot \nabla$,
\begin{align*}
{\rm Mod}_{{\mathbf W}} &= \sum_k \left( \frac{\dot\lambda_k}{\lambda_k} - \frac{a_k}{\lambda_k^{\frac 12}t^2} \right) \Lambda_k W_k
+ \sum_k \dot {\mathbf{y}}_k \cdot \nabla W_k \\
{\rm Mod}_{{\mathbf X}} &= 
- \sum_k \left( \frac{\dot\lambda_k}{\lambda_k} - \frac{a_k}{\lambda_k^{\frac 12}t^2} \right) ({\boldsymbol{\ell}}_k \cdot \nabla) \Lambda_k W_k 
- \sum_k ({\dot {\mathbf{y}}}_k \cdot \nabla) ( {\boldsymbol{\ell}}_k \cdot \nabla) W_k ,
\end{align*}
\begin{align}
\vec{\mathbf{R}}= \left(\begin{array}{c} {\mathbf R}_{{\mathbf W}}
\\[.2cm] {\mathbf R}_{{\mathbf X}}\end{array}\right),\quad 
\|\vec{\mathbf{R}}\|_{\dot H^1\times L^2} +\|\nabla \vec{\mathbf{R}}\|_{\dot H^1\times L^2}\lesssim t^{-4+\delta}.
\label{RWX}\end{align}
Moreover, for all $0<\delta< 1$,
\begin{equation}\label{e:WX}\begin{aligned}
&|{\mathbf W}|+\langle x-{\boldsymbol{\ell}}_k t\rangle |\nabla {\mathbf W}|
\lesssim \sum_k \left( \langle x-{\boldsymbol{\ell}}_k t\rangle^{-3}+t^{-1}\langle x-{\boldsymbol{\ell}}_k t\rangle^{-3+\delta}\right),\\
&|{\mathbf X}|\lesssim \sum_k \left( \langle x-{\boldsymbol{\ell}}_k t\rangle^{-4}+t^{-2}\langle x-{\boldsymbol{\ell}}_k t\rangle^{-3+\delta}\right).
\end{aligned}\end{equation}
\end{lemma}
\begin{proof}
\textbf{Proof of~\eqref{e:WX}.}
The estimates~\eqref{e:WX} on ${\mathbf W}$ and ${\mathbf X}$ are consequences of the decay of the function $W$ and of the estimates~\eqref{e:n40} of $v_\ell$. See also~\eqref{rk33} for estimates on time derivatives.
\smallskip

\textbf{Equation of $\vec v_k$.}
We claim that $\vec v_k$ satisfies the following system
\begin{equation}\label{e:n40bisk} 
 \left\{\begin{aligned}
\partial_t v_k & = z_k - \frac{\kappa_{\ell_k}\epsilon_k}{2 t^2 \lambda_k^{\frac 12}} \Lambda_k W_k + O_{\dot H^{1}\cap \dot H^2}(t^{-4})\\
\partial_t z_k & = \Delta v_k + \frac 73 |W_k|^{\frac 43} v_k
+ \frac 1{t^3} |W_k|^{\frac 43} +\frac {\kappa_{\ell_k}\epsilon_k}{2t^2\lambda_k^{\frac 12}} ({\boldsymbol{\ell}}_k \cdot\nabla) \Lambda_k W_k
+ O_{H^1}\big(t^{-4+\delta}\big)
\end{aligned}\right.
\end{equation}
First, note that 
\begin{multline*}
\partial_tv_k(t,x) = \frac 1{ \lambda_k^4} (\partial_tv_{\ell_k})\left( \frac t{ \lambda_k} , \frac {x-\mathbf{y}_k}{ \lambda_k}\right)
-3 \frac{\dot \lambda_k}{\lambda_k^4} v_{\ell_k} \left( \frac t{ \lambda_k} , \frac {x-\mathbf{y}_k}{ \lambda_k}\right)
- \frac{\dot \lambda_k}{\lambda_k^4} \frac{t}{\lambda_k} (\partial_t v_{\ell_k}) \left( \frac t{ \lambda_k} , \frac {x-\mathbf{y}_k}{ \lambda_k}\right) \\ 
-\frac{\dot \lambda_k}{\lambda_k^4}\left(\frac{x-\mathbf{y}_k}{\lambda_k}\right)\cdot \nabla v_{\ell_k} \left( \frac t{\lambda_k} , \frac {x-\mathbf{y}_k}{ \lambda_k}\right) - \frac{\dot {\mathbf{y}}_k}{\lambda_k^4} \cdot\nabla v_{\ell_k} \left( \frac t{ \lambda_k} , \frac {x-\mathbf{y}_k}{ \lambda_k}\right).
\end{multline*}
Thus, by the definition of $z_k$ in~\eqref{def.wk} and using the notation $A_\ell= \partial_t + \ell \partial_{x_1}$, we obtain
\begin{multline*}
\partial_t v_k(t,x) = z_k- \frac{\kappa_{\ell_k}\epsilon_k}{2 \lambda_k^{\frac 12 }t^2} \Lambda_k W_k -3 \frac{\dot \lambda_k}{\lambda_k^4} v_{\ell_k} \left( \frac t{ \lambda_k} , \frac {x-\mathbf{y}_k}{ \lambda_k}\right)
- \frac{\dot \lambda_k}{\lambda_k^4} \frac{t}{\lambda_k}A_{\ell_k} v_{\ell_k} \left( \frac t{ \lambda_k} , \frac {x-\mathbf{y}_k}{ \lambda_k}\right) \\ 
-\frac{\dot \lambda_k}{\lambda_k^4}\left(\frac{x-{\boldsymbol{\ell}}_k t-\mathbf{y}_k}{\lambda_k}\right)\cdot \nabla v_{\ell_k} \left( \frac t{\lambda_k} , \frac {x-\mathbf{y}_k}{ \lambda_k}\right) - \frac{\dot {\mathbf{y}}_k}{\lambda_k^4} \cdot\nabla v_{\ell_k} \left( \frac t{ \lambda_k} , \frac {x-\mathbf{y}_k}{ \lambda_k}\right).\end{multline*}
 By~\eqref{BS0} and~\eqref{e:n40}, we obtain the first line of~\eqref{e:n40bisk}.

For the second line, we compute using the definition of $z_k$,
\begin{align*}
& \partial_t z_k = \frac{1}{\lambda_k^{5}} (\partial_t^2 v_{\ell_k} )\left( \frac t {\lambda_k}, \frac{x-\mathbf{y}_k} {\lambda_k}\right) 
- \frac{\dot {\mathbf{y}}_k}{\lambda_k^5} \cdot \nabla \partial_t v_{\ell_k} \left( \frac t{ \lambda_k} , \frac {x-\mathbf{y}_k}{ \lambda_k}\right)-4 \frac{\dot \lambda_k}{\lambda_k^5} \partial_t v_{\ell_k}\left( \frac t{ \lambda_k} , \frac {x-\mathbf{y}_k}{ \lambda_k}\right) 
\\ &\quad 
 - \frac{\dot \lambda_k}{\lambda_k^5} \frac{t}{\lambda_k} A_{\ell_k} \partial_t v_{\ell_k}\left( \frac t{ \lambda_k} , \frac {x-\mathbf{y}_k}{ \lambda_k}\right) 
 - \frac{\dot \lambda_k}{\lambda_k^5} \left(\frac{x-{\boldsymbol{\ell}} t-\mathbf{y}_k}{\lambda_k}\right)\cdot
\nabla \partial_t v_{\ell_k} \left( \frac t{ \lambda_k} , \frac {x-\mathbf{y}_k}{ \lambda_k}\right) 
\\&\quad - \frac{\kappa_{\ell_k}\epsilon_k}{t^3 \lambda_k^{\frac 12}} \Lambda_k W_k
 -\frac {\kappa_{\ell_k}\epsilon_k\ell_k }{2t^2 \lambda_k^{\frac 12}} 
 \partial_{x_1} \Lambda_k W_k -\frac {\kappa_{\ell_k}\epsilon_k}{2t^2\lambda_k^{\frac 12}} \frac{\dot\lambda_k}{\lambda_k}
\left(\frac 12 \Lambda_k W_k+\Lambda_k^2 W_k\right)
-\frac {\kappa_{\ell_k}\epsilon_k\ell_k }{2t^2\lambda_k^{\frac 12}} \dot{\mathbf{y}}_k\cdot\nabla\Lambda_k W_k
 \end{align*}
Thus, as before, by~\eqref{BS0} and~\eqref{e:n40} (see also~\eqref{rk33})
\begin{align*}
\partial_t z_k & = \frac{1}{\lambda_k^{5}} (\partial_t^2 v_{\ell_k} )\left( \frac t {\lambda_k}, \frac{x-\mathbf{y}_k} {\lambda_k}\right) - \frac{\kappa_{\ell_k}\epsilon_k}{t^3 \lambda_k^{\frac 12}} \Lambda W_k -\frac {\kappa_{\ell_k}\epsilon_k\ell_k }{2t^2\lambda_k^{\frac 12}} 
\partial_{x_1} \Lambda_k W_k
+O_{H^1} (t^{-4+\delta}).
\end{align*}
Therefore, inserting now~\eqref{e:n40bis} for $v_{\ell_k}$,
\begin{align*}
\partial_t z_k
& = \frac 1{\lambda_k^5} (\Delta v_{\ell_k})\left( \frac t {\lambda_k}, \frac{x-\mathbf{y}_k} {\lambda_k}\right)
 +\frac 73 \frac 1{\lambda_k^3} |W_k|^{\frac 43}v_{\ell_k}\left( \frac t {\lambda_k}, \frac{x-\mathbf{y}_k} {\lambda_k}\right)+ \frac 1{t^3} \left( |W_k|^{\frac 43} + \frac{\kappa_{\ell_k}\epsilon_k}{\lambda_k^{\frac 12}} \Lambda_k W_k\right)\\
&\quad 
 + \frac {\kappa_{\ell_k}\epsilon_k\ell_k}{t^2\lambda_k^{\frac 12}} \partial_{x_1} \Lambda W_k
- \frac{\kappa_{\ell_k}\epsilon_k}{t^3 \lambda_k^{\frac 12}}\Lambda_k W_k 
-\frac {\kappa_{\ell_k}\epsilon_k\ell_k }{2t^2\lambda_k^{\frac 12 }} 
 \partial_{x_1} \Lambda W_k
+O_{H^1}(t^{-4+\delta}),
\end{align*}
which gives the second line of~\eqref{e:n40bisk}.
\smallbreak

\textbf{Equation of $\vec {\mathbf W}$.} 
By direct computations, we check
\[
 \partial_t W_k =
- {{\boldsymbol{\ell}}_k} \cdot \nabla W_k
- \frac{\dot \lambda_k }{\lambda_k} \Lambda_k W_k
- \dot {\mathbf{y}}_k \cdot \nabla W_k.
\]
Thus, using also~\eqref{e:n40bisk},
\begin{equation} \label{eqe1}
\partial_t {\mathbf W}	= {\mathbf X} - {\rm Mod}_{{\mathbf W}}+ O_{\dot H^1\cap \dot H^2}(t^{-4}).
\end{equation}
Moreover, using~\eqref{aL} and~\eqref{e:n40bisk}
\begin{align*}
\partial_t{\mathbf X} & = - \partial_t \left(\sum_{k} {{\boldsymbol{\ell}}_k} \cdot \nabla W_k\right) +\sum_k c_k \partial_t z_k\\
& = \sum_{k} ({\boldsymbol{\ell}}_k\cdot \nabla)^2 W_k 
+ \sum_{k} \left(\frac{\dot \lambda_k } {\lambda_k}+\frac{c_k\kappa_{\ell_k}\epsilon_k}{2\lambda_k^{\frac 12}t^2}\right) 
({\boldsymbol{\ell}}_k\cdot\nabla)\Lambda_k W_k 
+ \sum_{k} (\dot {\mathbf{y}}_k\cdot \nabla)({\boldsymbol{\ell}}_k\cdot\nabla)W_k
\\
&
\quad +\sum_k c_k \left( \Delta v_k + \frac 73 |W_k |^{\frac 43} v_k
+ t^{-3} |W_k |^{\frac 43} \right)
+O_{H^1} (t^{-4+\delta})
\end{align*}
Note that
$({\boldsymbol{\ell}}_k\cdot \nabla)^2 W_k = \Delta W_k + |W_k|^{\frac 43} W_k$, and thus
$
\partial_t {\mathbf X} = \Delta {\mathbf W} + |{\mathbf W}|^{\frac 43} {\mathbf W} -{\rm Mod}_{\mathbf X}-{\mathbf R}_{\mathbf X}
$
with
\[
\mathbf{R}_{\mathbf X} = |{\mathbf W}|^{\frac 43} {\mathbf W}- \sum_k |W_k|^{\frac 43} W_k - t^{-3} \sum_k c_k |W_k |^{\frac 43} - \frac 73 \sum_k c_k |W_k |^{\frac 43} v_k +O_{H^1} (t^{-4+\delta}).\]
Note that 
$\mathbf{R}_{\mathbf X}= {\mathbf R}_v + {\mathbf R}_{\Sigma}+O_{H^1} (t^{-4+\delta})$
where 
\[
{\mathbf R}_v =|{\mathbf W}|^{\frac 43} {\mathbf W} - \left|\sum W_k\right|^{\frac 43} \left( \sum W_k\right)
- \frac 73\sum c_k|W_k|^{\frac 43} v_k,
\]
and from Lemma~\ref{le:int}, $\|{\mathbf R}_{\Sigma}\|_{H^1} \lesssim t^{-4}$.
Now, we prove $\|{\mathbf R}_v\|_{H^1} \lesssim t^{-4+\delta}$.
We decompose\begin{align*}
{\mathbf R}_v& =
\left|\sum W_k + c_k v_k \right|^{\frac 43} \left(\sum W_k + c_k v_k\right) - \left|\sum W_k\right|^{\frac 43} \left( \sum 
W_k\right)- \frac 73 \left|\sum W_k\right|^{\frac 43} \left(\sum c_k v_k\right) \\
& +\frac 73 \sum \left( c_k v_k \left(|W_1+W_2|^{\frac 43}-|W_k|^{\frac 43}\right)\right).
\end{align*}
First, we observe from~\eqref{e:n40}
\begin{align*}
&\left| \left|\sum W_k + c_k v_k \right|^{\frac 43} \left(\sum W_k + c_k v_k\right) - \left|\sum W_k\right|^{\frac 43} \left( \sum 
W_k\right)- \frac 73 \left|\sum W_k\right|^{\frac 43} \left(\sum c_k v_k\right) \right| \\
 &\quad \lesssim \sum |v_k|^2\lesssim t^{-4} \sum \langle x-{\boldsymbol{\ell}}_k t\rangle^{-4+\delta}
\end{align*}
and so this term is $O_{H^1}(t^{-4})$.
Second, we estimate $v_k \left( \left|W_1+W_2\right|^{\frac 43} - |W_k|^{\frac 43}\right)$ for $k=1,2$.
For $|x-{\boldsymbol{\ell}}_1 t|\leq \frac{|{\boldsymbol{\ell}}_1-{\boldsymbol{\ell}}_2|}{10} t$, using~\eqref{e:n40}, we have 
\[
\left| v_1 \left( \left|W_1+W_2\right|^{\frac 43} - |W_1|^{\frac 43}\right)\right| 
\lesssim 
|v_1| |W_2| \left( |W_1|^{\frac 13} + |W_2|^{\frac 13}\right)
\lesssim t^{-4} |W_1|^{1-\delta} \left( |W_1|^{\frac 13} + |W_2|^{\frac 13}\right).
\]
For $|x-{\boldsymbol{\ell}}_1 t|> \frac{|{\boldsymbol{\ell}}_1-{\boldsymbol{\ell}}_2|}{10} t$, also using~\eqref{e:n40}, we have
\begin{align*}
\left| v_1 \left( \left|W_1+W_2\right|^{\frac 43} - |W_1|^{\frac 43}\right)\right| 
&\lesssim t^{-4+\delta} \left(|W_1|^{\frac 43} + |W_2|^{\frac 43}\right).
\end{align*}
The same holds for the term $v_2 \left( \left|W_1+W_2\right|^{\frac 43} - |W_2|^{\frac 43}\right)$ and we obtain
$\|{\mathbf R}_v\|_{H^1} \lesssim t^{-4+\delta}$.\end{proof}
\section{Refined construction of a two-soliton solution}\label{sec:5}
To construct the two-soliton solution at $+\infty$, we follow the strategy of~\cite{MMwave1} using the refined approximate solution $\vec{\mathbf W}$ defined in the previous section. As in~\cite{MMwave1} and several other previous papers on multiple solitons, see \emph{e.g.}~\cite{Mnls,Ma,MMnls,CMM,CMkg}, we argue by compactness and obtain the solution $u(t)$ as the limit of a sequence of approximate multi-solitons $u_n(t)$.
\begin{proposition}\label{pr:S1}
There exist $T_0>0$
and a solution $u(t)$ of~\eqref{wave} on $[T_0,+\infty)$ satisfying, for all $t\in [T_0,+\infty)$,
\begin{equation}\label{pr:S11}
\left\|\nabla u(t) - \nabla {\mathbf W}(t)\right\|_{L^2}
+\left\| \partial_t u(t) - {\mathbf X}(t) \right\|_{L^2} \lesssim t^{-3+\frac 1{10}}
\end{equation}
where $\lambda_k(t)$, $\mathbf{y}_k(t)$ are such that, for all $t\in [T_0,+\infty)$,
\begin{equation}\label{blay}
|\lambda_k(t)-\lambda_k^\infty|+|\mathbf{y}_k(t)-\mathbf{y}_k^\infty|\lesssim t^{-1}.
\end{equation}
\end{proposition}
This section is devoted to the proof of Proposition~\ref{pr:S1}.
Since
the ansatz $\vec{\mathbf W}$ takes into account the consequence of the main order of the interactions of the two waves, the two-soliton solution $u(t)$ is computed in~\eqref{pr:S11} up to order $t^{-3+\frac 1{10}}$ (the loss of the exponent $\frac 1{10}$ has no special meaning here) to be compared with~\cite{MMwave1}, where the corresponding error is of size $t^{-2}$ (see (4.9) in~\cite{MMwave1}). A computation at order $t^{-3+\frac1{10}}$ will allow us to justify the non-zero dispersive part and thus to finish the proof of Theorem~\ref{th.1} in the next section.
\smallbreak

Let $S_n\to +\infty$.
Let $\zeta_{k,n}^{\pm}\in {\mathbb{R}}$ small to be determined later. These free parameters correspond to two exponentially stable/unstable directions for each soliton - see statements of Proposition~\ref{pr:s4}, Claim~\ref{le:modu2} and Lemma~\ref{le:bs2}.
For any large $n$, we consider the solution $u_n$ of
\begin{equation}\label{defun}
\left\{
\begin{aligned} 
& \partial_t^2 u_n - \Delta u_n - |u_n|^{\frac 4{3}} u_n = 0 \\ 
& (u_n(S_n),\partial_t u_n(S_n))^\mathsf{T} = \sum_{k=1,2} \left( \vec W_{k}^\infty (S_n) + c_k \vec v_k(S_n)
+\sum_\pm \zeta_{k,n}^\pm ( \vec \theta_k^\infty \vec Z_{{\boldsymbol{\ell}}_k}^\pm)(S_n)\right)
\end{aligned}
\right.
\end{equation}
Note that since $(u_n(S_n),\partial_t u_n(S_n))\in \dot H^1\times L^2$, the solution $\vec u_n$ is well-defined in $\dot H^1\times L^2$ at least on a small interval of time around $S_n$.

Now, we state uniform estimates on $u_n$ backwards in time up to some uniform $T_0\gg1$.
\begin{proposition}\label{pr:s4} There exist $n_0>0$ and $T_0>0$ such that, for any $n\geq n_0$, there exist
$(\zeta_{k,n}^{\pm})_{k=1,2}\in {\mathbb{R}}^2\times {\mathbb{R}}^2$, with
\[
 \sum_k |\zeta_{k,n}^{\pm}|^2 \lesssim {S_n^{-7}},
\]
and such that the solution $\vec u_n=(u_n,\partial_t u_n)^\mathsf{T}$ of~\eqref{defun} is well-defined in $\dot H^1\times L^2$ on the time interval $[T_0,S_n]$
and satisfies, for all $t\in [T_0,+\infty)$,
\begin{equation}\label{eq:un} 
\left\|\vec u_n(t) -\vec {\mathbf W}_{n}(t)\right\|_{\dot H^1\times L^2} \lesssim t^{-3+\frac 1{10}}
\end{equation}
where $
\vec {\mathbf W}_{n}(t,x)=
\vec {\mathbf W}\left(t,x;\{\lambda_{k,n}(t)\},\{\mathbf{y}_{k,n}(t)\}\right)$ is defined in \S\ref{sec42}
and
\begin{equation}\label{eq:deux}
|\lambda_{k,n}(t)-\lambda_k^\infty|+|\mathbf{y}_{k,n}(t)-\mathbf{y}_{k}^\infty|\lesssim t^{-1},\quad
\left| \dot \lambda_{k,n} (t)\right|+\left|\dot {\mathbf{y}}_{k,n} (t)\right|\lesssim t^{-2}.
\end{equation}
Moreover, $\vec u_n\in \mathcal C([T_0,S_n],\dot H^2\times \dot H^1)$ and satisfies, for all $t\in [T_0,S_n]$,
$\|\vec u_n(t)\|_{\dot H^2\times \dot H^1} \lesssim 1$.
\end{proposition}
\subsection{Decomposition around $\vec{\mathbf W}$}
\begin{lemma}[Properties of the decomposition]\label{le:4} 
There exist $T_0\gg 1$ and $0<\delta_0\ll 1$ such that if $u(t)$ is a solution of~\eqref{wave} which satisfies on $I$,
\begin{equation}\label{hyp:4}
\left\|\vec u - \sum_{k=1,2} \left( \vec W_k^\infty + c_k\vec v_k \right) \right\|_{\dot H^1\times L^2}< \delta_0,
\end{equation}
then there exist $\mathcal C^1$ functions $\lambda_k>0$, $\mathbf{y}_k$ on $I$ such that, 
$\vec \varepsilon(t)$ being defined by
\begin{equation}
\label{eps2}
\vec \varepsilon=\left(\begin{array}{c}\varepsilon \\ \eta \end{array}\right),\quad 
\vec u = \left(\begin{array}{c} u \\ \partial_t u\end{array}\right) =
\vec {\mathbf W} + \vec \varepsilon, 
\end{equation}
the following hold on $I$, for $k=1,2$.
\begin{enumerate}[label=\emph{(\roman*)}]
\item\emph{First properties of the decomposition.} For $j=1,\ldots,5$, 
\begin{equation}
\label{ortho}
\pshbbk {\varepsilon}{ \Lambda_k W_k}=
\pshbbk {\varepsilon}{ \partial_{x_j} W_k}= 0,
\end{equation}
\begin{equation}\label{bounds}
|\lambda_k -\lambda_k^{\infty}|
+|\mathbf{y}_k -\mathbf{y}_k^\infty|
+\|\vec \varepsilon \|_{\dot H^1\times L^2}
\lesssim \left\|\vec u -\sum_{k=1,2} \left( \vec W_k^\infty + c_k\vec v_k \right) \right\|_{\dot H^1\times L^2}
\end{equation}
\item\emph{Equation of $\vec \varepsilon$.} 
\begin{equation}\label{syst_e}\left\{\begin{aligned}
\partial_t \varepsilon & = \eta + {\rm Mod}_{{\mathbf W}}+{\mathbf R}_{{\mathbf W}}\\
\partial_t \eta & 
= \Delta \varepsilon +\left| {\mathbf W} + \varepsilon\right|^{\frac 43} ({\mathbf W} + \varepsilon)
- |{\mathbf W}|^{\frac 43}{\mathbf W} + {\rm Mod}_{{\mathbf X}}+ {\mathbf R}_{{\mathbf X}} .
\end{aligned}\right.\end{equation} 
\item\emph{Parameter estimates.} For any $0<\delta<1$, \begin{equation}
\label{le:p}
\left|\frac {\dot \lambda_k }{\lambda_k} - \frac{a_k}{\lambda_k^{\frac 12}t^2}\right|+|\dot {\mathbf{y}}_k|
\lesssim \left\|\vec\varepsilon\right\|_{\dot H^1\times L^2} +t^{-4+\delta}.
\end{equation}
\item\emph{Unstable directions.} Let
$z_k^{\pm} = (\vec \varepsilon ,\vec {\tilde\theta}_k \vec Z_{{\boldsymbol{\ell}}_k}^{\pm})$.
Then, for any $0<\delta<1$, 
\begin{equation}
\label{le:z}
\left| \frac d{dt} z_k^{\pm} \mp \frac{\sqrt{\lambda_0}}{\lambda_k}(1-|{\boldsymbol{\ell}}_k|^2)^{\frac 12}z_k^{\pm} \right|
\lesssim \left\| {\vec\varepsilon }\right\|_{\dot H^1\times L^2}^2 + t^{-1} \left\|{\vec\varepsilon }\right\|_{\dot H^1\times L^2} + t^{-4+\delta}.
\end{equation}
\end{enumerate}
\end{lemma}
\begin{proof}
\textbf{Decomposition.} The existence of parameters ${\lambda_k}$ and ${\mathbf y}_k$ such that~\eqref{ortho} and~\eqref{bounds} hold is proved similarly as~(i) of Lemma~3.1 in~\cite{MMwave1}.
\smallbreak

\textbf{Equation of $\vec \varepsilon$.}
The equation of $\vec \varepsilon(t)$ is easily derived from the equation~\eqref{wave} of $u$ and~\eqref{syst_WX}.
Indeed, first, since $\varepsilon=u-{\mathbf W}$, we have
\[
\partial_t \varepsilon = \partial_t u- \partial_t {\mathbf W} = \eta +{\mathbf X} - \partial_t {\mathbf W} = \eta +{\rm Mod}_{\mathbf W}+\mathbf{R}_{\mathbf W}.
\]
Second, since $\eta=\partial_t u-{\mathbf X}$, we have
\begin{align*}
\partial_t \eta & = \partial_t ^2 u -\partial_t {\mathbf X} =\Delta u +|u|^{\frac 43} u -\Delta{\mathbf W} -|{\mathbf W}|^{\frac 43} {\mathbf W} +{\rm Mod}_{\mathbf X} + \mathbf{R}_{\mathbf X}\\
& = \Delta \varepsilon +|{\mathbf W}+\varepsilon|^{\frac 43} ({\mathbf W}+\varepsilon) - |{\mathbf W}|^{\frac 43} {\mathbf W} + {\rm Mod}_{{\mathbf X}} + \mathbf{R}_{\mathbf X}.
\end{align*}
We also denote
\begin{align*}
{\mathbf R}_{\rm NL} &= \left| {\mathbf W} + \varepsilon\right|^{\frac 43} ({\mathbf W} + \varepsilon)
- |{\mathbf W}|^{\frac 43}{\mathbf W} -\frac 73 \sum_k |W_k|^{\frac 43}\varepsilon ={\mathbf{R}}_{1}+ {\mathbf{R}}_{2}, \\
 {\mathbf{R}}_{1} &= \frac 73 \left( |{\mathbf W}|^{\frac 43} - \sum_k |W_k|^{\frac 43}\right) \varepsilon,\quad 
 {\mathbf{R}}_{2} = \left| {\mathbf W} + \varepsilon\right|^{\frac 43} ({\mathbf W} + \varepsilon)
- |{\mathbf W}|^{\frac 43}{\mathbf W} - \frac 73 |{\mathbf W}|^{\frac 43} \varepsilon ,
\end{align*}
and
\[
\vec {\mathcal L} = \left(\begin{array}{cc}0 & 1 \\
\Delta + \frac 73 \sum_k |W_k|^{\frac 43} & 0\end{array}\right),\quad
\vec {\rm Mod}= \left(\begin{array}{c} {\rm Mod}_{\mathbf W} \\{\rm Mod}_{\mathbf X} \end{array}\right),\quad
\vec {{\mathbf R}} = \left(\begin{array}{c} {\mathbf R}_{\mathbf W} \\{\mathbf R}_{\mathbf X}\end{array}\right),
\]
\[ 
\vec {{\mathbf R}}_{\rm NL} = \left(\begin{array}{c}0 \\{\mathbf R}_{\rm NL}\end{array}\right),\quad
\vec {{\mathbf R}}_{1} = \left(\begin{array}{c}0 \\{\mathbf R}_{1}\end{array}\right),\quad \vec {{\mathbf R}}_{2} = \left(\begin{array}{c}0 \\{\mathbf R}_{2}\end{array}\right).
\] 
With this notation, the system~\eqref{syst_e} rewrites 
\begin{equation}\label{syst_e3}
\partial_t \vec \varepsilon = \vec {\mathcal L} \vec \varepsilon + \vec {\rm Mod} + \vec {\mathbf{R}}+\vec {\mathbf{R}}_{\rm NL} = \vec {\mathcal L} \vec \varepsilon + \vec {\rm Mod} + \vec {\mathbf{R}}+\vec {\mathbf{R}}_1+\vec {\mathbf{R}}_2.
\end{equation}
We claim the following estimates on ${\mathbf R}_1$ and ${\mathbf R}_2$
\begin{equation}\label{R1bis}
\| {\mathbf{R}}_1\|_{L^{\frac{10}7}}\lesssim 
t^{-1} \|\varepsilon\|_{L^{\frac{10}3}} ,\quad
 | {\mathbf{R}}_2|\lesssim |{\mathbf W}|^{\frac13}\varepsilon^2 + |\varepsilon|^{\frac 73},\quad
 \| {\mathbf{R}}_2\|_{L^{\frac {10}7}} \lesssim \|\varepsilon\|_{L^{\frac{10}3}}^2 \lesssim \|\varepsilon\|_{\dot H^1}^{2}.
 \end{equation}
 The estimate on $\mathbf R_2$ follows from~\eqref{lor}. To prove the estimate on $\mathbf R_1$, we first recall
 the inequality, for $p>1$, for any reals $(r_k)$,
\begin{equation}\label{convp}
\left|\left|\sum r_k \right|^{p} - \sum |r_k|^p\right| \lesssim \sum_{k'\neq k} |r_{k'}||r_k|^{p-1}
\end{equation}
Therefore,
\begin{multline*}
\left| |{\mathbf W}|^{\frac 43} - \sum |W_k|^{\frac 43}\right|\lesssim \left| |{\mathbf W}|^{\frac 43} - \left|\sum W_k\right|^{\frac 43}\right|+\left| \left|\sum W_k\right|^{\frac 43} - \sum|W_k|^{\frac 43}\right|\\
\lesssim \left(\sum |v_k|\right) \left( \sum |W_k| +|v_k|\right)^{\frac 13} + \sum_{k'\neq k} |W_k| |W_{k'}|^{\frac 13}.
\end{multline*}
and thus
\[
| {\mathbf{R}}_1|\lesssim 
 |\varepsilon| \left(\sum |v_k|\right) \left( \sum |W_k| +|v_k|\right)^{\frac 13} + |\varepsilon|\sum_{k'\neq k} |W_k| |W_{k'}|^{\frac 13}.
\]
By~\eqref{lor},
 we obtain
\[
\| {\mathbf{R}}_1\|_{L^{\frac{10}7}}\lesssim 
 \|\varepsilon\|_{L^{\frac{10}3}}
 \left( \sum \|v_k\|_{L^{\frac{10}3}}\right) \left(\sum\|W_k\|_{L^{\frac{10}3}}^{\frac 13} + \|v_k\|_{L^{\frac{10}3}}^{\frac 13}\right) + \|\varepsilon\|_{L^{\frac{10}3}}\sum_{k'\neq k} 
 \left \|W_k|W_{k'}|^{\frac 13}\right\|_{L^{\frac 52}}.
\]
By~\eqref{e:n40}, we have $\|v_k\|_{L^{\frac{10}3}}\lesssim t^{-2}$.
Moreover $\|W_k|W_{k'}|^{\frac 13} \|_{L^{\frac 52}}\lesssim t^{-1}$ is a consequence of the following technical result.
\begin{claim}[Claim~2 in~\cite{MMwave1}]\label{WW}
Let $0<r_2\leq r_1$ be such that $r_1+r_2> \frac 53$. For $t$ large, 
if $r_1> \frac 53$ then $\int |W_1|^{r_1} |W_2|^{r_2}\lesssim t^{-3 r_2}$, whereas if $r_1\leq \frac 53$ then $\int |W_1|^{r_1} |W_2|^{r_2} \lesssim t^{5-3 (r_1+r_2)}$.
\end{claim}

\textbf{Parameter estimates.}
Now, we derive the equations of $\lambda_k$ and $\mathbf{y}_k$ from the orthogonality conditions~\eqref{ortho}. First,
\begin{align*}
 \frac d{dt} \pshbbun {\varepsilon}{ \Lambda_1 W_1 } = 
\pshbbun{\partial_t \varepsilon}{ \Lambda_1 W_1 }+
\pshbbun {\varepsilon}{\partial_t \left(\Lambda_1 W_1\right) }=0
 \end{align*}
 Thus, using the first line of~\eqref{syst_e}, and the expression of ${\rm Mod}_{\mathbf W}$ in Lemma~\ref{le:43},
 \begin{align*}
0= & \pshbbun{\eta}{ \Lambda_1 W_1 }
 -\pshbbun{\varepsilon}{ {\boldsymbol{\ell}}_1 \cdot \nabla (\Lambda_1 W_1)} -\frac{\dot\lambda_1}{\lambda_1}\pshbbun {\varepsilon}{ \Lambda^2_1W_1 } -\pshbbun {\varepsilon}{{\dot {\mathbf{y}}_1}\cdot \nabla\Lambda_1 W_1} 
\\
& +\left( \frac{\dot\lambda_1}{\lambda_1} - \frac{a_1}{\lambda_1^{\frac 12}t^2}\right) \pshbbun{\Lambda_1 W_1 }{ \Lambda_1 W_1 }
 + \pshbbun{ \dot {\mathbf{y}}_1 \cdot \nabla W_1 }{ \Lambda_1 W_1 } \\
 & + \left( \frac{\dot\lambda_2}{\lambda_2} - \frac{a_2}{\lambda_2^{\frac 12}t^2}\right) \pshbbun{\Lambda_2 W_2}{\Lambda_1 W_1}
 + \pshbbun{ \dot {\mathbf{y}}_2 \cdot \nabla W_2 }{ \Lambda_1 W_1 }+ \pshbbun{{\mathbf R}_{\mathbf W}}{ \Lambda_1 W_1 } .
\end{align*}
By the decay properties of $W$, we note that 
\[
\left| \pshbbun{\eta}{ \Lambda_1 W_1 }\right|
+ \left|\pshbbun{\varepsilon}{ \nabla (\Lambda_1 W_1)} \right|
+\left| \pshbbun {\varepsilon}{ \Lambda_1^2 W_1 }\right| 
+\left|\pshbbun {\varepsilon}{ \nabla\Lambda_1 W_1}\right|\lesssim \|\vec \varepsilon\|_{\dot H^1\times L^2}.
\]
Next, $
 \pshbbun{\Lambda_1 W_1 }{ \Lambda_1 W_1 } 
 = (1-|{\boldsymbol{\ell}}_1|^2)^{\frac 12} \|\Lambda W\|_{\dot H^1}^2 $
 and by parity,
$
 \pshbbun{ \nabla W_1 }{ \Lambda_1 W_1 }=0$.
Using Claim~\ref{WW}, we have
\[
\left| \pshbbun{ \Lambda_2 W_2}{ \Lambda_1 W_1 }\right| 
+ \left| \pshbbun{ \nabla W_2}{ \Lambda_1 W_1 } \right|
 \lesssim t^{-3} 
\]
In conclusion, the orthogonality condition
$\pshbbun { \varepsilon}{ \Lambda_1 W_1}=0$, gives 
\[
\left|\frac{\dot \lambda_1}{\lambda_1} - \frac{a_1}{\lambda_2^{\frac 12}t^2}\right| \lesssim \left\|\vec \varepsilon\right\|_{\dot H^1\times L^2} + |\dot {\mathbf{y}}_1 | \left\|\vec \varepsilon\right\|_{\dot H^1\times L^2}
+ t^{-3}\sum_{k=1, 2} \left(\left|\frac{\dot\lambda_k}{\lambda_k} - \frac{a_k}{\lambda_k^{\frac 12}t^2}\right|+\left|\dot {\mathbf{y}}_k \right|\right)+t^{-4+\delta}.
\]
Using the other orthogonality conditions, 
we obtain
\begin{multline*}
\sum_{k=1,2} \left(\left|\frac{\dot\lambda_k}{\lambda_k} - \frac{a_k}{\lambda_k^{\frac 12}t^2}\right|+\left|\dot {\mathbf{y}}_k \right| \right)
\\ \lesssim \left\|\vec \varepsilon\right\|_{\dot H^1\times L^2} + \left(\left\|\vec \varepsilon\right\|_{\dot H^1\times L^2}+t^{-3}\right) \sum_{k=1,2} \left(\left|\frac{\dot\lambda_k}{\lambda_k} - \frac{a_k}{\lambda_k^{\frac 12}t^2}\right|+\left|\dot {\mathbf{y}}_k \right| \right) + t^{-4+\delta}.\end{multline*} 
Therefore, for $\delta_0$ small enough and $T_0$ large enough, we find~\eqref{le:p}. 
 \smallbreak

\textbf{Equations of the unstable directions.}
Recall that $\vec Z_{{\boldsymbol{\ell}}_k}^{\pm} \in \mathcal S$ by their definition in \S\ref{sec21}.
By~\eqref{syst_e3}, we have
\begin{align*}
 \frac d{dt} z_1^\pm & =\frac d{dt} \psl {\vec \varepsilon}{\vec {\tilde \theta}_1\vec Z_{{\boldsymbol{\ell}}_1}^\pm} = 
\psl {\partial_t \vec \varepsilon}{\vec {\tilde \theta}_1 \vec Z_{{\boldsymbol{\ell}}_1}^{\pm}}+
\psl {\vec \varepsilon}{\partial_t\left(\vec {\tilde \theta}_1 \vec Z_{{\boldsymbol{\ell}}_1}^{\pm}\right)} \\
 & = \psl {\vec {\mathcal L} \vec \varepsilon}{\vec {\tilde \theta}_1 \vec Z_{{\boldsymbol{\ell}}_1}^{\pm}}
 - \frac {{\boldsymbol{\ell}}_1}{\lambda_1} \cdot \psl{\vec \varepsilon}{\vec{\tilde \theta}_1 \nabla\vec Z_{\ell_1}^{\pm}} 
 -\frac{\dot\lambda_1}{\lambda_1} \psl{\vec \varepsilon}{\vec{\tilde \theta}_1 \vec\Lambda \vec Z_{{\boldsymbol{\ell}}_1}^{\pm}} - \frac {\dot {\mathbf{y}}_1}{\lambda_1} \cdot \psl{\vec \varepsilon}{\vec{\tilde \theta}_1 \nabla\vec Z_{{\boldsymbol{\ell}}_1}^{\pm}} 
\\ &\quad + \psl{\vec {\rm Mod}}{\vec {\tilde \theta}_1 \vec Z_{{\boldsymbol{\ell}}_1}^{\pm}} + \psl{\vec{\mathbf{R}}}{\vec {\tilde \theta}_1 \vec Z_{{\boldsymbol{\ell}}_1}^{\pm}} +\psl{\vec{\mathbf{R}}_{\rm NL}}{\vec {\tilde \theta}_1 \vec Z_{{\boldsymbol{\ell}}_1}^{\pm}}
\end{align*}
First, by direct computations, using~(ii) of Lemma~\ref{surZZ},
\begin{align*}
 \psl {\vec {\mathcal L} \vec \varepsilon}{\vec {\tilde \theta}_1 \vec Z_{{\boldsymbol{\ell}}_1}^{\pm}}
 - \frac {{\boldsymbol{\ell}}_1}{\lambda_1} \cdot \psl{\vec \varepsilon}{\vec{\tilde \theta}_1 \nabla\vec Z_{{\boldsymbol{\ell}}_1}^{\pm}} &= \frac 1{\lambda_1} \psl{\vec \varepsilon}{\vec{\tilde \theta}_1 \left( -H_{{\boldsymbol{\ell}}_1} J \vec Z_{{\boldsymbol{\ell}}_1}^\pm\right)} + \psl \varepsilon {f'(W_2) (\theta_1 Z_{{\boldsymbol{\ell}}_1,2}^\pm)}\\
 & = \pm \frac {\sqrt{\lambda_0}}{\lambda_1} (1-|{\boldsymbol{\ell}}_1|^2)^{\frac 12} z_1^\pm 
 + \psl \varepsilon {f'(W_2) (\theta_1 Z_{{\boldsymbol{\ell}}_1,2}^\pm)}.
\end{align*}
By the decay properties of $\vec Z_{\ell_1}^\pm$ and Claim~\ref{WW},
\[
	\left| \psl \varepsilon {f'(W_2) (\theta_1 Z_{{\boldsymbol{\ell}}_1,2}^\pm)}\right| \lesssim t^{-4} {\| \varepsilon\|_{\dot H^1}} .
\]
Next, by~\eqref{le:p},
\[\left| \frac{\dot\lambda_1}{\lambda_1}\right|\left|\psl{\vec \varepsilon}{\vec{\tilde \theta}_1 \vec \Lambda \vec Z_{\ell_1}^{\pm}} \right| +
\left|\frac {\dot {\mathbf{y}}_1}{\lambda_1} \cdot \psl{\vec \varepsilon}{\vec{\tilde \theta}_1 \nabla\vec Z_{{\boldsymbol{\ell}}_1}^{\pm}} \right|
 \lesssim t^{-1} \left\|\vec\varepsilon\right\|_{\dot H^1\times L^2}
 \lesssim \left\|\vec\varepsilon\right\|_{\dot H^1\times L^2}^2 +t^{-4+\delta}\left\|\vec \varepsilon\right\|_{\dot H^1\times L^2}.
\]
Concerning the term with $\vec{\rm Mod}$,
(ii) of Lemma~\ref{surZZ} yields
$(\vec \theta_1 \vec Z_{{\boldsymbol{\ell}}_1}^{\Lambda},\vec {\tilde \theta}_1 \vec Z_{{\boldsymbol{\ell}}_1}^{\pm})= (\vec Z_{{\boldsymbol{\ell}}_1}^{\Lambda},\vec Z_{{\boldsymbol{\ell}}_1}^{\pm}) =0$.
Moreover, by Claim~\ref{WW}, we have
\[
\left|\psl {\vec\theta_2 \vec Z_{\ell_2}^{\Lambda}}
{\vec {\tilde \theta}_1 \vec Z_{\ell_1}^{\pm}}\right|
+\left|\psl{\vec \theta_2 \vec Z_{\ell_2}^{\nabla}}{\vec {\tilde \theta}_1 \vec Z_{\ell_1}^{\pm}}\right| \lesssim t^{-3},
\]
and thus, by~\eqref{le:p},
\[
\left|\frac{\dot\lambda_2}{\lambda_2} - \frac{a_2}{\lambda_2^{\frac 12}t^2}\right| \left|\psl {\vec\theta_2 \vec Z_{\ell_2}^{\Lambda}}
{\vec {\tilde \theta}_1 \vec Z_{\ell_1}^{\pm}}\right|+
\left| \frac {\dot {\mathbf{y}}_2}{\lambda_2} \cdot\psl{\vec \theta_2 \vec Z_{\ell_2}^{\nabla}}{\vec {\tilde \theta}_1 \vec Z_{\ell_1}^{\pm}}\right| \lesssim 
t^{-3} \left\|\vec\varepsilon\right\|_{\dot H^1\times L^2}.
\]
Finally, we claim
\[
\left| \psl{\vec {\mathbf{R}}}{\vec {\tilde \theta}_1 \vec Z_{\ell_1}^{\pm}}\right|+
\left| \psl{\vec {{\mathbf R}}_{\rm NL}}{\vec {\tilde \theta}_1 \vec Z_{\ell_1}^{\pm}}\right|
\lesssim
t^{-4+\delta} + t^{-1} {\| \varepsilon\|_{\dot H^1}} + \| \varepsilon\|_{\dot H^1}^2.
\]
Indeed, from~\eqref{RWX}, we have
$
 |(\vec {\mathbf{R}},\vec {\tilde \theta}_1 \vec Z_{\ell_1}^{\pm})|\lesssim
t^{-4+\delta}$.
Second, by~\eqref{R1bis} and the decay of~$Y$,
$ |(\vec {{\mathbf R}}_1,\vec {\tilde \theta}_1 \vec Z_{\ell_1}^{\pm})| \lesssim \|{{\mathbf R}}_1\|_{L^{\frac {10}7}} \lesssim 
t^{-1}\| \varepsilon\|_{\dot H^1}$ and 
$ | (\vec {{\mathbf R}}_2,\vec {\tilde \theta}_1 \vec Z_{\ell_1}^{\pm})
 | \lesssim \|{{\mathbf R}}_2\|_{L^{\frac {10}7}} \lesssim \| \varepsilon\|_{\dot H^1}^2$.
\smallbreak

The computation for $z_2^\pm$ is the same and we obtain for $k=1,2$,
\[
\left| \frac d{dt} z_k^{\pm}(t) \mp \frac{\sqrt{\lambda_0}}{\lambda_k(t)}(1-|{\boldsymbol{\ell}}_k|^2)^{\frac 12}z_k^{\pm} (t)\right| \lesssim 
t^{-4+\delta}+\left\|\vec\varepsilon(t)\right\|_{\dot H^1\times L^2}^2 +t^{-1} \left\|\vec\varepsilon(t)\right\|_{\dot H^1\times L^2} .
\]
The proof of Lemma~\ref{le:4} is complete.
\end{proof}

\subsection{Bootstrap setting}
We denote by $\mathcal B_{{\mathbb{R}}^2}(r)$ (respectively, $\mathcal S_{{\mathbb{R}}^2}(r)$) the open ball (respectively, the sphere) of ${\mathbb{R}}^2$ of center $0$ and of radius $r>0$, for 
the norm $|(\xi_k)_k|= (\sum_{k=1,2} \xi_k^2 )^{1/2}$. 

For $t=S_n$ and for $t<S_n$ as long as $u_n(t)$ is well-defined in $\dot H^1\times L^2$ and satisfies~\eqref{hyp:4}, we will consider the decomposition of $\vec u_n(t)$ from Lemma~\ref{le:4}. For simplicity of notation, we will denote the parameters $\lambda_{k,n}$, $\mathbf y_{k,n}$ and $\vec\varepsilon_n$ of this decomposition by $\lambda_k$, $\mathbf y_k$ and $\vec\varepsilon$.

We start with a technical result similar to Lemma 3 in~\cite{CMM}. This claim will allow us to adjust the initial values of
$(z_k^\pm(S_n))_k$ from the choice of $\zeta_{k,n}^\pm$ in~\eqref{defun}.
\begin{claim}[Choosing the initial unstable modes] \label{le:modu2}
There exist $n_0>0$ and $C>0$ such that, for all $n\geq n_0$, for any
$(\xi_k)_{k\in\{1,2\}}\in \overline{\mathcal B}_{{\mathbb{R}}^2}(S_n^{-7/2})$, there exists
a unique $(\zeta_{k,n}^{\pm})_{k\in\{1,2\}}\in \mathcal B_{{\mathbb{R}}^2}(C S_n^{-7/2})$ such that
the decomposition of $u_n(S_n)$ satisfies
\begin{equation}\label{modu:2}
z_k^-(S_n)=\xi_k,\quad
z_k^+(S_n)=0,
\end{equation}
\begin{equation}\label{modu3}
|\lambda_k(S_n)-\lambda_k^\infty|+
|\mathbf{y}_k(S_n)-\mathbf{y}_k^\infty|
+ \left\|\vec \varepsilon(S_n)\right\|_{\dot H^1\times L^2}\lesssim S_n^{-7/2}.
\end{equation}
\end{claim}
\begin{proof}[Sketch of the proof of Claim~\ref{le:modu2}]
The proof of existence of $(\zeta_{k,n}^{\pm})_k$ in Claim~\ref{le:modu2} is similar to Lemma~3 in~\cite{CMM} and we omit it.
Estimates in~\eqref{modu3} are consequences of~\eqref{bounds}.
\end{proof}
From now on, for any $ (\xi_{k})_k\in \overline{\mathcal B}_{{\mathbb{R}}^2}(S_n^{-7/2})$,
we fix $(\zeta_{k,n}^{\pm})_k$ as given by Claim~\ref{le:modu2} and the corresponding solution $u_n$ of~\eqref{defun}.
We fix $\delta=\frac 1{20}$.
\smallbreak

The proof of Proposition~\ref{pr:s4} is based on the following bootstrap estimates, for $k=1,2$,
\begin{equation}\label{eq:BS}\left\{\begin{aligned}
& |\lambda_k(t)-\lambda_k^\infty|\leq C_0 t^{-1},\quad |\mathbf{y}_k(t)-\mathbf{y}_k^\infty|\leq C_0 t^{-1}, \\
& |z_k^\pm(t)|^2\leq t^{-7},\quad 
\left\|\vec \varepsilon(t)\right\|_{\dot H^1\times L^2}\leq t^{-3+\frac 1{10}}.\end{aligned}\right.
\end{equation}
Set
\begin{equation}\label{def:tstar}
T^*=T_n^*((\xi_{k})_k)=\inf\{t\in[T_0,S_n]\ ; \ \hbox{$u_n$ satisfies~\eqref{hyp:4} and (\ref{eq:BS}) holds on $[t,S_n]$} \} .
\end{equation}
Note that by Claim~\ref{le:modu2},
estimate~\eqref{eq:BS} is satisfied 
at $t=S_n$.
Moreover, if~\eqref{eq:BS} is satisfied
on $[\tau,S_n]$ for some $\tau\leq S_n$ then by the well-posedness theory 
and continuity, 
$u_n(t)$ is well-defined and satisfies the decomposition of Lemma~\ref{le:4} on $[\tau',S_n]$, 
 for some $\tau'<\tau$. 
In particular, the definition of $T^*$ makes sense. In what follows, we will prove that there exists $T_0$ large enough and at least one choice of $(\xi_{k})_{k}\in \mathcal B_{{\mathbb{R}}^2}(S_n^{-7/2})$ so that $T^*=T_0$, which is enough to finish the proof of Proposition~\ref{pr:s4}.
For this, we derive general estimates for any $(\xi_k)_k\in \overline{\mathcal B}_{{\mathbb{R}}^2}(S_n^{-7/2})$ (see Lemma~\ref{le:bs1}) and use a topological argument (see Lemma~\ref{le:bs2}) to control the instable directions, in order to strictly improve~\eqref{eq:BS} on $[T^*,S_n]$.

\smallbreak
As a consequence of the bootstrap estimates~\eqref{eq:BS} and~\eqref{le:p}, we have, for $k=1,2$,
\[
 \left|\frac {\dot \lambda_k }{\lambda_k} - \frac{a_k}{\lambda_k^{\frac 12}t^2}\right|+|\dot {\mathbf{y}}_k| \lesssim \left\|\vec \varepsilon(t)\right\|_{\dot H^1\times L^2} +t^{-4+\delta} \lesssim t^{-3+\frac 1{10}}.
\]
In particular, from the expression of ${\rm Mod}_{{\mathbf W}}$ and ${\rm Mod}_{{\mathbf X}}$ in Lemma~\ref{le:43}, for all $\alpha \in {\mathbb{N}}^5$,\begin{equation}\label{Idem}
 |\partial_x^\alpha{\rm Mod}_{{\mathbf W}}(t)|\lesssim t^{-3+\frac 1{10}} \sum |W_k|^{1+\frac{|\alpha|}{3}},\quad
 |\partial_x^\alpha{\rm Mod}_{{\mathbf X}}(t)|\lesssim t^{-3+\frac 1{10}} \sum |W_k|^{1+\frac{1+|\alpha|}{3}}.
\end{equation}

\subsection{Energy functional}\label{pr:end}
One of the main points of the proof of Proposition~\ref{pr:s4} is to derive suitable estimates in the energy norm that will strictly improve the bound on $\left\|\vec \varepsilon(t)\right\|_{\dot H^1\times L^2}$ from~\eqref{eq:BS}; the other estimates then follow easily.
In this section, for brievity of notation, we denote 
$f(u)=|u|^{\frac 43} u$ and $F(u)=\frac 3{10} |u|^{\frac {10}3}$.
For
$
0<\sigma<\frac {\ell_2-\ell_1}{10}
$
small enough to be fixed, set
\[
 \ell_1^{+}= \ell_1+\sigma(\ell_2-\ell_1),\quad \ell_2^{-}=\ell_2-\sigma (\ell_2-\ell_1),\quad
\ell_2<\bar \ell <1, 
\]
and for $t>0$,
\[
\Omega(t) = (\ell_1^+ t,\ell_{2}^- t) \times {\mathbb{R}}^4,\quad
\Omega^C(t) = {\mathbb{R}}^5\setminus \Omega(t).
\]
We consider the continuous function $\chi(t,x)=\chi(t,x_1)$ defined as follows, for all $t>0$,
\begin{equation}\label{defchiK}
\left\{\begin{aligned}
 &\hbox{$\chi(t,x) = \ell_1$ for $x_1\in (-\infty, \ell_1^+ t]$},\\
 	& \hbox{$\chi(t,x)= \ell_2$ for $x_1\in [\ell_{2}^- t,+\infty)$},\\
 	& \chi(t,x) = \frac{x_1}{(1-2\sigma)t} - \frac {\sigma}{1-2 \sigma}(\ell_2+\ell_1) 
	\hbox{ for $x_1 \in [\ell_1^+ t,\ell_{2}^-t]$.}
\end{aligned}
\right.
\end{equation}
In particular,
\begin{equation}\label{derchi}\left\{\begin{aligned}
& \partial_t \chi(t,x) =0,\quad \nabla \chi(t,x)=0 \quad \hbox{on $\Omega^C(t)$},\\
 & \partial_{x_1} \chi(t,x)= \frac{1}{(1-2\sigma)t} \quad \hbox{for $x\in \Omega(t)$},\\
 	& \partial_{t} \chi(t,x)= -\frac {x_1}t \frac{1}{(1-2\sigma)t} \quad \hbox{for $x\in \Omega(t)$}.
\end{aligned} \right.\end{equation}
We define (see~\cite{MMT,Ma,MMnls,CMkg,MMwave1,MRnls} for similar functionnals)
\[\mathcal H =\int \left\{|\nabla\varepsilon|^2+|\eta|^2- 2 (F ({\mathbf W} +\varepsilon )-F ({\mathbf W} )-f ({\mathbf W} )\varepsilon ) \right\}
 + 2 \int \chi (\partial_{x_1} \varepsilon ) \eta,
\] 
\begin{lemma}\label{mainprop}
There exists $\mu>0$ such that, for all $t>1$, the following hold.
\begin{enumerate}[label=\emph{(\roman*)}]
\item\emph{Bound.}
\begin{equation}\label{boun}
|\mathcal H(t)| \leq \frac {\|\vec \varepsilon(t)\|_{\dot H^1\times L^2}^2}{\mu}.
\end{equation}
\item\emph{Coercivity.}
\begin{equation}\label{coer}
\mathcal H(t) \geq \mu\|\vec \varepsilon(t)\|_{\dot H^1\times L^2}^2 - \frac {t^{-7}}{\mu} .
\end{equation}
\item\emph{Time variation.} For all $0<\delta<1$,
\begin{equation}\label{time}
- \frac d{dt} \left( t^{2} \mathcal H\right) \lesssim t^{-5+\frac 1{10}+\delta}.
\end{equation}
\end{enumerate}
\end{lemma}
\begin{proof}[Proof of Lemma~\ref{mainprop}]
 \textbf{Proof of~\eqref{boun}.}
 Since
\[
| F({\mathbf W}+\varepsilon) - F({\mathbf W}) - f({\mathbf W}) \varepsilon|\lesssim |\varepsilon|^{\frac {10}3} + |\varepsilon|^2|{\mathbf W}|^{\frac 43},
\]
estimate~\eqref{boun} on $\mathcal H$ follows~\eqref{sobolev},~\eqref{holder1} and $\|\vec \varepsilon\|_{\dot H^1\times L^2}+\|{\mathbf W}\|_{\dot H^1}\lesssim 1$.

\smallbreak

\textbf{Proof of~\eqref{coer}.}
Set
\[ 
{\mathcal N_{\Omega}}(t) = \int_{\Omega} \left( |\nabla \varepsilon(t)|^2 + \eta^2(t) + 2 (\chi(t)\partial_{x_1} \varepsilon(t) ) \eta(t)\right)
,\quad 
{\mathcal N_{\Omega^C}}(t) = \int_{{\Omega^C}} \left( |\nabla \varepsilon(t)|^2 + \eta^2(t) \right).
\]
Note that, since $|\chi|<\overline \ell$,
\begin{equation}\label{nint} 
\begin{aligned}
{\mathcal N_{\Omega}} & = \overline \ell \int_{\Omega} \left|\frac {\chi}{\overline \ell} \partial_{x_1} \varepsilon + \eta\right|^2 
+\int_{\Omega} |\overline \nabla \varepsilon|^2
+ \int_{\Omega} \left( 1- \frac {\chi^2}{\overline \ell} \right) (\partial_{x_1} \varepsilon)^2 
+ (1-\overline \ell) \int \eta^2\\
 & \geq \overline \ell \int_{\Omega} \left|\frac {\chi}{\overline \ell} \partial_{x_1} \varepsilon + \eta\right|^2 
 + (1-\overline \ell) \int_{\Omega}\left( |\nabla \varepsilon|^2 + \eta^2\right).
\end{aligned}
\end{equation} 
We claim the following estimate, for some small $ \gamma>0$,
\begin{equation}\label{lF} 
\mathcal H(t) \geq 
 {\mathcal N_{\Omega}}(t) + \mu {\mathcal N_{\Omega^C}}(t) -\frac {t^{-7}}\mu - \frac{t^{-4\gamma}}\mu \left\|\vec \varepsilon\right\|_{\dot H^1\times L^2}^2 - \frac{1}\mu \left\|\vec \varepsilon\right\|_{\dot H^1\times L^2}^3.
\end{equation}
Note that~\eqref{lF} imples~\eqref{coer}, since ${\mathcal N_{\Omega}}(t) + \mu {\mathcal N_{\Omega^C}}(t)\gtrsim \|\vec \varepsilon\|_{\dot H^1\times L^2}$.
To prove~\eqref{lF}, we decompose 
$
\mathcal H = {\bf f_1}+ {\bf f_2} + {\bf f_3},
$
where
\begin{align*}
{\bf f_1} &
= \int |\nabla \varepsilon|^2 - \sum_k \int f'(W_k) \varepsilon^2 + \int \eta^2 + 2\int (\chi\partial_{x_1} \varepsilon) \eta,
\\{\bf f_2} &
= - 2 \int \left( F\left({\mathbf W} +\varepsilon\right)-F\left({\mathbf W}\right) -f\left({\mathbf W}\right)\varepsilon - \frac 12 f' \left({\mathbf W}\right) \varepsilon^2 \right),
\\{\bf f_3} &
= \int \left( \sum_k f'(W_k) - f' \left({\mathbf W}\right) \right) \varepsilon^2.
\end{align*}
We claim the following estimates
\begin{align}
& {\bf f_1} \geq {\mathcal N_{\Omega}} + \mu {\mathcal N_{\Omega^C}} - \frac {t^{-7}}{\mu} - \frac{t^{-4\gamma}}\mu \left\|\vec \varepsilon\right\|_{\dot H^1\times L^2}^2, \label{bisff2}\\
& |{\bf f_2}|+|{\bf f_3}|\lesssim \left\|\vec\varepsilon\right\|_{\dot H^1\times L^2}^3+t^{-1} \left\|\vec\varepsilon\right\|_{\dot H^1\times L^2}^2 .\label{bisff3}
\end{align}
which imply~\eqref{lF} for $T_0$ large enough.
\smallbreak

\emph{Proof of~\eqref{bisff2}.} For $\varphi_\gamma$ defined in~\eqref{phia}, set
\[
\varphi_k(t,x) = \varphi_\gamma\left(\frac {x-\ell_k \mathbf{e}_1 t - \mathbf{y}_{k}(t)}{\lambda_k(t)}\right).
\]
We decompose $\bf f_1$ as follows
\begin{align*}
{\bf f_1} & = {\mathcal N_{\Omega}} + \sum_k \left(\int |\nabla \varepsilon|^2\varphi_k^2 - \int f'(W_k)\varepsilon^2 + \int \eta^2 \varphi_k^2+ 2 \int (\chi \partial_{x_1} \varepsilon) \eta\varphi_k^2 \right)\\
&\quad +\int_{\Omega^C} \left(|\nabla \varepsilon|^2+\eta^2+2 \chi(\partial_{x_1} \varepsilon)\eta\right)\left(1-\sum_k \varphi_k^2\right)\\
&\quad- \int_{\Omega} \left(|\nabla \varepsilon|^2+\eta^2+2 \chi(\partial_{x_1} \varepsilon)\eta\right)\left(\sum_k \varphi_k^2\right)\\
&\quad +2 \sum_k \int (\chi- \ell_k) (\partial_{x_1} \varepsilon) \eta \varphi_k^2
={\mathcal N_{\Omega}}+ {\bf f_{1,1}}+{\bf f_{1,2}} +{\bf f_{1,3}}+{\bf f_{1,4}}.
\end{align*}
By Lemma~\ref{surZZ} (iii), the orthogonality conditions on $\vec \varepsilon$ and a change of variable, we have
\begin{align*}
{\bf f_{1,1}} & \geq \mu \int \left( |\nabla \varepsilon|^2+\eta^2\right) \left(\sum_k\varphi_k^2\right)
-\frac 1{\mu} \sum_k \left( (z_k^-)^2 + (z_k^+)^2\right).
\end{align*}
Thus, using~\eqref{eq:BS},
\begin{align*}
{\bf f_{1,1}} & \geq \mu \int \left( |\nabla \varepsilon|^2+\eta^2\right) \left(\sum_k\varphi_k^2\right)
-\frac 1{\mu} t^{-7}
\geq \mu \int_{\Omega^C} \left( |\nabla \varepsilon|^2+\eta^2\right) \left(\sum_k\varphi_k^2\right)
-\frac 1{\mu}t^{-7}.
\end{align*}
Next, note that if $x$ is such that
$\varphi_k(t,x) >\frac 12$, then $\varphi_{k'}^2(x) \lesssim t^{-4 \gamma}$ for $k'\neq k$.
Thus, the estimate $1-\sum_k \varphi_k^2 \gtrsim - t^{-4 \gamma}$ holds on ${\mathbb{R}}$.
By direct computations (with the notation $v_+ = \max (0,v)$),
\begin{align*}
{\bf f_{1,2}} & = \overline \ell \int_{\Omega^C} \left|\frac{\chi}{\overline \ell} \partial_{x_1} \varepsilon + \eta\right|^2 \left(1-\sum_k \varphi_k^2\right) +\int_{\Omega^C} |\overline \nabla \varepsilon|^2\left(1-\sum_k \varphi_k^2\right) \nonumber\\
&+ \int_{\Omega^C} \left( 1- \frac{ \chi^2}{\overline\ell}\right) |\partial_{x_1} \varepsilon|^2\left(1-\sum_k \varphi_k^2\right) 
+ (1-\overline\ell) \int_{\Omega^C} \eta^2 \left(1-\sum_k \varphi_k^2\right)
\\ & \geq (1-\overline \ell) \int_{\Omega^C} \left(| \nabla \varepsilon|^2 + \eta^2\right)\left(1-\sum_k \varphi_k^2\right)_+ - \frac { \left\|\vec \varepsilon\right\|_{\dot H^1\times L^2}^2}{t^{4\gamma}}. 
\end{align*}
Also, we see easily that 
$
|{\bf f_{1,3}}| \lesssim t^{-4\gamma} \left\|\vec \varepsilon\right\|_{\dot H^1\times L^2}^2. 
$
Last, by the definition of $\chi$ in~\eqref{defchiK}, the decay property of $\varphi$ and~\eqref{eq:BS} (for a bound on $\mathbf{y}_k$), we have
\[
\|(\chi - \ell_k)\varphi_k\|_{L^\infty} \leq t^{-4 \gamma}.
\]
Thus,
$
|{\bf f_{1,4}}| \lesssim t^{-4\gamma} \left\|\vec \varepsilon\right\|_{\dot H^1\times L^2}^2 .
$

In conclusion, for some $\mu >0$, and $T_0$ large enough, it holds
\[
{\bf f_{1,1}}+{\bf f_{1,2}}+{\bf f_{1,3}}+{\bf f_{1,4}}
\geq \mu {\mathcal N_{\Omega^C}} - \frac 1\mu t^{-7} - t^{-4\gamma} \left\|\vec \varepsilon\right\|_{\dot H^1\times L^2}^2 .
\]
 
\emph{Proof of~\eqref{bisff3}.}
Using~\eqref{sobolev},~\eqref{holder1},~\eqref{e:WX} and~\eqref{eq:BS}, we have
\[
|{\bf f_{2}}| \lesssim\int|\varepsilon|^{\frac{10}3}+|\varepsilon|^3|{\mathbf W}|^{\frac13} \lesssim \left\|\vec\varepsilon\right\|_{\dot H^1\times L^2}^{3}.
\]
Last, we observe that by~\eqref{R1bis},
$
|{\bf f_{3}}| \lesssim \|{\mathbf{R}}_{1}\|_{L^{\frac {10}7}} \|\varepsilon\|_{L^{\frac{10}3}}^2 \lesssim t^{-1} \left\|\vec\varepsilon\right\|_{\dot H^1\times L^2}^2.
$ 
\smallbreak

\textbf{Proof of~\eqref{time}.} 
We decompose
\begin{multline*}
\frac d{dt} \mathcal H = \int \partial_t \left\{|\nabla\varepsilon|^2+|\eta|^2- 2 (F ({\mathbf W} +\varepsilon )-F ({\mathbf W} )-f ({\mathbf W} )\varepsilon ) \right\}\\
 + 2 \int \chi \partial_t\left( (\partial_{x_1} \varepsilon) \eta \right) 
+ 2 \int (\partial_t \chi) (\partial_{x_1} \varepsilon) \eta = {\bf g_1} + {\bf g_2} + {\bf g_3}.
\end{multline*}
We claim the following estimates
\begin{align*} 
{\bf g_1} & = 2 \int \varepsilon \left( -\Delta {\rm Mod}_{{\mathbf W}} - f'({\mathbf W}) {\rm Mod}_{{\mathbf W}} \right) +2 \int \eta{\rm Mod}_{{\mathbf X}} \nonumber\\
& \quad + 2 \int \left( \sum_k \ell_k \partial_{x_1} W_k\right) \left( f({\mathbf W} + \varepsilon) - f({\mathbf W})-f'({\mathbf W}) \varepsilon\right)
 + O\left(t^{-7+\frac 1{10}+\delta}\right),
\end{align*}
\begin{align*}
{\bf g_2} & = 
 -\frac 1{(1-2\sigma) t} \int_{\Omega}
 \left(\eta^2 + (\partial_{x_1} \varepsilon)^2
 - |\overline \nabla \varepsilon|^2 \right)\nonumber\\
& \quad - 2\int \chi\left( \sum_k \partial_{x_1} W_k\right)\left( f({\mathbf W} + \varepsilon) - f({\mathbf W})-f'({\mathbf W}) \varepsilon\right)\nonumber\\
& \quad +2 \int (\chi \partial_{x_1} {\rm Mod}_{{\mathbf W}}) \eta - 2 \int \varepsilon \chi \partial_{x_1} {\rm Mod}_{{\mathbf X}}
+ O\left(t^{-7}\right), 
\\
{\bf g_3} & = -\frac 2{(1-2\sigma)t} \int_{\Omega} \frac{x_1}{t} (\partial_{x_1} \varepsilon) \eta.
\end{align*}

\emph{Estimate on $\bf g_1$.}
From direct differentiation and integration by parts, we have
\begin{align*}
{\bf g_1} & = 2 \int (\partial_t \varepsilon) \left( -\Delta \varepsilon - \left( f({\mathbf W} + \varepsilon) - f({\mathbf W}) \right)\right) +2 \int (\partial_t \eta) \eta\\
& \quad -2 \int (\partial_t {\mathbf W}) \left( f({\mathbf W} + \varepsilon) - f({\mathbf W})-f'({\mathbf W}) \varepsilon \right) .
\end{align*}
Using~\eqref{syst_WX} and~\eqref{syst_e}
\begin{align*}
{\bf g_1} & = 2 \int \left(-\Delta \varepsilon -f'({\mathbf W})\varepsilon\right){\rm Mod}_{\mathbf W} + 2 \int \eta{\rm Mod}_{\mathbf X} \\
&\quad +2\int \left(-\Delta \varepsilon -f'({\mathbf W})\varepsilon\right){{\mathbf R}}_{{\mathbf W}} + \int \eta {{\mathbf R}}_{\mathbf X} \\
& \quad -2 \int {\mathbf X} \left( f({\mathbf W} + \varepsilon) - f({\mathbf W})-f'({\mathbf W}) \varepsilon\right) 
 ={\bf g_{1,1}} + {\bf g_{1,2}}+ {\bf g_{1,3}}\end{align*}
We integrate by parts terms in $\bf g_{1,1}$.
Next, by~\eqref{sobolev},~\eqref{holder1},~\eqref{RWX},~\eqref{e:WX}, and~\eqref{eq:BS}, we obtain
\[
|{\bf g_{1,2}}|\lesssim \|\varepsilon\|_{\dot H^1}\|{{\mathbf R}}_{{\mathbf W}}\|_{\dot H^1} + \|\varepsilon\|_{L^{\frac{10}3}} \|{\mathbf W}\|_{L^{\frac{10}3}}^{\frac 43} \|{{\mathbf R}}_{\mathbf W}\|_{L^{\frac {10}3}}
+ \|\eta\|_{L^2}\|{{\mathbf R}}_{\mathbf X}\|_{L^2} \lesssim t^{-7+\frac 1{10}+\delta}.
\]
Recall from~\eqref{def:Wapp} that ${\mathbf X} = \sum_k \left(\ell_k \partial_{x_1} W_k +c_k z_k\right)$. Moreover,
from the definition of $z_k$ in~\eqref{def.wk} and~\eqref{e:n40}, it follows that
$
\|z_k\|_{L^{\frac {10}3}}\lesssim \|\nabla z_k\|_{L^2}\lesssim t^{-2}.
$
Thus,
\begin{multline*}
\left|{\bf g_{1,3}} - 2 \int \left( \sum_k \ell_k \partial_{x_1} W_k\right) \left( f({\mathbf W} + \varepsilon) - f({\mathbf W})-f'({\mathbf W}) \varepsilon\right)\right|\\
\lesssim 
\int \left(\sum_k |z_k|\right) \left( |\varepsilon|^{\frac 73} +\varepsilon^2 |{\mathbf W}|^{\frac 13}\right)
\lesssim \|z_k\|_{L^{\frac {10}3}} \|\varepsilon\|_{L^{\frac {10}3}}^2 \lesssim t^{-8+\frac 15}\lesssim t^{-7}.
\end{multline*}

\smallbreak

\emph{Estimate on $\bf g_2$.}
\begin{align*}
{\bf g_2} & = 2 \int (\chi \partial_{x_1} \partial_t \varepsilon) \eta+ 2 \int (\chi \partial_{x_1} \varepsilon) \partial_t \eta \\
& = 2 \int (\chi \partial_{x_1} \eta) \eta + 2 \int (\chi \partial_{x_1} \varepsilon) \left[ \Delta \varepsilon + \left( f({\mathbf W} + \varepsilon) - f({\mathbf W}) \right) \right]
 \\& \quad+2 \int (\chi \partial_{x_1} {\rm Mod}_{{\mathbf W}}) \eta + 2 \int (\chi \partial_{x_1} \varepsilon) {\rm Mod}_{{\mathbf X}}
 +2 \int (\chi \partial_{x_1} {{\mathbf R}}_{{\mathbf W}}) \eta + 2 \int (\chi \partial_{x_1} \varepsilon) {{\mathbf R}}_{{\mathbf X}}
\end{align*}
Note that by integration by parts and~\eqref{derchi}
\begin{multline*}
 2 \int (\chi \partial_{x_1} \eta) \eta + 2 \int (\chi \partial_{x_1} \varepsilon) \Delta \varepsilon 
 = - \int \partial_{x_1} \chi \left(\eta^2 +(\partial_{x_1} \varepsilon)^2- |\overline \nabla \varepsilon|^2\right)\\
 =- \frac 1{(1-2\sigma) t} \int_{\Omega}
 \left(\eta^2 + (\partial_{x_1} \varepsilon)^2
 - |\overline \nabla \varepsilon|^2 \right).
\end{multline*}
Next, we observe
\begin{multline*}
 \int (\chi \partial_{x_1} \varepsilon) \left( f({\mathbf W} + \varepsilon) - f({\mathbf W}) \varepsilon\right) 
 \\= \int \chi \partial_{x_1} ( F({\mathbf W} + \varepsilon) - F({\mathbf W})-f({\mathbf W}) \varepsilon )
 - \int \chi (\partial_{x_1} {\mathbf W}) \left( f({\mathbf W} + \varepsilon) - f({\mathbf W})-f'({\mathbf W}) \varepsilon\right).
\end{multline*}
Integrating by parts and using~\eqref{derchi},
\[ -\int \chi \partial_{x_1} ( F({\mathbf W} + \varepsilon) - F({\mathbf W})-f({\mathbf W}) \varepsilon )\\
 = \frac 1{(1-2\sigma) t} \int_{\Omega} \left( F({\mathbf W} + \varepsilon) - F({\mathbf W})-f({\mathbf W}) \varepsilon\right).
\]
Thus, by~\eqref{eq:BS} and 
\[
\|{\mathbf W}\|_{L^{\frac{10}3}(\Omega)}\lesssim \sum_{k} \left(\|W_k\|_{L^{\frac{10}3}(\Omega)}+\|v_k\|_{\dot H^1}\right)
\lesssim t^{-\frac 32},
\]
we obtain
\[
\left| \int \chi \partial_{x_1} ( F({\mathbf W} + \varepsilon) - F({\mathbf W})-f({\mathbf W}) \varepsilon )\right|
\lesssim t^{-\frac 52} \int_\Omega \left(|\varepsilon|^{\frac {10}3} + {\mathbf W}^{\frac 43} |\varepsilon|^2 \right)
\lesssim t^{-10}.
\]
Second, again by~\eqref{e:n40} and~\eqref{eq:BS}
\begin{multline*}
\left| \int \chi \left(\partial_{x_1} {\mathbf W} -\partial_{x_1} \sum W_k\right) \left( f({\mathbf W} + \varepsilon) - f({\mathbf W})-f'({\mathbf W}) \varepsilon\right)\right|
\\ = \left| \int \chi \partial_{x_1} \left(\sum c_k v_k\right) \left( f({\mathbf W} + \varepsilon) - f({\mathbf W})-f'({\mathbf W}) \varepsilon\right)\right|
\lesssim \sum_k \int |\partial_{x_1} v_k| \left( |\varepsilon|^{2} {\mathbf W}^{\frac 13} + |\varepsilon|^{\frac 73}\right)\\
\lesssim \left( \sum_k \|\partial_{x_1} v_k\|_{L^\frac{10}3} \right) \|\varepsilon\|_{L^{\frac {10}3}}^2\lesssim t^{-8+\frac 15}\lesssim t^{-7}.
\end{multline*}
Last, integrating by parts, 
\[
 2 \int (\chi \partial_{x_1} \varepsilon) {\rm Mod}_{{\mathbf X}} 
 = -2 \int (\chi \varepsilon) \partial_{x_1} {\rm Mod}_{{\mathbf X}} + O\left(t^{-7}\right),
\]
since by~\eqref{eq:BS},~\eqref{derchi} and~\eqref{Idem}
\[
\left|\int (\partial_{x_1}\chi)\varepsilon {\rm Mod}_{{\mathbf X}}\right|
 \lesssim t^{-4+\frac 1{10}} \int_{\Omega} |\varepsilon| \left( \sum_k |W_k|^{\frac 43}\right)
 \lesssim t^{-4+\frac 1{10}} \|\varepsilon\|_{L^{\frac {10}3}} \left(\sum_k \|W_k\|_{L^{\frac{10}3}(\Omega)} \right)^{\frac 43}
\lesssim t^{-7}.
\]
Last, we finish the estimate of~$\bf g_2$ by observing that~\eqref{RWX} and~\eqref{eq:BS} yield
\[
\left| \int (\chi \partial_{x_1} {{\mathbf R}}_{{\mathbf W}}) \eta \right|+\left|\int (\chi \partial_{x_1} \varepsilon) {{\mathbf R}}_{{\mathbf X}}\right|
\lesssim \|\eta\|_{L^2} \|\partial_{x_1} {{\mathbf R}}_{{\mathbf W}}\|_{L^2} + \|\partial_{x_1} \varepsilon\|_{L^2} \|{{\mathbf R}}_{{\mathbf X}}\|_{L^2}
\lesssim t^{-7+\frac 1{10}+\delta}.
\]

\emph{Estimate on $\bf g_3$.} This estimate is a direct consequence of~\eqref{derchi}.
\smallbreak

Gathering the above estimates, we rewrite
\[\frac d{dt} \mathcal H = {\bf h_1}+{\bf h_2}+{\bf h_3}+{\bf h_4}+ O\left(t^{-7+\frac 1{10}+\delta}\right),\]
where
\begin{align*}
{\bf h_1} &= -\frac 1{(1-2\sigma) t} \int_{\Omega}
 \left(\eta^2 + (\partial_{x_1} \varepsilon)^2 +2\frac{x_1}{t} (\partial_{x_1} \varepsilon) \eta
 - |\overline \nabla \varepsilon|^2 \right), \\ 
{\bf h_2}& = 2 \int \left(\sum_k \left(\ell_k - \chi\right) \partial_{x_1} W_k\right) \left( f({\mathbf W} + \varepsilon) - f({\mathbf W})-f'({\mathbf W}) \varepsilon\right),\\
{\bf h_3} &= 2 \int \eta\left({\rm Mod}_{{\mathbf X}} + \chi \partial_{x_1} {\rm Mod}_{{\mathbf W}}\right),\\
{\bf h_4} &= 2 \int \varepsilon \left( -\Delta {\rm Mod}_{{\mathbf W}} - \chi \partial_{x_1} {\rm Mod}_{{\mathbf X}}- f'({\mathbf W}) {\rm Mod}_{{\mathbf W}} \right).\end{align*}
First, by~\eqref{nint} and the definition of $\chi$ in~\eqref{defchiK},
\begin{multline*}
-((1-2\sigma)t) {\bf h_1} 
 = \overline\ell \int_{\Omega} \left| \frac{\chi}{\overline \ell} \partial_{x_1} \varepsilon + \eta\right|^2 + \int_{\Omega} \left( 1- \frac{\chi^2}{\overline \ell} \right) (\partial_{x_1} \varepsilon)^2 + (1- {\overline \ell}) \int \eta^2\\
 + 2 \int_{\Omega} \left( \frac{x_1}t- \chi\right) (\partial_{x_1} \varepsilon) \eta 
 \leq {\mathcal N_{\Omega}} + C\sigma \int \left(|\partial_{x_1} \varepsilon|^2 + \eta^2\right)
 \leq (1+C \sigma) {\mathcal N_{\Omega}}.
\end{multline*}
Second, we observe that by the definition of $\chi$ in~\eqref{derchi} and the decay of $\partial_{x_1} W$ and $W$,
\begin{align*}
& \left\| \left(\ell_k - \chi \right) \partial_{x_1}W_k \right\|_{L^{\frac{10}3}}\lesssim t^{-\frac 52}. 
 \end{align*}
Thus, by~\eqref{holder1},
 $ |{\bf h_2}|\lesssim t^{-\frac 52} \|\varepsilon\|_{L^{\frac{10}3}}^2 
 \lesssim t^{-\frac {17}2+\frac 15}\lesssim t^{-8}.
$
 
Denote
 \[
 M_k = \left(\frac {\dot \lambda_k }{\lambda_k} - \frac{a_k}{\lambda_k^{\frac 12}t^2}\right)\Lambda W_k + \dot{\mathbf{y}}_k\cdot \nabla W_k \]
so that
${\rm Mod}_{{\mathbf W}} = \sum_k M_k$ and ${\rm Mod}_{{\mathbf X}} = - \sum_k \ell_k \partial_{x_1} M_k$.
Using~\eqref{e:WX} and the definition of $\chi$ (see~\eqref{derchi}), we have 
$\|(\ell_k-\chi)\partial_{x_1} M_k\|_{L^2}\lesssim t^{-\frac 92+\frac 1{10}}$. 
It follows from~\eqref{le:p}
\[
\left\|{\rm Mod}_{{\mathbf X}} + \chi \partial_{x_1} {\rm Mod}_{{\mathbf W}}\right\|_{L^2} \lesssim t^{-\frac 92+\frac 1{10}},
\]
and thus
\begin{align*}
|{\bf h_3}|=\left| \int \eta\left({\rm Mod}_{{\mathbf X}} + \chi \partial_{x_1} {\rm Mod}_{{\mathbf W}}\right) \right|
\lesssim t^{-\frac 92+\frac 1{10}} \|\eta\|_{L^2} \lesssim t^{-\frac {15} 2+\frac 1{10}}\lesssim t^{-7}.
 \end{align*}
 
Last, we see that 
by~(i) of Lemma~\ref{surZZ},
$
-\Delta M_k + \ell_k^2 \partial_{x_1}^2 M_k - f'(W_k) M_k=0. 
$ Thus,
\begin{align*}
\left| -\Delta M_k + \ell_k \chi \partial_{x_1}^2 M_k - f'({\mathbf W}) M_k\right|
&\lesssim \left| ( \chi -\ell_k) \partial_{x_1}^2 M_k\right| + \left|f'({\mathbf W})-f'(W_k)\right| |M_k|.
\end{align*}
As before, by~\eqref{e:WX},
$
\left\| ( \chi -\ell_k) \partial_{x_1}^2 M_k\right\|_{L^{\frac{10}7}}\lesssim t^{-\frac 92+\frac 1{10}}$.
Moreover, by~\eqref{e:n40} and Claim~\ref{WW},
\begin{multline*}
\left| |{\mathbf W}|^{\frac 43} - |W_k|^{\frac 43}\right| |M_k|
\lesssim \left(\sum |W_k|+\sum |v_k|\right)^{\frac 13} \left(\sum_{k'\neq k}|W_{k'}| + \sum |v_k|\right) |W_k|\\
\lesssim \sum_{k'\neq k''} |W_{k'}|^{\frac 43} |W_{k''}|+t^{-2} \sum_{k'} |W_{k'}|^{2}
+t^{-\frac 23} \sum_{k'\neq k''} |W_{k'}|^{1+\frac 29} |W_{k''}| +t^{-\frac 83} \sum_{k'} |W_{k'}|^{\frac{17}9}\end{multline*}
and so
\[
\left\| \left(f'({\mathbf W})-f'(W_k)\right) M_k\right\|_{L^{\frac {10}7}}\lesssim t^{-2}
\]
Therefore, using~\eqref{le:p},
\[
 \left\| -\Delta {\rm Mod}_{{\mathbf W}} - \chi \partial_{x_1} {\rm Mod}_{{\mathbf X}}- f'({\mathbf W}) {\rm Mod}_{{\mathbf W}} \right\|_{L^{\frac{10}7}}
 \lesssim t^{-\frac 92+\frac 1{10}}.
\]
It follows that by~\eqref{eq:BS},
\[
|{\bf h_4}|\lesssim \|\varepsilon\|_{L^{\frac{10}3}} \left\| -\Delta {\rm Mod}_{{\mathbf W}} - \chi \partial_{x_1} {\rm Mod}_{{\mathbf X}}- f'({\mathbf W}) {\rm Mod}_{{\mathbf W}} \right\|_{L^{\frac{10}7}}
\lesssim t^{-7}.
\]
 
In conclusion, using~\eqref{lF}, for $\sigma$ small, and $T_0$ large,
\[
 -\frac d{dt} \mathcal H
 \leq \frac {(1+C \sigma)} t {\mathcal N_{\Omega}} + O(t^{-7+\frac 1{10}+\delta})
 \leq \frac 2 t \mathcal H+ O(t^{-7+\frac 1{10}+\delta})
\]
and the proof of Lemma~\ref{mainprop} is complete.
\end{proof}
\subsection{Parameters and energy estimates}\label{ssec:cc}
The following result, mainly based on Lemma~\ref{mainprop}, improves all the estimates in~\eqref{eq:BS}, except
the ones on $(z_k^-)_k$.
\begin{lemma}[Closing estimates except $(z_k^-)_k$]\label{le:bs1}
For $C_0>0$ large enough, for all $t\in[T^*,S_n]$,
\begin{equation}\label{eq:BSd}\left\{\begin{aligned}
& |\lambda_k(t)-\lambda_k^\infty|\leq \frac {C_0} 2t^{-1},\quad |\mathbf{y}_k(t)-\mathbf{y}_k^\infty|\leq \frac {C_0}2t^{-1},\\
&|z_k^+(t)|^2\leq \frac 12 t^{-7},\quad 
\left\|\vec \varepsilon(t)\right\|_{\dot H^1\times L^2}\leq \frac 12 t^{-3+\frac 1{10}}.\end{aligned}\right. \end{equation}
\end{lemma}
\begin{proof} 
\textbf{Parameters estimates.} From~\eqref{le:p} and~\eqref{eq:BS}, we have
$\left|\frac {\dot \lambda_k}{\lambda_k}\right|+|\dot {\mathbf y}_k|\leq C t^{-2}$
where the constant $C$ depends on the parameters of the two solitons, but not on $C_0$. By integration on $[t,S_n]$ for 
$T^*\leq t\leq S_n$, and~\eqref{modu3}, we obtain
\[
|\lambda_k(t)-\lambda_k^\infty|\leq |\lambda_k(t)-\lambda_k(S_n)|+|\lambda_k(S_n)-\lambda_k^\infty|\leq C' t^{-1},
\]
and similarly, $|\mathbf{y}_k(t)-\mathbf{y}_k^{\infty}|\lesssim C't^{-1}$, 
where $C'$ is also independent of $C_0$. We choose $C_0=2 C'$.
Now, we prove the bound on $z_k^+(t)$. Let $\beta_k=\frac{\sqrt{\lambda_0}}{\lambda_k^{\infty}}(1-|{\boldsymbol{\ell}}_k|^2)^{1/2}>0$.
Then, from~\eqref{le:z} and~\eqref{eq:BS}, 
\[
\frac d{dt} \left( e^{-\beta_k t} z_k^+\right)\lesssim e^{-\beta_k t} t^{-4+\frac 1{10}}. 
\]
Integrating on $[t,S_n]$ and using~\eqref{modu:2}, we obtain
$- z_k^+(t) \lesssim {t^{-4+\frac 1{10}}}$.
Doing the same for $-e^{-\beta_k t} z_k^+$, we obtain the conclusion for $T_0$ large enough.
\smallbreak
 
\textbf{Bound on the energy norm.}
To prove the estimate on $\left\|\vec \varepsilon(t)\right\|_{\dot H^1\times L^2}$, we use Lemma~\ref{mainprop}.
Recall from~\eqref{modu3} and then~\eqref{boun} that
 $ \mathcal H(S_n) \lesssim {S_n^{-7}}$.
Integrating~\eqref{time} on $[t,S_n]$, we obtain, for all $t\in [T^*,S_n]$,
$
	\mathcal H(t) \lesssim {t^{-6+\frac 1{10}}} .
$
Using~\eqref{coer}, we conclude that
$
\|\vec \varepsilon\|_{\dot H^1\times L^2} \lesssim t^{-3+\frac 1{20}+\frac 12 \delta}$. The estimate follows from the choice $\delta=\frac 1{20}$ and for $T_0$ large enough.
\end{proof}
As in~\cite{CMM,CMkg,MMwave1}, the parameters $z_k^-$ require a specific argument.
\begin{lemma}[Control of unstable directions]\label{le:bs2}
There exist 
$(\xi_{k,n})_{k}\in \mathcal B_{{\mathbb{R}}^2}(S_n^{-7/2})$ such that, for $C^*>0$ large enough, $T^*((\xi_{k,n})_k)=T_0$.
In particular, let $(\zeta_n^\pm)$ be given by Claim~\ref{le:modu2} from such $(\xi_{k,n})_{k}$,
then the solution $u_n$ of~\eqref{defun} satisfies~\eqref{eq:un}.
\end{lemma}
\begin{proof} We follow the strategy of Lemma 6 in~\cite{CMM}. 
 The proof is by contradiction, we assume that for any $(\xi_k)_{k\in\{1,2\}}\in \overline{\mathcal B}_{{\mathbb{R}}^2}(S_n^{-7/2})$,
$T^*((\xi_k)_k)$ defined by~\eqref{def:tstar} satifies $T^*\in (T_0,S_n]$. In this case, by Lemma~\ref{le:bs1} and continuity,
 it holds necessarily
\[
\sum|z_{k}^-(T^*)|^2= \frac{1}{(T^*)^{7}}.
\]
Let $\overline \beta = \min_k \beta_k$.
From~\eqref{le:z} and~\eqref{eq:BS}, for all $t\in [T^*,S_n]$, one has
\[
\frac{d}{dt}\left(t^7 \left(z_{k}^- \right)^2\right) =
2 t^7 z_{k}^- \frac{d}{dt}z_{k}^-
+ 7 t^6 \left(z_{k}^- \right)^2 \leq 
- 2 t^7 \beta_k 
\left(z_{k}^- \right)^2 + \frac{C }{t} \leq -2 \overline \beta t^7 \left(z_{k}^- \right)^2 + \frac{C }{t}.
\]
Thus, for $T_0$ large enough, and any $t_0\in [T^*,S_n]$,
\[
\sum|z_{k}^-(t_0)|^2= \frac{1}{t_0^{7}}\quad \hbox{implies}\quad 
\left. \frac{d }{dt} \left(t^{7} \sum|z_{k}^-(t)|^2\right)\right|_{t=t_0} \leq 
- 2\overline \beta 
+ \frac{C}{T_0} \leq- \overline \beta. 
\]
As a standard consequence of this transversality property, the maps
\[
(\xi_k)_{k}\in \overline{\mathcal B}_{{\mathbb{R}}^2}(S_n^{-7/2}) \mapsto T^*((\xi_k)_k)
\]
and 
\[
(\xi_k)_k\in \overline{\mathcal B}_{{\mathbb{R}}^2}( S_n^{-7/2} ) \mapsto 
\mathcal M((\xi_k)_k)= \left(\frac {T^*}{S_n}\right)^{7/2} (z_k^-(T^*))_k\in \mathcal S_{{\mathbb{R}}^2}( S_n^{-7/2} )\]
are continuous.
Moreover, $\mathcal M$ restricted to $\mathcal S_{{\mathbb{R}}^2}( S_n^{-7/2} )$ is the identity
and this is contradictory with Brouwer's fixed point theorem. 
\smallbreak

Estimates~\eqref{eq:un} follow directly from the estimates~\eqref{eq:BS} on $\varepsilon(t)$, $\lambda_k(t)$, $\mathbf{y}_k(t)$.
\end{proof}
\subsection{Proof of the $\dot H^2\times \dot H^1$ bound}
We introduce a functional of energy type for $\partial_{x_j}\vec \varepsilon $, for any $j=1,\ldots, 5$,
\[
\mathcal F_j = \int \left( |\partial_{x_j} \eta|^2 + |\nabla \partial_{x_j} \varepsilon|^2 - \frac 73 (\partial_{x_j} \varepsilon)^2
|{\mathbf W}+\varepsilon|^{\frac 43} \right).
\]
Note that by~\eqref{holder1}, $\|{\mathbf W}\|_{L^\infty}\lesssim 1$ and~\eqref{eq:BS}, 
\[\int (|{\mathbf W}|^{\frac 43}+|\varepsilon|^{\frac 43}) (\partial_{x_j} \varepsilon)^2\lesssim \|\varepsilon\|_{\dot H^1}^2+\|\varepsilon\|_{\dot H^1}^{\frac 43}\|\partial_{x_j} \varepsilon\|_{\dot H^1}^2\lesssim t^{-6+\frac 15}+t^{-4+\frac 2{15}}\|\nabla\partial_{x_j} \varepsilon\|_{L^2}^2,\]
and so, for $T_0$ large enough,
$\mathcal F_j \geq \frac 12 \int ( |\partial_{x_j} \eta|^2 + |\nabla \partial_{x_j} \varepsilon|^2) - C t^{-6+\frac 15}$.
By~\eqref{Idem} and~\eqref{RWX}, we rewrite~\eqref{syst_WX} and~\eqref{syst_e} as follows
\begin{equation*}
\left\{\begin{aligned}
\partial_t{\mathbf W} & = {\mathbf X} - {\mathbf R}_{\varepsilon}\\
\partial_t	{\mathbf X} & 
= \Delta {\mathbf W} +|{\mathbf W}|^{\frac 43} {\mathbf W} - {\mathbf R}_{\eta}
\end{aligned}\right.\qquad\left\{\begin{aligned}
\partial_t\varepsilon & = \eta +{\mathbf R}_{\varepsilon}\\
\partial_t\eta & 
= \Delta \varepsilon +\left| {\mathbf W} + \varepsilon\right|^{\frac 43} ({\mathbf W} + \varepsilon)
- |{\mathbf W}|^{\frac 43}{\mathbf W} + {\mathbf R}_{\eta} 
\end{aligned}\right.\end{equation*} 
where $\|{\mathbf R}_\varepsilon\|_{\dot H^1\cap H^2}+ \|{\mathbf R}_\eta\|_{H^1}\lesssim t^{-3+\frac 1{10}}$.
We compute
\begin{multline*}
\frac d{dt} \mathcal F_j
= \frac {14}3 \int (\partial_{x_j}\eta)(\partial_{x_j}{\mathbf W} )\left(|{\mathbf W}+\varepsilon|^{\frac 43}-|{\mathbf W}|^{\frac 43}\right)
-\frac {28}9 \int ({\mathbf X}+\eta) ({\mathbf W}+\varepsilon) |{\mathbf W}+\varepsilon|^{-\frac 23} (\partial_{x_j}\varepsilon)^2 \\
+ 2\int (\partial_{x_j} {\mathbf R}_\eta)(\partial_{x_j} \eta) + 2 \int (\nabla\partial_{x_j} {\mathbf R}_\varepsilon) (\nabla\partial_{x_j}\varepsilon)-\frac {14}3\int (\partial_{x_j} {\mathbf R}_\varepsilon)(\partial_{x_j}\varepsilon)|{\mathbf W}+\varepsilon|^{\frac 43}.
\end{multline*}
Thus, 
\begin{multline*}
\left|\frac d{dt} \mathcal F_j\right|
\lesssim \int |\partial_{x_j}\eta||\partial_{x_j}{\mathbf W}| \left(|{\mathbf W}|^{\frac 13}|\varepsilon|+|\varepsilon|^{\frac 43}\right)
+ \int |{\mathbf X}+\eta||{\mathbf W}+\varepsilon|^{\frac 13} (\partial_{x_j}\varepsilon)^2 \\
+ \int |\partial_{x_j} {\mathbf R}_\eta| |\partial_{x_j} \eta|+ \int |\nabla\partial_{x_j}{\mathbf R}_\varepsilon||\nabla\partial_{x_j}\varepsilon| +\int |\partial_{x_j} {\mathbf R}_\varepsilon| |\partial_{x_j}\varepsilon| \left(|{\mathbf W}|^{\frac 43}+|\varepsilon|^{\frac 43}\right)\end{multline*}
Using Holder inequality (in particular,~\eqref{holder1}), Sobolev inequality and~\eqref{eq:BS}, we check the following estimates
\[
\int |\partial_{x_j}\eta||\partial_{x_j}{\mathbf W}| \left(|{\mathbf W}|^{\frac 13}|\varepsilon|+|\varepsilon|^{\frac 43}\right)
\lesssim \|\varepsilon\|_{\dot H^1}\|\eta\|_{\dot H^1}
\lesssim t^{-3+\frac 1{10}} \|\eta\|_{\dot H^1},
\]
\[
\int |{\mathbf X}||{\mathbf W}|^{\frac 13} (\partial_{x_j}\varepsilon)^2\lesssim
\|\varepsilon\|_{\dot H^1}^2\lesssim t^{-6+\frac 15},
\]
\[
\int |\eta||{\mathbf W}|^{\frac 13} (\partial_{x_j}\varepsilon)^2\lesssim
\int |\eta| (\partial_{x_j}\varepsilon)^2\lesssim
\|\eta\|_{L^2}^{\frac 12}\|\eta\|_{\dot H^1}^{\frac 12}\|\varepsilon\|_{\dot H^2}^2
\lesssim t^{-\frac 32+\frac 1{20}}\|\eta\|_{\dot H^1}^{\frac 12}\|\varepsilon\|_{\dot H^2}^2
\]
\[
\int |{\mathbf X}||\varepsilon|^{\frac 13} (\partial_{x_j}\varepsilon)^2\lesssim
\int |\varepsilon|^{\frac 13} (\partial_{x_j}\varepsilon)^2\lesssim
\|\varepsilon\|_{\dot H^1}^{\frac{11}6}\|\varepsilon\|_{\dot H^2}^{\frac 12}\lesssim t^{-\frac{11}2+\frac {11}{60}}\|\varepsilon\|_{\dot H^2}^{\frac 12},
\]
\[
\int |\eta||\varepsilon|^{\frac 13} (\partial_{x_j}\varepsilon)^2\lesssim
\|\varepsilon\|_{\dot H^1}^{\frac 13}\|\eta\|_{\dot H^1}\|\varepsilon\|_{\dot H^2}^2\lesssim t^{-1+\frac 1{30}}\|\eta\|_{\dot H^1}\|\varepsilon\|_{\dot H^2}^2,
\]
\[
\int |\partial_{x_j} {\mathbf R}_\eta| |\partial_{x_j} \eta|+ \int |\nabla\partial_{x_j}{\mathbf R}_\varepsilon||\nabla\partial_{x_j}\varepsilon|
\lesssim t^{-3+\frac 1{10}}\left(\|\eta\|_{\dot H^1}+\|\varepsilon\|_{\dot H^2}\right),
\]
\[
\int |\partial_{x_j} {\mathbf R}_\varepsilon| |\partial_{x_j}\varepsilon| \left(|{\mathbf W}|^{\frac 43}+|\varepsilon|^{\frac 43}\right)
\lesssim 
t^{-3+\frac 1{10}}\|\varepsilon\|_{\dot H^1}+t^{-3+\frac 1{10}}\|\varepsilon\|_{\dot H^1}^{\frac 43}\|\varepsilon\|_{\dot H^2}
\lesssim t^{-6+\frac 15}+t^{-7+\frac 7{30}}\|\varepsilon\|_{\dot H^2}.
\]
We deduce from these estimates that, for $\mathcal F=\sum_{j=1}^5 \mathcal F_j$,
\[
\left|\frac d{dt} \mathcal F\right|
\lesssim t^{-4}+\|\eta\|_{\dot H^1}^3+\|\varepsilon\|_{\dot H^2}^3\lesssim t^{-4}+\left|\mathcal F\right|^{\frac32}.
\]
By~\eqref{modu3}, we known that $|\mathcal F_j(S_n)|\lesssim S_n^{-7}$
and thus, by integration, we obtain the uniform bound $|\mathcal F|\lesssim t^{-3}$ on $[T_0,S_n]$.
It follows that $\|\vec u_n\|_{\dot H^2\times \dot H^1}\lesssim \|\vec {\mathbf W}\|_{\dot H^2\times \dot H^1}+\|\vec \varepsilon\|_{\dot H^2\times \dot H^1}
\lesssim 1$.
\subsection{End of the proof of Proposition~\ref{pr:S1}}
We claim the following property\begin{equation}\label{cpq}
\hbox{for all $\nu>0$, there exists $K>0$ such that, for all $n\geq n_0$, }
\|\vec u_n(T_0)\|_{(\dot H^1\times L^2)(|x|>K)} <\nu.
\end{equation}
\begin{proof}[Proof of~\eqref{cpq}] Let $0<\nu\ll 1$. First, fix $T_1>T_0$ independent of $n$ such that from~\eqref{eq:un},
$\|\vec u_n(T_1)-\vec {\mathbf W}(T_1)\|_{\dot H^1\times L^2} \lesssim T_1^{-3+\delta} <\nu$.
Second, by~\eqref{e:WX}, let $K_1>1$ independent of $n$ be such that $\|\vec {\mathbf W}(T_1)\|_{(\dot H^1\times L^2)(|x|>K_1)}<\nu$.
In particular, it holds $\|\vec u_n(T_1)\|_{(\dot H^1\times L^2)(|x|>K_1)}\lesssim \nu$.

Now, for $0<\gamma\ll 1$ and $K\gg 1$ we consider the function $g_K$ defined on ${\mathbb{R}}^5$ by
\[
g_K(x) = \left(\frac{1+|x|^2}{K^2+|x|^2}\right)^\gamma \hbox{ so that }
|\nabla g_K(x)|\leq 2\gamma\frac {g_K(x)}{|x|}.\]
Note that for any function $v\in \dot H^1$,
\[
\int |\nabla (vg_K)|^2 \leq 2 \int |\nabla v|^2 g_K^2 + 2 \int |v|^2 |\nabla g_K|^2 
\leq 2 \int |\nabla v|^2 g_K^2 + 8 \gamma^2 \int \frac{|vg_K|^2}{|x|^2}.
\]
By the Hardy inequality, $ \int \frac{|vg_K|^2}{|x|^2}\lesssim \int |\nabla (vg_K)|^2 $ and so we can fix $\gamma>0$ small, independently of $K$ and $v$, such that
\begin{equation}\label{ffk}
 \int |\nabla (vg_K)|^2 \leq 4 \int |\nabla v|^2 g_K^2 .
\end{equation}
From now on, $\gamma>0$ is fixed to such value.
Let $K>\max(K_1^2,\nu^{-2/\gamma})$. In particular,
\[
\int |\nabla_{t,x} u_n(T_1)|^2 g_K\lesssim 
\int_{x> \sqrt{K} } |\nabla_{t,x} u_n(T_1)|^2 + g_{K}(\sqrt{K}) \int |\nabla_{t,x} u_n(T_1)|^2\lesssim \nu^2.
\]
By usual computations using~\eqref{wave}, one has
\[
\frac d{dt} \int \left((\partial_t u_n)^2 + |\nabla u_n|^2 -\frac 35 |u_n|^{\frac {10}3}\right) g_K^{\frac {10}3}
= -2 \int \left( \nabla (g_K^{\frac {10}3}) \cdot \nabla u_n\right) \partial_t u_n.
\]
From the expression of $g_K$, one has
\[
\nabla (g_K^{\frac {10}3}) = \frac {10}3 g_K^{\frac 73} \nabla g_K = \frac {20} 3 \gamma x ( K^2-1) (1+|x|^2)^{\gamma-1} (K^2+|x|^2)^{-\gamma-1} g_K^{\frac 73}
\]
and so
$
|\nabla (g_K^{\frac {10}3})|\lesssim K^{-2\gamma},
$
which implies, by the uniform estimates in~\eqref{eq:un}-\eqref{eq:deux}
\[
\left|\frac d{dt} \int \left( (\partial_t u_n)^2 + |\nabla u_n|^2 -\frac 35 |u_n|^{\frac {10}3}\right) g_K^{\frac {10}3} \right|
\lesssim K^{-2\gamma}.
\]
Therefore, integrating on $[T_0,T_1]$ and using the properties of the function $g$,
\[ \int \left((\partial_t u_n)^2 + |\nabla u_n|^2 -\frac 35 |u_n|^{\frac {10}3}\right)(T_0)g_K^{\frac {10}3}
 \lesssim \int \left((\partial_t u_n)^2 + |\nabla u_n|^2 \right)(T_1) g_K^{\frac {10}3} + \frac{T_1-T_0}{K^{2\gamma}}
 \lesssim \nu^2,
\]
by choosing in addition $K$ such that $K^{2\gamma}>\frac {T_1-T_0}{\nu^2}$.

To finish the proof of~\eqref{cpq}, we recall that $\vec u_n = \vec {\mathbf W}+\vec \varepsilon$, where $\vec {\mathbf W}$ satisfies~\eqref{e:WX} and 
$\vec \varepsilon$ satisfies~\eqref{eq:BS}. In particular, for $K$ large depending on $\nu$, but independent of $n$,
\[
\int \left((\eta(T_0))^2 + |\nabla \varepsilon(T_0)|^2\right) g_K^{\frac {10}3}\lesssim
 \int \left((\partial_t u_n(T_0))^2 + |\nabla u_n(T_0)|^2\right)g_K^{\frac {10}3}+ \nu^2,
\]
and
\[
\int |u_n(T_0)|^{\frac {10}3}g_K^{\frac {10}3}
\lesssim \int |{\mathbf W}(T_0)|^{\frac {10}3} g_K^{\frac {10}3}
+\int |\varepsilon(T_0)|^{\frac {10}3}g_K^{\frac {10}3}
\lesssim \nu +\int |\varepsilon(T_0)|^{\frac {10}3}g_K^{\frac {10}3}.
\]
Therefore, 
\[
\int \left((\eta(T_0))^2 + |\nabla \varepsilon(T_0)|^2\right)g_K^{\frac {10}3}\lesssim \int |\varepsilon(T_0)|^{\frac {10}3}g_K^{\frac {10}3}+ \nu^2.
\]
Now, by~\eqref{ffk},
\begin{multline*}
\int |\varepsilon(T_0)|^{\frac {10}3}g_K^{\frac {10}3} \lesssim \left( \int |\nabla (\varepsilon(T_0) g_K)|^2 \right)^{\frac 53}
\lesssim \left( \int |\nabla \varepsilon(T_0)|^2 g_K^2 \right)^{\frac 53} \\
\lesssim \left( \int |\nabla \varepsilon(T_0)|^2 g_K^{\frac {10}3} \right)\left( \int |\nabla \varepsilon(T_0)|^2 \right)^{\frac 23}
\lesssim T_0^{-4+\frac2{15}} \int |\nabla \varepsilon(T_0)|^2 g_K^{\frac {10}3}.
\end{multline*}
Taking $T_0$ larger than a universal constant, we obtain
\[
 \int \left((\partial_t u_n(T_0))^2 + |\nabla u_n(T_0)|^2\right)g_K^{\frac {10}3}
 \lesssim
\int \left((\eta(T_0))^2 + |\nabla \varepsilon(T_0)|^2\right) g_K^{\frac {10}3}
+ \nu^2\lesssim \nu^2,
\]
and~\eqref{cpq} follows from the properties of $g_K$.
\end{proof}
From the estimates of Proposition~\ref{pr:s4} on $(\vec u_n(T_0))$ and~\eqref{cpq}, it follows that 
up to the extraction of a subsequence (still denoted by $(\vec u_n)$), the sequence $(\vec u_n(T_0))$ converges to some $(u_0,u_1)^\mathsf{T}$
in $\dot H^1\times L^2$ as $n\to +\infty$.
Consider the solution $u(t)$ of~\eqref{wave} associated to the initial data $(u_0,u_1)^\mathsf{T}$ at $t=T_0$.
Then, by the continuous dependence of the solution of~\eqref{wave} with respect to its initial data in the energy space $\dot H^1\times L^2$ 
(see \emph{e.g.}~\cite{KM} and references therein) and the uniform bounds~\eqref{eq:un}, the solution $u$ is well-defined in the energy space on $[T_0,\infty)$.

Recall that we denote by $\lambda_{k,n}$ and $\mathbf y_{k,n}$ the parameters of the decomposition of $u_n$ on $[T_0,S_n]$.
By the uniform estimates in~\eqref{eq:deux}, using Ascoli's theorem and a diagonal argument, it follows that there exist continuous functions $\lambda_k$ and $\mathbf y_k$ such that up to the extraction of a subsequence, $\lambda_{k,n}\to \lambda_k$, $\mathbf y_{k,n}\to \mathbf y_k$ uniformly on compact sets of $[T_0,+\infty)$,
and on $[T_0,+\infty)$,
\[
|\lambda_k(t)-\lambda_k^\infty|\lesssim t^{-1},\quad |\mathbf y_k(t)-\mathbf y_k^\infty|\lesssim t^{-1}.
\]
Passing to the limit in~\eqref{eq:un} for any $t\in [T_0,+\infty)$, 
we finish the proof of Proposition~\ref{pr:S1}.

\section{non-zero dispersion}\label{sec:6}
In this section, we finish the proof of Theorem~\ref{th.1} by proving~\eqref{eth:1}.
Let $
R\gg 1$ to be fixed large enough, $t_R = R^{\frac{11}{12}}$, 
and
$
\Sigma_R = \left\{ (t,x)\in {\mathbb{R}}\times {\mathbb{R}}^5 \hbox{ such that } |x|>R+|t-t_R|\right\}$.
Let $u(t)$ be the solution constructed in Proposition~\ref{pr:S1}.

\subsection{Approximate cut-off problem}\label{sec:5.1}
Let $\chi_1:{\mathbb{R}}^5\to {\mathbb{R}}$ be a smooth radially symmetric function such that $\chi_1\equiv 1$ for $|y|>1$
and $\chi_1\equiv 0$ for $|y|<\frac 12$. Let $\chi_R(x)= \chi_1(x/R)$. We define $\vec u_R=(u_R,\partial_t u_R)^\mathsf{T}$ the solution of~\eqref{wave} with the following data at the time $t_R$
\[
u_R(t_R) = u(t_R) \chi_R,\quad \partial_t u_R(t_R) = \partial_t u(t_R) \chi_R.
\]
\begin{claim}\label{cl:S5} For large $R$,
$\|\vec u_R(t_R) \|_{\dot H^1\times L^2} \lesssim R^{-\frac 32}$.
\end{claim}
\begin{proof}
First, by direct computations, using Hardy inequality, at $t=t_R$,
\begin{multline*}
 \int |\nabla u_R|^2 = \int |\nabla u|^2 \chi_R^2 - \int u^2 \chi_R \Delta \chi_R 
 \lesssim \int_{|x|>\frac R2} |\nabla u|^2 + \int_{\frac R2<|x|<R} \frac {u^2}{|x|^2} 
 \\ \lesssim \int_{|x|>\frac R2} \left( |\nabla {\mathbf W}|^2 +\frac{{\mathbf W}^2}{|x|^2} \right)+ \|\nabla (u-{\mathbf W})\|_{L^2}^2.
 \end{multline*}
Note that by~\eqref{e:WX}, 
we have, for $|x|>\frac R2$ and $t=t_R\ll R$,
\[ |\nabla {\mathbf W}(t_R)|^2 +\frac{{\mathbf W}^2(t_R)}{|x|^2} 
 \lesssim 
|x|^{-8} + t_R^{-2} |x|^{-8+\frac 1{10}},
\]
and so,
\[
\int_{|x|>\frac R2} \left( |\nabla {\mathbf W}(t_R)|^2 +\frac{{\mathbf W}^2(t_R)}{|x|^2} \right)
\lesssim R^{-3} + t_R^{-2} R^{-3+\frac 1{10}} \lesssim R^{-3}.
\]
Using also~\eqref{pr:S11},
$\int |\nabla u_R(t_R)|^2
\lesssim R^{-3}+ t_R^{-6+\frac 1{5}}\lesssim R^{-3}$.

The estimate for $\|\partial_t u(t_R)\|_{L^2}$ is similar and easier.
Indeed, at $t=T_R$,
\[
\int |\partial_t u_R|^2 = \int |\partial_t u|^2 \chi_R^2
\lesssim \int_{|x|>\frac R2} |\partial_t u|^2 
\lesssim \int_{|x|>\frac R2} |{\mathbf X}|^2 + \|\partial_t u-{\mathbf X}\|_{L^2}^2 .
\] 
By~\eqref{e:WX}, for $|x|>\frac R2$ and $t=t_R\ll R$, it holds
$ |{\mathbf X}(t_R)|^2 
\lesssim |x|^{-8} + t_R^{-4} |x|^{-6+\frac 15}$,
and so $\int_{|x|>\frac R2} |{\mathbf X}(t_R)|^2 \lesssim R^{-3}.$
Using also~\eqref{pr:S11},
$\int |\partial_t u_R(t_R)|^2\lesssim R^{-3} + t_R^{-6+\frac 1{5}}\lesssim R^{-3}$.
\end{proof}
Using this claim, by the small data Cauchy theory, for $R$ large enough, the solution $\vec u_R$ is global and bounded in $\dot H^1\times L^2$.
Moreover, since
$\vec u_R(t_R) = \vec u(t_R)$, for $|x|>R$,
by the property of finite speed of propagation of the wave equation, 
we can define globally $\vec u(t,x)$ on $\Sigma_R$ by setting $u(t,x)=u_R(t,x)$. This extension makes sense even if $u(t)$ is not global in $\dot H^1\times L^2$ in negative times.
We will prove in this section the following statement, for $R$ large,
\begin{equation}\label{reduction}
\liminf_{t\to -\infty} \|\nabla u(t)\|_{L^2(|x|>R+|t-t_R|)} \gtrsim R^{-\frac 52} ,
\end{equation}
which implies, for $A=R+t_R$ large enough,
\[
\liminf_{t\to -\infty} \|\nabla u(t)\|_{L^2(|x|>|t|+A)} \gtrsim A^{-\frac 52}.
\]
\subsection{Reduction to a linear problem}
We define $\vec u_{\rm L} = (u_{\rm L},\partial_t u_{\rm L})^\mathsf{T}$ the (global) solution of the $5$D linear wave equation with initial data at $t=t_R$,
\begin{equation}\label{uL}\left\{\begin{aligned}
& \partial_t^2 u_{\rm L} - \Delta u_{\rm L} = 0 \quad \hbox{on ${\mathbb{R}}\times {\mathbb{R}}^5$},\\
& u_{\rm L}(t_R)=u_R(t_R) = u(t_R) \chi_R,\quad \partial_t u_{\rm L}(t_R)=\partial_t u_R(t_R) = \partial_t u(t_R) \chi_R \quad
\hbox{on ${\mathbb{R}}^5$.}
\end{aligned}\right.\end{equation}
Using Claim~\ref{cl:S5} and Proposition~\ref{pr:CP}, it follows that for $R$ large enough, 
\begin{equation}\label{reduc_lin}
\sup_{t\in {\mathbb{R}}} \|\vec u_{\rm L}-\vec u_{R}\|_{\dot H^1\times L^2} \lesssim R^{-\frac 73 \cdot \frac 32} = R^{-\frac 72}.
\end{equation}
Therefore it suffices to prove~\eqref{reduction} on $\vec u_{\rm L}$ instead of $\vec u_R$.

We prove a similar result for truncations of solitons.
For any fixed ${\boldsymbol{\ell}}\in {\mathbb{R}}^5$, $|{\boldsymbol{\ell}}|<1$, $\lambda>0$, $\mathbf y\in {\mathbb{R}}^5$ and $\epsilon=\pm 1$, set $\beta=({\boldsymbol{\ell}},\lambda,\mathbf y,\epsilon)$. Denote
\[
w_\beta(t,x)=\frac \epsilon{\lambda^{\frac 32}} W_{\boldsymbol{\ell}}\left(\frac{x-{\boldsymbol{\ell}} t-\mathbf y}{\lambda} \right), \quad \vec w_\beta= \left(\begin{array}{c} w_\beta \\ \partial_t w_\beta \end{array}\right).
\]
Define also $\vec w_{\beta,R}=(w_{\beta,R},\partial_t w_{\beta,R})^\mathsf{T}$ the solution of~\eqref{wave} with truncated data at $t_R$
\[
w_{\beta,R}(t_R) = w_\beta(t_R) \chi_R,\quad \partial_t w_{\beta,R}(t_R) = \partial_t w_\beta(t_R) \chi_R,
\]
and
$\vec w_{\beta,\rm L} = (w_{\beta,\rm L},\partial_t w_{\beta,\rm L})^\mathsf{T}$ the solution of the $5$D linear wave equation with data at $t_R$
\[
 w_{\beta,\rm L}(t_R)= w_{\beta,R}(t_R) = w_\beta(t_R) \chi_R,\quad \partial_t w_{\beta,\rm L}(t_R)=\partial_t w_{\beta,R}(t_R) = \partial_t w_\beta(t_R) \chi_R.
\]
We claim the following on $\vec w_{\beta,\rm L}$.
\begin{claim}\label{cl:wbell}
For any $R$ large enough, for all $t\in {\mathbb{R}}$,
\[
\|\nabla w_{\beta,\rm L}(t)\|_{L^2(|x|>R+|t-t_R|)} + \|\partial_t w_{\beta,\rm L}(t)\|_{L^2(|x|>R+|t-t_R|)}
\lesssim (R+|t|)^{-\frac 32} + R^{-\frac 72}.
\]
\end{claim}
\begin{proof}
First, as in the proof of Claim~\ref{cl:S5}, we see that 
$
\|\vec w_{\beta,R}(t_R) \|_{\dot H^1\times L^2} \lesssim R^{-\frac 32}$.
In particular, for $R$ large enough, the solution $\vec w_{\beta,R}$ is global in $\dot H^1\times L^2$ and, by Proposition~\ref{pr:CP}, \begin{equation}\label{reduc_lin_w}
\sup_{t\in {\mathbb{R}}} \|\vec w_{\beta,\rm L}-\vec w_{\beta,R}\|_{\dot H^1\times L^2} \lesssim R^{-\frac 73 \cdot \frac 32} = R^{-\frac 72}.
\end{equation}
Second, by direct computations, we see that for all $t$,
\begin{multline*}
\|\nabla w_{\beta}(t)\|_{L^2(|x|>R+|t-t_R|)}^2 + \|\partial_t w_{\beta}(t)\|_{L^2(|x|>R+|t-t_R|)}^2
\lesssim \int_{|x|>R+|t-t_R|} |x-{\boldsymbol{\ell}} t|^{-8} dx
 \\ \lesssim \int_{|y|>R+|t-t_R|-|{\boldsymbol{\ell}}| |t|} |y|^{-8} dy
 \lesssim \int_{r>R-|t_R|+(1-|{\boldsymbol{\ell}}|)|t|} r^{-4} dr \lesssim (R+|t|)^{-3},
\end{multline*}
where we have used $t_R=R^{\frac {11}{12}}\ll R$.
For $x\in \Sigma_R$, $w_{\beta}$ and $w_{\beta,R}$ coincide by finite speed of propagation.
\end{proof}

\subsection{Reduction to a radial linear problem}
To use the method of channels of energy, we work on a radial solution. 
Since the solitons $W_1$ and $W_2$ at time $t=t_R$ are not centered at $x=0$, we remove their contribution from the linear solution $u_{\rm L}$ before reducing to a radial problem using Claim~\ref{cl:wbell}. For $k=1,2$, set
\[
\beta_k = ({\boldsymbol{\ell}}_k,\lambda_k(t_R),\mathbf y_k(t_R),\epsilon_k)\quad \hbox{so that}\quad \vec W_k(t_R,x)= \vec w_{\beta_k}(t_R).
\]
In view of Lemma~\ref{le:spherical}, we introduce the radial solution $U_{\rm L}$ of the $5$D linear wave equation,
defined by, for all $t,x\in {\mathbb{R}}^5$, $r=|x|$,
\begin{equation}\label{defUL}
U_{\rm L}(t,x) = \fint_{|y|=|x|} \left( u_{\rm L} - \sum_k w_{\beta_k,\rm L}\right)(t,y) d\omega(y),\quad
\vec U_{\rm L} = \left(\begin{array}{c} U_{\rm L} \\ \partial_t U_{\rm L} \end{array}\right).
\end{equation}
Our goal is to apply Proposition~\ref{pr:ch} to $\vec U_{\rm L}$.
By~\eqref{c.th.1},
\[
\Psi = \frac{(1-\ell_1^2)^{\frac 32} (1-\ell_2^2)^{\frac 32} }{|\ell_1-\ell_2|^3}
 \lambda_1^\infty \lambda_2^\infty \left( \epsilon_1 (\lambda_1^\infty)^{\frac 12}
 +\epsilon_2(\lambda_2^\infty)^{\frac 12}
 \right) \neq 0.
\]
\begin{lemma}\label{le:UL} 
For $R$ large enough, it holds
\[
\|\pi_R \vec U_{\rm L}(t_R)\|_{(\dot H^1\times L^2)(|x|>R)}^2 \gtrsim \Psi^2 R^{-5}.
\]
\end{lemma}
\begin{proof}
Define the radial function $V_{\rm L}$ as follows
\[
V_{\rm L} (x) = \fint_{|y|=|x|} \sum_k c_k v_k(t_R,y) d\omega(y),\quad
\vec V_{\rm L} = \left(\begin{array}{c} V_{\rm L} \\ \partial_t V_{\rm L} \end{array}\right).
\]
We claim the following result on $V_{\rm L}$.
\begin{claim}\label{VLR}
For $R$ large enough, it holds
\[
\|\pi_R^\perp(V_{\rm L},0)\|_{(\dot H^1\times L^2)(|x|>R)}^2 \gtrsim \Psi^2 R^{-5} .
\]
\end{claim}
 \begin{proof}[Proof of Claim~\ref{VLR}]
By the definition of $v_k$ in~\eqref{def.vk}, we have
\[
v_k(t_R,x) = \lambda_k^{-3}(t_R) v_{\ell_k} \left(\frac{t_R}{\lambda_k(t_R)},\frac{x-\mathbf{y}_k(t_R)}{\lambda_k(t_R)}\right).
\]
Similarly, set
\[
v^\sharp_k(t_R,x) = (\lambda_k^\infty)^{-3} v^\sharp_{\ell_k} \left(\frac{t_R}{\lambda_k^\infty},\frac{x}{\lambda_k^\infty}\right),\quad
V^\sharp_{\rm L}(t,x) = -\frac 32(15)^{\frac 32} \fint_{|y|=r} \sum_k \kappa_{\ell_k} c_k v^\sharp_k(t_R,y) d\omega(y)
\]
and 
\[
\tilde v^\sharp_k(t_R,x) = (\lambda_k(t_R))^{-3} v^\sharp_{\ell_k} \left(\frac{t_R}{\lambda_k(t_R)},\frac{x-\mathbf{y}_k(t_R)}{\lambda_k(t_R)}\right)
\]
where $v^\sharp_{\ell_k}$ is defined in~\eqref{eq:v33}. 
By~\eqref{e:n42}, we have
\[
\int_{|x|>R} \left|\nabla \left( v_k +\frac 32(15)^{\frac 32} \kappa_\ell \tilde v^{\sharp}_k \right)(t_R,x) \right|^2 dx 
\lesssim t_R^{-4} R^{-3+2\delta}=R^{-\frac {20}3+2\delta} \lesssim R^{-6}.
\]
Moreover, from $|\lambda_k(t_R)-\lambda^\infty_k|\lesssim t_R^{-1}$, $|\mathbf{y}_k|\lesssim 1$ and~\eqref{asymp2},
we check that
\[
\int_{|x|>R} \left|\nabla \left( v^{\sharp}_k - \tilde v^{\sharp}_k \right)(t_R,x) \right|^2 dx 
\lesssim R^{-6}.
\]
It follows that, for $R$ large,
\[
\int_{|x|>R} \left|\nabla \left( v_k +\frac 32(15)^{\frac 32} \kappa_{\ell_k} v^{\sharp}_k \right)(t_R,x) \right|^2 dx 
\lesssim R^{-6},
\]
and thus
\[
\int_{|x|>R} \left|\nabla \left(V_{\rm L}-V^\sharp_{\rm L}\right)(t_R,x)\right|^2 dx\lesssim R^{-6}.
\]
\smallbreak

Let $\phi_{\ell_k}$ be defined as in~\eqref{def:w} for $\ell=\ell_k$, \emph{i.e.} $\phi_k (t,r) =(\lambda_k^\infty)^{-2} \phi_{\ell_k}\left(\frac{t}{\lambda_k^\infty},\frac r{\lambda_k^\infty}\right)$ and
\[ \phi (t,r) = \frac 32(15)^{\frac 32}\sum_{k}\kappa_{\ell_k} c_k \phi_k(t,r)=r^{-1} \partial_r \left( r^3 V^\sharp_{\rm L}(t,r)\right).
 \]
Then, from~\eqref{asymp} and the definition of $\kappa_\ell$ in~\eqref{def:FG}, for $r>R$, 
\[
\phi(t_R,r)=- \frac 32(15)^{\frac 32}\frac{(W^{\frac 43},\Lambda W)}{\|\Lambda W\|_{L^2}^2} \Psi r^{-3} + O(r^{-1}t^{-\frac 94}),\quad 
 \Psi= \sum_{k} (1-\ell_k^2)^{\frac 32} c_k \lambda_k^\infty .
\]
Using the values of $c_1$, $c_2$ from~\eqref{e:simple}, we see that $\Psi\neq 0$ under the assumption~\eqref{c.th.1}.
In particular, for $R$ large enough,
\[
\int_{r>R} \phi^2(t_R,r) dr \geq C \Psi^2 R^{-5} -C' R^{-1}t_R^{-\frac 92} \gtrsim \Psi^2 R^{-5}.
\]
 From Remark~\ref{rk:proj}, we have
\[
\|\pi_R^\perp(V_{\rm L}^\sharp,0)\|_{(\dot H^1\times L^2)(|x|>R)}^2
=\int_{r>R} \phi^2(t_R,r) dr
\] which finishes the proof of the claim.
 \end{proof}
For $|x|>R$, we have $\vec u_{\rm L}(t_R,x)=\vec u(t_R,x)$ and $\vec w_{\beta_k,\rm L}(t_R)=\vec W_k(t_R,x)$ and thus,
by the definition of $\vec {\mathbf W} = \sum_k (\vec W_k +c_k \vec v_k)$, one has using Proposition~\ref{pr:S1},
\begin{multline*}
\left\|\left(\vec u_{\rm L}-\sum \vec w_{\beta_k,\rm L}\right)(t_R) - \sum c_k \vec v_k(t_R)\right\|_{(\dot H^1\times L^2)(|x|>R)}
\\
=\left\|\left(\vec u-\sum \vec W_k\right)(t_R) - \left( \vec {\mathbf W} - \sum W_k\right)(t_R)\right\|_{(\dot H^1\times L^2)(|x|>R)}\\
=\left\| \vec u (t_R)- \vec {\mathbf W} (t_R)\right\|_{(\dot H^1\times L^2)(|x|>R)}
\lesssim t_R^{-6+\frac 15} = R^{-\frac{319}{60}}.
\end{multline*}
Thus, $\| \vec U_{\rm L}(t_R)-\vec V_{\rm L}\|_{(\dot H^1\times L^2)(|x|>R)} \lesssim R^{-\frac{319}{60}}$, 
which, combined with Claim~\ref{VLR} finishes the proof of the lemma.
\end{proof}
\subsection{Channels of energy}
We finish the proof of Theorem~\ref{th.1}.
Using Lemma~\ref{le:UL} and applying Proposition~\ref{pr:ch} to the function $\vec U_{\rm L}$, we find that for $R$ large enough,
either
\[
\liminf_{t\to -\infty} \|\vec U_{\rm L}(t)\|_{(\dot H^1\times L^2)(|x|>R+|t-t_R|)} \gtrsim R^{-\frac 52} .
\]
or
\[
\liminf_{t\to +\infty} \|\vec U_{\rm L}(t)\|_{(\dot H^1\times L^2)(|x|>R+|t-t_R|)} \gtrsim R^{-\frac 52} .
\]
Now, we transfer this information back to $u(t)$, using $u_{\rm L}(t)$.
By the definition of $\vec U_{\rm L}$ in~\eqref{defUL} and Claim~\ref{cl:wbell}, we have
\begin{multline*}
 \|\vec u_{\rm L}(t)\|_{(\dot H^1\times L^2)(|x|>R+|t-t_R|)}\\
\geq 
 \| \vec U_{\rm L} (t) \|_{(\dot H^1\times L^2)(|x|>R+|t-t_R|)}
-\sum_k \| \vec w_{\beta_k,\rm L}(t) \|_{(\dot H^1\times L^2)(|x|>R+|t-t_R|)}
\\ \geq \| \vec U_{\rm L} (t) \|_{(\dot H^1\times L^2)(|x|>R+|t-t_R|)} - C_1 R^{-\frac 72} -C_2 (R+|t|)^{-\frac 32}.
\end{multline*}
Thus, either
\[
\liminf_{t\to +\infty} \|\vec u_{\rm L}(t)\|_{(\dot H^1\times L^2)(|x|>R+|t-t_R|)} \gtrsim R^{-\frac 52} ,
\]
or
\[
\liminf_{t\to -\infty} \|\vec u_{\rm L}(t)\|_{(\dot H^1\times L^2)(|x|>R+|t-t_R|)}\\
\gtrsim R^{-\frac 52} .
\]
By~\eqref{reduc_lin}, it follows that, for large $R$,
either 
\[
\liminf_{t\to +\infty} \|\vec u(t)\|_{(\dot H^1\times L^2)(|x|>R+|t-t_R|)} \gtrsim R^{-\frac 52} 
\quad \hbox{or}\quad
\liminf_{t\to -\infty} \|\vec u(t)\|_{(\dot H^1\times L^2)(|x|>R+|t-t_R|)} \gtrsim R^{-\frac 52} .
\]
Moroever, by~\eqref{pr:S11} and~\eqref{e:WX}, we have, for any large $R$,
\[
\lim_{t\to +\infty} \|\vec u(t)\|_{(\dot H^1\times L^2)(|x|>R+|t-t_R|)} =0.
\]
Therefore, from Remark~\ref{rk:ch}, we have
both, for large $R$,
\[
\liminf_{t\to -\infty} \|\nabla u(t)\|_{L^2(|x|>R+|t-t_R|)} \gtrsim R^{-\frac 52}
\quad\hbox{and}\quad
\liminf_{t\to -\infty} \|\partial_t u(t)\|_{L^2(|x|>R+|t-t_R|)} \gtrsim R^{-\frac 52} .
\]
\section{Extensions to the case $K\geq 3$}\label{sec:10}
\subsection{Collinear speeds}
For $K\geq 3$ collinear speeds ${\boldsymbol{\ell}}_k = \ell_k \mathbf e_1$ where $-1<\ell_1<\cdots<\ell_K<1$, the existence of a multi-soliton at $+\infty$ is proved in~\cite{MMwave1}. The method used in the present paper to prove Theorem~\ref{th.1} can be extended to this case, using a refined approximate solution $\vec {\mathbf W}$ of the form $\vec {\mathbf W} = \sum_{k=1}^K (\vec W_k+c_k \vec v_k)$).
Similarly as in Lemma~\ref{le:UL}, for $j,k\in \{1,\ldots,K\}$ with $j\neq k$, define
\[
\Psi_{j,k} = \frac{(1-\ell_j^2)^{\frac 32} (1-\ell_k^2)^{\frac 32} }{|\ell_j-\ell_j|^3}
\lambda_j^\infty \lambda_k^\infty \left( \epsilon_j (\lambda_j^\infty)^{\frac 12}
+\epsilon_k(\lambda_k^\infty)^{\frac 12}
\right).
\]
Then, the collision is inelastic under the non-vanishing condition
$\sum_{j\neq k} \Psi_{j,k}\neq 0$. Note that this condition is Lorentz invariant since using the notation of \S\ref{clLor}, for any $-1<\beta<1$,
\[
\frac{(1-\ell_j^2)^{\frac 32} (1-\ell_k^2)^{\frac 32} }{|\ell_j-\ell_j|^3}
=\frac{(1-\tilde \ell_j^2)^{\frac 32} (1-\tilde \ell_k^2)^{\frac 32} }{|\tilde \ell_j-\tilde\ell_j|^3}.
\]
\subsection{Non-collinear speeds} 
The arguments in~\cite{MMwave1} do not apply to $K\geq 3$ for non-collinear speeds. However, under the smallness condition $|{\boldsymbol{\ell}}_k|< \frac 35$, the existence of a multi-soliton with speeds $\{{\boldsymbol{\ell}}_k\}_{1\leq k\leq K}$ can be proved using a refined approximate solution similar to the function ${\mathbf W}$ defined in \S4 and a variant of the energy estimates of \S\ref{pr:end}.
Actually, any further improvement in the approximate solution $\vec {\mathbf W}$ would lead to a existence result with a weaker condition on the speeds. Inelasticity of the collisions then holds under the following general non-vanishing condition
\[
\sum_{k=1}^K (1-|{\boldsymbol{\ell}}_k|^2)^{\frac 32} c_k \lambda_k^\infty\neq 0,
\]
where the coefficients $c_k$ are explicitly defined in Lemma~\ref{le:int}.
\appendix\section{End of the proof of Lemma~\ref{le:B2}}\label{app:3}
By Lemmas~\ref{le:spherical} and~\ref{le:red1D}, we have
\[
\phi_\ell(t,r) 
 = \frac 12 \int_0^{+\infty} \int_{|r-\sigma|}^{r+\sigma} a\left\{\fint_{|x|=a}\left[x\cdot \nabla \left(f^{\sharp}_\ell+g^{\sharp}_\ell\right) +3\left(f^{\sharp}_\ell+g^{\sharp}_\ell\right)\right](t+\sigma)d\omega(x)\right\}da d\sigma.
\]

\textbf{Computation of $x\cdot \nabla \left(f^{\sharp}_\ell+g^{\sharp}_\ell\right)+3 \left(f^{\sharp}_\ell+g^{\sharp}_\ell\right)$.}
We compute
\begin{multline*}
x\cdot\nabla f^{\sharp}_\ell=-3 t^{-3} \left(\frac{x_1(x_1-\ell t)}{1-\ell^2} +|\bar x|^2 \right) \langle x_\ell\rangle^{-5}
\\= -3 t^{-3} \left(\frac{(x_1-\ell t)^2}{1-\ell^2} +|\bar x|^2 \right) \langle x_\ell\rangle^{-5}
- \frac {3\ell}{1-\ell^2}t^{-2}(x_1-\ell t)\langle x_\ell\rangle^{-5}\\
=-3t^{-3} \langle x_\ell\rangle^{-3}+3t^{-3} \langle x_\ell\rangle^{-5}-\frac{3\ell}{1-\ell^2}t^{-2}(x_1-\ell t)\langle x_\ell\rangle^{-5},
\end{multline*}
and so
\[
x\cdot\nabla f^{\sharp}_\ell+3f^{\sharp}_\ell=
3t^{-3} \langle x_\ell\rangle^{-5}-\frac{3\ell}{1-\ell^2}t^{-2}(x_1-\ell t)\langle x_\ell\rangle^{-5}.
\]
Next, $g^{\sharp}_\ell=-\frac{3\ell}{1-\ell^2} t^{-2} (x_1-\ell t) \langle x_\ell\rangle^{-5}$,
\begin{align*}
x\cdot \nabla g^{\sharp}_\ell & = 
-\frac{3\ell}{1-\ell^2} t^{-2} \left( x_1 \langle x_\ell\rangle^{-5} - \frac5{1-\ell^2} x_1(x_1-\ell t)^2\langle x_\ell\rangle^{-7}
-5 (x_1-\ell t) |\bar x|^2\langle x_\ell\rangle^{-7}\right)\\
& = 
-\frac{3\ell}{1-\ell^2} t^{-2} \left( (x_1-\ell t) \langle x_\ell\rangle^{-5} - 5(x_1-\ell t) \left(\frac{(x_1-\ell t)^2}{1-\ell^2} +|\bar x|^2\right) \langle x_\ell\rangle^{-7} \right.\\&\left.\quad +\ell t \langle x_\ell\rangle^{-5} -\frac{5 \ell t}{1-\ell^2} (x_1-\ell t)^2 \langle x_\ell\rangle^{-7}\right)
\\&=
\frac{12 \ell }{1-\ell^2} t^{-2} (x_1-\ell t) \langle x_\ell\rangle^{-5} - \frac {15 \ell }{1-\ell^2} t^{-2} (x_1-\ell t)\langle x_\ell\rangle^{-7}
- \frac {3\ell^2}{1-\ell^2} t^{-1} \langle x_\ell\rangle^{-5} \\
& \quad + \frac {15\ell^2}{(1-\ell^2)^{2}}t^{-1} (x_1-\ell t)^2 \langle x_\ell\rangle^{-7},
\end{align*}
and
\begin{multline*}
x\cdot \nabla g^{\sharp}_\ell +3g^{\sharp}_\ell =\frac{3 \ell }{1-\ell^2} t^{-2} (x_1-\ell t) \langle x_\ell\rangle^{-5} - \frac {15 \ell }{1-\ell^2} t^{-2} (x_1-\ell t)\langle x_\ell\rangle^{-7}\\
- \frac {3\ell^2}{1-\ell^2} t^{-1} \langle x_\ell\rangle^{-5} 
+ \frac {15\ell^2}{(1-\ell^2)^{2}}t^{-1} (x_1-\ell t)^2 \langle x_\ell\rangle^{-7}.
\end{multline*}
Summing up, we find
\begin{multline*}
x\cdot \nabla \left(f^{\sharp}_\ell+g^{\sharp}_\ell\right)+3\left(f^{\sharp}_\ell+g^{\sharp}_\ell\right) =
3t^{-3} \langle x_\ell\rangle^{-5} - \frac {15 \ell }{1-\ell^2} t^{-2} (x_1-\ell t)\langle x_\ell\rangle^{-7}\\
- \frac {3\ell^2}{1-\ell^2} t^{-1} \langle x_\ell\rangle^{-5} + \frac {15\ell^2}{(1-\ell^2)^{2}}t^{-1} (x_1-\ell t)^2 \langle x_\ell\rangle^{-7}.
\end{multline*}

\textbf{Case $\ell=0$.} In this case, we claim that, for $1\ll t\leq r^{\frac{11}{12}}$,
$
\phi_0(t,r)=r^{-3}+O(r^{-1}t^{-\frac 94})$.
Note that $x\cdot \nabla f^{\sharp}_0+3f^{\sharp}_0=3t^{-3} \langle x\rangle^{-5}$ and $g^{\sharp}_0=0$.
Thus,
\begin{align*}
&\phi(t,r) = \frac 32 \int_0^{+\infty} (t+\sigma)^{-3} \int_{|r-\sigma|}^{r+\sigma} a\left(\fint_{|x|=a} \langle x\rangle^{-5} d\omega(x) \right)da d\sigma\\
&\ = \frac 32 \int_0^{+\infty} (t+\sigma)^{-3}\left(\int_{|r-\sigma|}^{r+\sigma} \frac a{(1+a^2)^{\frac 52}}da\right) d\sigma
= -\frac 12 \int_0^{+\infty} (t+\sigma)^{-3} \left[(1+a^2)^{-\frac 32} \right]_{|r-\sigma|}^{r+\sigma} d\sigma\\
&\ = \frac 12\int_0^{+\infty} (t+\sigma)^{-3}(1+(r-\sigma)^2)^{-\frac 32} d\sigma
-\frac 12\int_0^{+\infty} (t+\sigma)^{-3}(1+(r+\sigma)^2)^{-\frac 32} d\sigma.
\end{align*}
First, we estimate
\[
\int_0^{+\infty} (t+\sigma)^{-3}(1+(r+\sigma)^2)^{-\frac 32} d\sigma \lesssim
r^{-3}\int_0^{+\infty} (t+\sigma)^{-3}d\sigma \lesssim r^{-3} t^{-2}.
\]
Second, we compare
\begin{multline*}
\left|\int_0^{+\infty} (t+\sigma)^{-3}(1+(r-\sigma)^2)^{-\frac 32} d\sigma
- (t+r)^{-3} \int_0^{+\infty} (1+(r-\sigma)^2)^{-\frac 32} d\sigma \right|\\
\lesssim r^{-1}t^{-3} \int_0^{+\infty} |r-\sigma| (1+(r-\sigma)^2)^{-\frac 32} d\sigma 
\lesssim r^{-1}t^{-3} \int_{-\infty}^{+\infty} |\sigma'| (1+(\sigma')^2)^{-\frac 32} d\sigma' \lesssim r^{-1}t^{-3}.
\end{multline*}
Third, we compute
\[
\int_0^{+\infty} (1+(r-\sigma)^2)^{-\frac 32} d\sigma
=\int_{-r}^{+\infty} (1+\sigma^2)^{-\frac 32} d\sigma =2+O(r^{-2}).
\]
In conclusion, we have obtained, for $r$, $t$ large, with $t<r^{\frac {11}{12}}$,
\[
\phi(t,x) = r^{-3} +O(r^{-1}t^{-3})+O(r^{-4} t) =r^{-3}+O(r^{-1}t^{-\frac 94}).
\]
\smallbreak

From now on, we focus on the case $0<\ell<1$.
\smallbreak

\textbf{Rewriting $x\cdot \nabla \left(f^{\sharp}_\ell+g^{\sharp}_\ell\right)+3 \left(f^{\sharp}_\ell+g^{\sharp}_\ell\right)$.} First, we compute $\Delta(\langle x_\ell\rangle^{-3})$.
We have
\[
\partial_{x_1}(\langle x_\ell\rangle^{-3})=-\frac3{1-\ell^2}(x_1-\ell t)\langle x_\ell\rangle^{-5},
\]
\[
\partial_{x_1}^2(\langle x_\ell\rangle^{-3})=-\frac3{1-\ell^2}\langle x_\ell\rangle^{-5}
+\frac{15}{(1-\ell^2)^2}(x_1-\ell t)^2\langle x_\ell\rangle^{-7},
\]
and
\[
\bar \Delta(\langle x_\ell\rangle^{-3})=-12 \langle x_\ell\rangle^{-5}+15|\bar x|^2\langle x_\ell\rangle^{-7}.
\]
Thus,
\[
\Delta(\langle x_\ell\rangle^{-3})=-15 \langle x_\ell\rangle^{-5}+3\left(1-\frac{1}{1-\ell^2}\right) \langle x_\ell\rangle^{-5}
+15\left(\frac{(x_1-\ell t)^2}{(1-\ell^2)^2}+|\bar x|^2\right)\langle x_\ell\rangle^{-7},
\]
which we rewrite as follows
\begin{multline*}
\Delta(\langle x_\ell\rangle^{-3})=-15 \langle x_\ell\rangle^{-5}-3\frac{\ell^2}{1-\ell^2} \langle x_\ell\rangle^{-5}
+15\left(\langle x_\ell\rangle^2-1-\frac{(x_1-\ell t)^2}{(1-\ell^2)}\left(1-\frac1{1-\ell^2}\right)\right)\langle x_\ell\rangle^{-7}\\
=-3\frac{\ell^2}{1-\ell^2} \langle x_\ell\rangle^{-5}+15 \frac{\ell^2}{(1-\ell^2)^2}(x_1-\ell t)^2 \langle x_\ell\rangle^{-7}
-15\langle x_\ell\rangle^{-7}.
\end{multline*}
We rewrite
\[
x\cdot \nabla \left(f^{\sharp}_\ell+g^{\sharp}_\ell\right)+3 \left(f^{\sharp}_\ell+g^{\sharp}_\ell\right) = 
f^{\rm I}+f^{\rm II}+f^{\rm III},
\]
where
\begin{align*}
&f^{\rm I}(t,x)=t^{-1}\Delta(\langle x_\ell\rangle^{-3}),\quad f^{\rm II}(t,x)=15t^{-1}\langle x_\ell\rangle^{-7},
\\&f^{\rm III}(t,x)=3t^{-3}\langle x_\ell\rangle^{-5} -15\frac {\ell}{1-\ell^2}t^{-2}(x_1-\ell t)\langle x_\ell\rangle^{-7},
\end{align*}
and set
\[
\phi^{\rm I, II, III}(t,r) = \frac 12 \int_0^{+\infty} \int_{|r-\sigma|}^{r+\sigma} a\left(\fint_{|x|=a}
f^{\rm I, II, III}(t+\sigma,x)d\omega(x)\right)da d\sigma
\]
\smallbreak

\textbf{Computation of $\phi^{\rm I}$.} It is a standard fact that for a smooth function $h$,
\begin{multline*}
\fint_{|x|=r} \Delta h(x)d\omega(x)
= \left(\frac {d^2}{dr^2} +\frac 4r \frac {d}{dr} \right)\left( \fint_{|x|=r} h(x)d\omega(x)\right) 
\\ = \frac 1r \frac {d}{dr}\left (r^{-2}\frac{d}{dr}\left(r^3 \fint_{|x|=r} h(x)d\omega(x)\right)\right).
\end{multline*}
We set $N(x)=\langle x\rangle^{-3}$, $N_\ell(t,x)=N(x_\ell)$, $M(x)=\langle x\rangle^{-5}$, $M_\ell(t,x)=M(x_\ell)$, $K(x)=\langle x\rangle^{-7}$ and $K_\ell(t,x)=K(x_\ell)$.
We have
\begin{multline*}
\phi^{\rm I}(t,r) = \frac 12 \int_0^{+\infty}(t+\sigma)^{-1}\int_{|r-\sigma|}^{r+\sigma} \frac d{da}\left( a^{-2} \frac d{da} \left(a^3 \fint_{|x|=a}
N_\ell(t+\sigma,x) d\omega(x)\right)\right)da d\sigma\\
 = \frac 12 \int_0^{+\infty}(t+\sigma)^{-1}\left[ a^{-2} \frac d{da} \left(a^3 \fint_{|x|=a}
N_\ell(t+\sigma,x) d\omega(x)\right)\right]_{|r-\sigma|}^{r+\sigma} d\sigma 
=\phi^{\rm I,1} +\phi^{\rm I,2} +\phi^{\rm I,3} 
\end{multline*}
where
\[
\phi^{\rm I,1}(t,r)=\frac 12 \int_0^{+\infty}(t+\sigma)^{-1} \left[a^{-2}\frac d{da} \left(a^3 \fint_{|x|=a}
N_\ell(t+\sigma,x) d\omega(x)\right)\right]_{a=r+\sigma} d\sigma
\]
\[
\phi^{\rm I,2}(t,r)=-\frac 12 \int_0^{r}(t+\sigma)^{-1} \left[a^{-2}\frac d{da} \left(a^3 \fint_{|x|=a}
N_\ell(t+\sigma,x) d\omega(x)\right)\right]_{a=r-\sigma} d\sigma
\]
\[
\phi^{\rm I,3}(t,r)=-\frac 12 \int_r^{+\infty}(t+\sigma)^{-1} \left[a^{-2}\frac d{da} \left(a^3 \fint_{|x|=a}
N_\ell(t+\sigma,x) d\omega(x)\right)\right]_{a=\sigma-r} d\sigma
\]
To compute $\phi^{\rm I,1}$, $\phi^{\rm I,2}$ and $\phi^{\rm I,3}$, we will use the following identity
\begin{multline*}
a^{-2} \frac d{da} \left(a^3 \fint_{|x|=a} N_\ell(t+\sigma,x) d\omega(x)\right)
= \fint_{|x|=a} \left( 3N_\ell(t+\sigma,x) +x\cdot \nabla N_\ell(t+\sigma,x)\right) d\omega(x)\\
= 3 \fint_{|x|=a}M_\ell(t+\sigma,x) d\omega(x) + \ell (t+\sigma) \fint_{|x|=a}\partial_{x_1} N_\ell(t+\sigma,x) d\omega(x),
\end{multline*}
since by direct computations
\[
3N_\ell +x\cdot \nabla N_\ell=3\langle x_\ell\rangle^{-3}-3|x_\ell|^2\langle x_\ell\rangle^{-5}+\ell t \partial_{x_1} N_\ell
=3M_\ell+\ell t \partial_{x_1} N_\ell .
\]
To compute $\phi^{\rm I,1}$, we observe as above that
\begin{multline*}
(r+\sigma)^{-2} \frac d{d\sigma} \left((r+\sigma)^3 \fint_{|x|=r+\sigma} N_\ell(t+\sigma,x) d\omega(x)\right)\\
= \fint_{|x|=r+\sigma} \left( 3N_\ell(t+\sigma,x) +x\cdot \nabla N_\ell(t+\sigma,x) \right) d\omega(x)
-\ell (r+\sigma) \fint_{|x|=r+\sigma} \partial_{x_1} N_\ell(t+\sigma,x) d\omega(x) \\
= 3 \fint_{|x|=r+\sigma} M_\ell(t+\sigma,x) d\omega(x) + \ell (t-r) \fint_{|x|=r+\sigma}\partial_{x_1} N_\ell(t+\sigma,x) d\omega(x),
\end{multline*}
and thus eliminating the terms containing $\partial_{x_1} N_{\ell}$, we find
\begin{multline*}
\left[ a^{-2} \frac d{da} \left(a^3 \fint_{|x|=a} N_\ell(t+\sigma,x) d\omega(x)\right) \right]_{a=r+\sigma}
= 3\left( \frac{r+\sigma}{r-t}\right)\fint_{|x|=r+\sigma} M_\ell(t+\sigma,x) d\omega(x)\\
-(r+\sigma)^{-2} \left(\frac{t+\sigma}{r-t}\right)\frac d{d\sigma} \left((r+\sigma)^3 \fint_{|x|=r+\sigma} N_\ell(t+\sigma,x) d\omega(x)\right)
.
\end{multline*}
Therefore, we have obtained
\begin{multline*}
\phi^{\rm I,1}
=-\frac 12 (r-t)^{-1} \int_0^{+\infty}(r+\sigma)^{-2} \frac d{d\sigma} \left((r+\sigma)^3 \fint_{|x|=r+\sigma} N_\ell(t+\sigma,x) d\omega(x)\right)d\sigma
\\ + \frac 32(r-t)^{-1} \int_0^{+\infty}\left( \frac{r+\sigma}{t+\sigma}\right) \fint_{|x|=r+\sigma} M_\ell(t+\sigma,x) d\omega(x) d\sigma.
\end{multline*}
Integrating by parts, we find
\begin{multline*}
\phi^{\rm I,1}
= -(r-t)^{-1} \int_0^{+\infty} \fint_{|x|=r+\sigma} N_\ell(t+\sigma,x) d\omega(x) d\sigma
+\frac r2 (r-t)^{-1} \fint_{|x|=r} N_\ell(t,x) d\omega(x) 
\\ + \frac 32(r-t)^{-1} \int_0^{+\infty} \left(\frac{r+\sigma}{t+\sigma}\right)\fint_{|x|=r+\sigma} M_\ell(t+\sigma,x) d\omega(x) d\sigma,
\end{multline*}
which rewrites
\begin{multline}\label{phiun}
\frac{8\pi^2}3\phi^{\rm I,1}
= -(r-t)^{-1} \int_{|x|>r} |x|^{-4} N_\ell(t+|x|-r,x) dx
+\frac {r^{-3}}2 (r-t)^{-1} \int_{|x|=r} N_\ell(t,x) d\omega(x) 
\\ + \frac 32(r-t)^{-1} \int_{|x|>r} (t+|x|-r)^{-1} |x|^{-3}M_\ell(t+|x|-r,x) dx.
\end{multline}
\smallbreak

We compute $\phi^{\rm I,2}$ similarly. First, for $0<\sigma <r$,
\begin{multline*}
(r-\sigma)^{-2} \frac d{d\sigma} \left((r-\sigma)^3 \fint_{|x|=r-\sigma} N_\ell(t+\sigma,x) d\omega(x)\right)\\
= -\fint_{|x|=r-\sigma} \left( 3N_\ell(t+\sigma,x) +x\cdot \nabla N_\ell(t+\sigma,x) \right) d\omega(x)
-\ell (r-\sigma) \fint_{|x|=r-\sigma} \partial_{x_1} N_\ell(t+\sigma,x) d\omega(x) \\
= -3 \fint_{|x|=r-\sigma} M_\ell(t+\sigma,x) d\omega(x) - \ell (r+t) \fint_{|x|=r-\sigma}\partial_{x_1} N_\ell(t+\sigma,x) d\omega(x),
\end{multline*}
and thus 
\begin{multline*}
\left[ a^{-2} \frac d{da} \left(a^3 \fint_{|x|=a} N_\ell(t+\sigma,x) d\omega(x)\right) \right]_{a=r-\sigma}
= 3\left( \frac{r-\sigma}{r+t}\right)\fint_{|x|=r-\sigma} M_\ell(t+\sigma,x) d\omega(x)\\
-(r-\sigma)^{-2} \left(\frac{t+\sigma}{r+t}\right)\frac d{d\sigma} \left((r-\sigma)^3 \fint_{|x|=r-\sigma} N_\ell(t+\sigma,x) d\omega(x)\right)
.
\end{multline*}
Therefore, 
\begin{multline*}
\phi^{\rm I,2}
=\frac 12 (r+t)^{-1} \int_0^r(r-\sigma)^{-2} \frac d{d\sigma} \left((r-\sigma)^3 \fint_{|x|=r-\sigma} N_\ell(t+\sigma,x) d\omega(x)\right)d\sigma
\\ - \frac 32(r+t)^{-1} \int_0^r\left( \frac{r-\sigma}{t+\sigma}\right) \fint_{|x|=r-\sigma} M_\ell(t+\sigma,x) d\omega(x) d\sigma.
\end{multline*}
Integrating by parts, we find
\begin{multline*}
\phi^{\rm I,2}
= -(r+t)^{-1} \int_0^r \fint_{|x|=r-\sigma} N_\ell(t+\sigma,x) d\omega(x) d\sigma
-\frac r2 (r+t)^{-1} \fint_{|x|=r} N_\ell(t,x) d\omega(x) 
\\ -\frac 32(r+t)^{-1} \int_0^{r} \left(\frac{r-\sigma}{t+\sigma}\right)\fint_{|x|=r-\sigma} M_\ell(t+\sigma,x) d\omega(x) d\sigma,
\end{multline*}
and thus
\begin{multline}\label{phideux}
\frac{8\pi^2}3\phi^{\rm I,2}
= -(r+t)^{-1} \int_{|x|<r} |x|^{-4} N_\ell(t+r-|x|,x) dx
-\frac {r^{-3}}2 (r+t)^{-1} \int_{|x|=r} N_\ell(t,x) d\omega(x) 
\\ - \frac 32(r+t)^{-1} \int_{|x|<r} (t+r-|x|)^{-1} |x|^{-3}M_\ell(t+r-|x|,x) dx.
\end{multline}

\smallbreak

Finally, we compute $\phi^{\rm I,3}$. For $\sigma>r$,
\begin{multline*}
(\sigma-r)^{-2} \frac d{d\sigma} \left((\sigma-r)^3 \fint_{|x|=\sigma-r} N_\ell(t+\sigma,x) d\omega(x)\right)\\
= \fint_{|x|=\sigma-r} \left( 3N_\ell(t+\sigma,x) +x\cdot \nabla N_\ell(t+\sigma,x) \right) d\omega(x)
-\ell (\sigma-r) \fint_{|x|=\sigma-r} \partial_{x_1} N_\ell(t+\sigma,x) d\omega(x) \\
= 3 \fint_{|x|=\sigma-r} M_\ell(t+\sigma,x) d\omega(x) +\ell (r+t) \fint_{|x|=\sigma -r}\partial_{x_1} N_\ell(t+\sigma,x) d\omega(x),
\end{multline*}
and thus 
\begin{multline*}
\left[ a^{-2} \frac d{da} \left(a^3 \fint_{|x|=a} N_\ell(t+\sigma,x) d\omega(x)\right) \right]_{a=\sigma-r}
= -3\left( \frac{\sigma-r}{r+t}\right)\fint_{|x|=\sigma-r} M_\ell(t+\sigma,x) d\omega(x)\\
+(\sigma-r)^{-2} \left(\frac{t+\sigma}{r+t}\right)\frac d{d\sigma} \left((\sigma-r)^3 \fint_{|x|=\sigma-r} N_\ell(t+\sigma,x) d\omega(x)\right)
.
\end{multline*}
Therefore, we write
\begin{align*}
\phi^{\rm I,3}
&=-\frac 12 (r+t)^{-1} \int_r^{+\infty}(\sigma-r)^{-2} \frac d{d\sigma} \left((\sigma-r)^3 \fint_{|x|=\sigma-r} N_\ell(t+\sigma,x) d\omega(x)\right)d\sigma
\\&\quad + \frac 32(r+t)^{-1} \int_r^{+\infty}\left( \frac{\sigma-r}{t+\sigma}\right) \fint_{|x|=\sigma-r} M_\ell(t+\sigma,x) d\omega(x) d\sigma
\end{align*}
and by integration by parts,
\begin{align*}
\phi^{\rm I,3}
&= -(r+t)^{-1} \int_r^{+\infty} \fint_{|x|=\sigma-r} N_\ell(t+\sigma,x) d\omega(x) d\sigma
\\ &\quad+\frac 32(r+t)^{-1} \int_{r}^{+\infty} \left(\frac{\sigma-r}{t+\sigma}\right)\fint_{|x|=\sigma-r} M_\ell(t+\sigma,x) d\omega(x) d\sigma.
\end{align*}
We obtain the following expression concerning $\phi^{\rm I,3}$
\begin{multline}\label{phitrois}
\frac{8\pi^2}3\phi^{\rm I,3}
= -(r+t)^{-1} \int |x|^{-4} N_\ell(t+r+|x|,x) dx
\\ + \frac 32(r+t)^{-1} \int (t+r+|x|)^{-1} |x|^{-3}M_\ell(t+r+|x|,x) dx.
\end{multline}
\smallbreak

\textbf{Asymptotics of $\phi^{\rm I}$.} We extract the asymptotics of $\phi^{\rm I}$ for $r\gg 1$, $1\ll t\leq r^{\frac {11}{12}}$ from the exact expressions~\eqref{phiun},~\eqref{phideux} and~\eqref{phitrois}.
First, in view of~\eqref{phiun}, we set
\begin{align*}
\Gamma_1(\ell)&=\frac 3{8\pi^2} \int_{|y|>1} |y|^{-4} \left( \frac{\left(y_1-\ell |y|+\ell\right)^2}{{1-\ell^2}}+|\bar y|^2\right)^{-\frac 32} dy,\\
\Theta_1(\ell)&=\frac 12\fint_{|y|=1} \left( \frac{|y_1|^2}{1-\ell^2} + |\bar y|^2\right)^{-\frac 32} d \omega(y).
\end{align*}
Observe that $\Gamma_1<+\infty$. Indeed, if $|y|>1$ and $y_1<0$, then
$y_1-\ell |y|+\ell <y_1$ and so $\frac{\left(y_1-\ell |y|+\ell\right)^2}{{1-\ell^2}}+|\bar y|^2\geq |y|^2$.
If $|y|>1$ and $y_1>0$, then $y_1-\ell |y| + \ell > y_1-\ell(y_1+|\bar y|)=(1-\ell)|y_1|-|\bar y|$ and so
$\frac{\left(y_1-\ell |y|+\ell\right)^2}{1-\ell^2}+|\bar y|^2\geq
\frac{\left(y_1-\ell |y|+\ell\right)^2}{4}+|\bar y|^2 \geq \frac{(1-\ell)^2}{8} |y_1|^2-\frac 12 |\bar y|^2+|\bar y|^2\gtrsim |y|^2$.
Thus,
$\Gamma_1(\ell)\lesssim \int_{|y|>1} |y|^{-7} dy<+\infty$.
Using the inequality $|A^{-\frac 32}-B^{-\frac 32}|\lesssim (A^{-\frac 52}+B^{-\frac 52})|A-B|$
and the lower bounds
\[
\frac{(y_1-\ell|y|+\ell -\ell \frac {t}{r})^2}{1-\ell^2}+|\bar y|^2+r^{-2}\gtrsim |y|^2,\quad
\frac{(y_1-\ell|y|+\ell )^2}{1-\ell^2}+|\bar y|^2\gtrsim |y|^{2}
\]
we estimate, for $r\gg 1$, $t\leq r^{\frac {11}{12}}$ large,
\begin{multline*}
\left| r^2 \int_{|x|>r} |x|^{-4} N_\ell(t+|x|-r,x) dx - \frac{8\pi^2}3\Gamma_1(\ell)\right|\\
\lesssim \int_{|y|>1} |y|^{-4} \left|\left(\frac{(y_1-\ell|y|+\ell -\ell \frac {t}{r})^2}{1-\ell^2}+|\bar y|^2+r^{-2}\right)^{-\frac 32}
-\left(\frac{(y_1-\ell|y|+\ell )^2}{1-\ell^2}+|\bar y|^2\right)^{-\frac 32}\right| dy\\
\lesssim \int_{|y|>1} |y|^{-9} \left(\frac {t|y|}r+\frac{t^2}{r^2}\right) dy
\lesssim \frac tr \int_{|y|>1} |y|^{-8} dy \lesssim \frac tr.
\end{multline*}
It follows that
\begin{multline*}
-\frac 3{8\pi^2} (r-t)^{-1} \int_{|x|>r} |x|^{-4} N_\ell(t+|x|-r,x) dx
= - \Gamma_1(\ell) r^{-3} +O(tr^{-4})\\
=- \Gamma_1(\ell) r^{-3} +O(r^{-1}t^{-\frac 94}).
\end{multline*}
Similarly,
\begin{multline*}
\left|\frac 12r^{-1} \int_{|x|=r} N_\ell(t,x) d\omega(x) -\frac{8\pi^2}3 \Theta_1\right| 
\\ \lesssim \fint_{|y|=1} \left|\left( \frac{\left(y_1-\ell \frac tr\right)^2}{1-\ell^2} +|\bar y|^2+r^{-2}\right)^{-\frac 32} - \left(\frac{y_1^2}{1-\ell^2} +|\bar y|^2\right)^{-\frac 32}\right|\left(\frac {t|y|}r+\frac{t^2}{r^2}\right) d\omega(y)
 \lesssim \frac tr,
\end{multline*}
and thus the second term in~\eqref{phiun} is estimated as
\[
\frac 3{8\pi^2} \frac {r^{-3}}2 (r-t)^{-1} \int_{|x|=r} N_\ell(t,x) d\omega(x) = r^{-3}\Theta_1(\ell)+O(r^{-1}t^{-\frac 94}).
\]
Now, we bound the last term in~\eqref{phiun} as follows
\begin{multline*}
(r-t)^{-1} \int_{|x|>r} (t+|x|-r)^{-1} |x|^{-3}M_\ell(t+|x|-r,x) dx\\
\lesssim r^{-4} t^{-1}\int_{|y|>1} |y|^{-3} \left( {\left(y_1-\ell|y|+\ell -\ell \frac {t}{r}\right)^2} +|\bar y|^2+r^{-2}\right)^{-\frac 52}dy\\
\lesssim r^{-4} t^{-1}\int_{y>1} |y|^{-8} dy\lesssim r^{-4} t^{-1} .
\end{multline*}
Thus, $
\phi^{\rm I,1} = -\Gamma_1(\ell) r^{-3} +\Theta_1(\ell) r^{-3} + O(r^{-1}t^{-\frac 94})$.

Second, in view of~\eqref{phideux}, we set
\[
\Gamma_2(\ell)=\frac 3{8\pi^2} \int_{|y|<1} |y|^{-4} \left( \frac{\left(y_1-\ell |y|+\ell\right)^2}{{1-\ell^2}}+|\bar y|^2\right)^{-\frac 32} dy.
\]
Observe that $\Gamma_2(\ell)<+\infty$ for $0<\ell<1$. Indeed, for $|y|<1$, we have if $y_1>0$, 
$|y_1-\ell |y|+\ell|\geq |(1-\ell) y_1+\ell| - \ell |y_1-|y||\geq (1-\ell) y_1+\ell -|\bar y|$ and so
$|y_1-\ell |y|+\ell|^2+|\bar y|^2 \gtrsim |y|^2+\ell^2$. For $|y|<1$, if $y_1<0$,
$|y_1-\ell |y|+\ell|\geq |(1+\ell) y_1 +\ell|-\ell ||y|-|y_1||\geq |(1+\ell) y_1 +\ell|-|\bar y|$ and so
$|y_1-\ell |y|+\ell|^2+|\bar y|^2 \gtrsim |y_1+\frac \ell{1+\ell}|^2+|\bar y|^2$.
Thus, for $0<\ell<1$,
\[
\Gamma_2(\ell)\lesssim \int_{|y|<1} |y|^{-4} \left( \left(y_1+\frac{\ell}{1+\ell}\right)^2+|\bar y|^2\right)^{-\frac 32} dy
<+\infty.
\]
Moreover, using the inequality $|A^{-\frac 32}-B^{-\frac 32}|\lesssim (A^{-\frac 52+\frac 14}+B^{-\frac 52+\frac 14})|A-B|^{\frac 34}$,
we obtain
\begin{multline*}
\left|r^2\int_{|x|<r} |x|^{-4} N_\ell(t+r-|x|,x) dx -\frac{8\pi^2}3 \Gamma_2(\ell) \right|\\
\lesssim \int_{|y|<1} |y|^{-4} \left|\left(\frac{(y_1-\ell|y|+\ell -\ell \frac {t}{r})^2}{1-\ell^2}+|\bar y|^2+r^{-2}\right)^{-\frac 32}
-\left(\frac{(y_1-\ell|y|+\ell )^2}{1-\ell^2}+|\bar y|^2\right)^{-\frac 32}\right| dy\\
\lesssim \left(\frac {t}r\right)^{\frac 34}\int_{|y|<1} |y|^{-4} \left( \left(y_1+\frac{\ell(1-\frac tr)}{1+\ell}\right)^2+|\bar y|^2\right)^{-\frac 52+\frac 14}dy\\
+\left(\frac {t}r\right)^{\frac 34}\int_{|y|<1} |y|^{-4}\left( \left(y_1+\frac{\ell}{1+\ell}\right)^2+|\bar y|^2\right)^{-\frac 52+\frac 14} dy\lesssim \left(\frac {t}r\right)^{\frac 34}.
\end{multline*}
The second term in~\eqref{phideux} writes as before
\[
- \frac 3{8\pi^2}\frac {r^{-3}}2 (r+t)^{-1} \int_{|x|=r} N_\ell(t,x) d\omega(x)=-r^{-3} \Theta_1(\ell) + O(r^{-1}t^{-\frac 94}).
\]
Now, we bound the last term in~\eqref{phideux} as follows
\begin{multline*}
(r+t)^{-1} \int_{|x|<r} (t+r-|x|)^{-1} |x|^{-3}M_\ell(t+r-|x|,x) dx\\
\lesssim r^{-4+\frac 12} t^{-1}\int_{|y|<1} |y|^{-3} \left( \left(y_1-\ell|y|+\ell -\ell \frac {t}{r}\right)^2+|\bar y|^2+r^{-2}\right)^{-\frac 52+\frac 14}dy\\
\lesssim r^{- \frac 72} t^{-1}\int_{|y|<1} |y|^{-3} \left( \left(y_1+\frac{\ell(1-\frac tr)}{1+\ell}\right)^2+|\bar y|^2\right)^{-\frac 52+\frac 14} dy
\lesssim r^{- \frac 72} t^{-1} .
\end{multline*}
In conclusion of these estimates, we obtain
$
\phi^{\rm I,2} = -\Gamma_2(\ell) r^{-3} -\Theta_1(\ell) r^{-3} + O(r^{-1}t^{-\frac 94})$.

Third, in view on~\eqref{phitrois}, we set
\[
\Gamma_3(\ell)=\frac 3{8\pi^2} \int |y|^{-4} \left( \frac{\left(y_1-\ell |y|-\ell\right)^2}{{1-\ell^2}}+|\bar y|^2\right)^{-\frac 32} dy.
\]
Observe that $\Gamma_3(\ell)<+\infty$ for $0<\ell<1$. Indeed, we have if $y_1<0$, $|y_1-\ell |y|-\ell|^2+|\bar y|^2 \gtrsim |y|^2+\ell^2$, and if $y_1>0$,
$|y_1-\ell |y|-\ell|\geq |(1-\ell) y_1 -\ell|-\ell ||y|-|y_1||\geq |(1-\ell) y_1 -\ell|-|\bar y|$ and so
$|y_1-\ell |y|-\ell|^2+|\bar y|^2 \gtrsim |y_1-\frac \ell{1-\ell}|^2+|\bar y|^2$.
Thus, $0<\ell<1$,
\[
\Gamma_3(\ell)\lesssim \int |y|^{-4} \left( \left(y_1-\frac{\ell}{1-\ell}\right)^2+|\bar y|^2\right)^{-\frac 32} dy
<+\infty.
\]
As before, we estimate the first term in~\eqref{phitrois},
\begin{multline*}
\left|r^2\int |x|^{-4} N_\ell(t+r+|x|,x) dx -\frac{8\pi^2}3 \Gamma_3(\ell) \right|\\
\lesssim \int |y|^{-4} \left|\left(\frac{(y_1-\ell|y|-\ell -\ell \frac {t}{r})^2}{1-\ell^2}+|\bar y|^2+r^{-2}\right)^{-\frac 32}
-\left(\frac{(y_1-\ell|y|-\ell )^2}{1-\ell^2}+|\bar y|^2\right)^{-\frac 32}\right| dy\\
\lesssim \left(\frac tr\right)^{\frac 34} \int |y|^{-4}(1+|y|^{\frac 34}) \left( \left(y_1-\frac{\ell}{1-\ell}\right)^2+|\bar y|^2\right)^{-\frac 52+\frac 14} dy\\
+ \left(\frac tr\right)^{\frac 34} \int |y|^{-4}(1+|y|^{\frac 34}) \left( \left(y_1-\frac{\ell(1+\frac tr)}{1-\ell}\right)^2+|\bar y|^2\right)^{-\frac 52+\frac 14} dy \lesssim \left(\frac tr\right)^{\frac 34}.
\end{multline*}
Now, we bound the last term in~\eqref{phitrois} as follows
\begin{multline*}
(r+t)^{-1} \int (t+r+|x|)^{-1} |x|^{-3}M_\ell(t+r+|x|,x) dx\\
\lesssim r^{-4+\frac 12} t^{-1}\int |y|^{-3} \left( \left(y_1-\ell|y|-\ell -\ell \frac {t}{r}\right)^2+|\bar y|^2+r^{-2}\right)^{-\frac 52+\frac 14}dy\\
\lesssim r^{- \frac 72} t^{-1}\int |y|^{-3} \left( \left(y_1-\frac{\ell(1+\frac tr)}{1-\ell}\right)^2+|\bar y|^2\right)^{-\frac 52+\frac 14} dy
\lesssim r^{- \frac 72} t^{-1} .
\end{multline*}
Thus, $
\phi^{\rm I,3} = -\Gamma_3(\ell) r^{-3} + O(r^{-1}t^{-\frac 94})$.

Gathering these estimates, we obtain
$
\phi^{\rm I} = -\Gamma(\ell) r^{-3} + O(r^{-1}t^{-\frac 94})$,
where
\[
\Gamma(\ell)=
\frac {3}{8\pi^2 \ell^2} \int\left[\left( \frac{\left(x_1-\ell |x|+1\right)^2}{{1-\ell^2}}+|\bar x|^2\right)^{-\frac 32}
+\left( \frac{\left(x_1+\ell |x|+ 1\right)^2}{{1-\ell^2}}+|\bar x|^2\right)^{-\frac 32}\right] \frac{dx}{|x|^4}
\]
\smallbreak

\textbf{Computation and asymptotics of $\phi^{\rm II}$.}
Now, we compute the asymptotic of $\phi^{\rm II}$ for $r$ large and $t<r^{\frac {11}{12}}$ large, 
\begin{multline*}
\phi^{\rm II} = \frac {45}{16\pi^2}\int_0^\infty (t+\sigma)^{-1} \int_{|r-\sigma|}^{r+\sigma} a^{-3}
\int_{|x|=a} K_\ell(t+\sigma) d\omega(x) da d\sigma\\
 = \frac {45}{16\pi^2}\int_0^\infty (t+\sigma)^{-1} 
 \int_{|r-\sigma|<|x|<r+\sigma} |x|^{-3} K_\ell(t+\sigma)dx d\sigma
\end{multline*}
Note that for $|x|>r+\sigma$, we have
$|x-\ell \mathbf{e}_1 (t+\sigma)|\geq |x|-\ell (t+\sigma)\geq (1-\ell) |x|$, and so
\[
\int_{|x|>(r+\sigma)} |x|^{-3} K_\ell(t+\sigma) dx
\lesssim \int_{|x|>(r+\sigma)} |x|^{-10} dx \lesssim (r+\sigma)^{-5}.
\]
Thus,
\[
\left|\phi^{\rm II} - \frac {45}{16\pi^2}\int_0^\infty (t+\sigma)^{-1} 
 \int_{|r-\sigma|<|x|} |x|^{-3} K_\ell(t+\sigma)dx d\sigma\right|\lesssim r^{-4} t^{-1}.
\]
We remark that $|r-\sigma|<\ell (t+\sigma)$ is equivalent to
$\frac{r-\ell t}{1+\ell} < \sigma < \frac{r+\ell t}{1-\ell}$.
Thus it is natural to decompose the integral according to the three regions $0<\sigma<\frac{r-\ell t}{1+\ell} $, $\sigma >\frac{r+\ell t}{1-\ell}$
and $\frac{r-\ell t}{1+\ell} < \sigma < \frac{r+\ell t}{1-\ell}$.

First, for $0<\sigma<\frac{r-\ell t}{1+\ell}<\frac r{1+\ell}<r$ and $|x|>r-\sigma\gtrsim r$, we observe that
\[(|x_1-\ell (t+\sigma)|^2+|\bar x|^2)^{\frac 12} \geq |x|-\ell (t+\sigma)\geq r-\sigma -\ell(t+\sigma)
\gtrsim \frac{r-\ell t}{1+\ell}-\sigma.\]
Thus, using the change of variable $\sigma=\frac{r-\ell t}{1+\ell}\sigma'$,
\begin{multline*}
\int_0^{\frac{r-\ell t}{1+\ell}} (t+\sigma)^{-1} 
 \int_{|r-\sigma|<|x|} |x|^{-3}K_\ell(t+\sigma) dx d\sigma \\
\lesssim 
t^{-1}r^{-3} \int_0^{\frac{r-\ell t}{1+\ell}} \left(\frac{r-\ell t}{1+\ell}-\sigma\right)^{-\frac 34} 
 \int \left( (x_1-\ell(t+\sigma))^2+|\bar x|^2+1\right)^{-3-\frac 18} dx d\sigma\\
\lesssim t^{-1} r^{-3+\frac 14} \int_0^1 (1-\sigma')^{-\frac 34} d\sigma'\lesssim r^{-1}t^{-\frac 94}.
\end{multline*}

Second, for $\sigma>\frac{r+\ell t}{1-\ell}\geq \frac r{1-\ell}\geq r$, $|x|>\sigma -r$, we observe that
\[(|x_1-\ell (t+\sigma)|^2+|\bar x|^2)^{\frac 12} \geq |x|-\ell (t+\sigma)\geq \sigma-r -\ell(t+\sigma)
\gtrsim \sigma-\frac{r+\ell t}{1-\ell}.\]
Thus,
\begin{multline*}
\int_{\frac{r+\ell t}{1-\ell}}^{+\infty} (t+\sigma)^{-1} 
 \int_{|r-\sigma|<|x|} |x|^{-3} K_\ell(t+\sigma) dx d\sigma \\
\lesssim r^{-1}
 \int_{\frac{r+\ell t}{1-\ell}}^{+\infty} \sigma^{-3} \left(\sigma-\frac{r+\ell t}{1-\ell}\right)^{-\frac 34} 
 \int \left( (x_1-\ell(t+\sigma))^2+|\bar x|^2+1\right)^{-3-\frac 18} dx d\sigma\\
\lesssim r^{-4+\frac 14} \int_1^{+\infty} (\sigma')^{-3} (1-\sigma')^{-\frac 34} d\sigma'\lesssim r^{-\frac {15}{4}}
\lesssim r^{-1}t^{-\frac 94}.
\end{multline*}

Third, we consider the region $\frac{r-\ell t}{1+\ell} < \sigma < \frac{r+\ell t}{1-\ell}$.
We observe that for $|x|>10(t+\sigma)$, we have
$|x-\ell \mathbf{e}_1 (t+\sigma)|\geq |x|-\ell (t+\sigma)\geq \frac 12 |x|$, and so
\begin{multline*}
\int_{\frac{r-\ell t}{1+\ell}}^{\frac{r+\ell t}{1-\ell}} (t+\sigma)^{-1}
\int_{|x|>10\ell (t+\sigma)} \left( |x|^{-3}+(t+\sigma)^{-3}\right) K_\ell(t+\sigma)dx d\sigma\\
\lesssim \int_{\frac{r-\ell t}{1+\ell}}^{\frac{r+\ell t}{1-\ell}} (t+\sigma)^{-4} \int_{|x|>10(t+\sigma)} |x|^{-7} dx d\sigma
\lesssim \int_{\frac{r-\ell t}{1+\ell}}^{\frac{r+\ell t}{1-\ell}} (t+\sigma)^{-6}d\sigma\lesssim r^{-5}.
\end{multline*}
Next, using the inequality 
$|A^{-3}-B^{-3}|\lesssim (A^{-4}+B^{-4})|A-B|$, we observe that
\[\left| |x|^{-3} - (\ell (t+\sigma))^{-3}\right| 
\lesssim \left( |x|^{-4}+(t+\sigma)^{-4}\right) \left(|x_1-\ell(t+\sigma)|+|\bar x| \right) 
\]and thus using the change of variable $x=(t+\sigma)y$,
\begin{multline*}
 \int_{\frac{r-\ell t}{1+\ell}}^{\frac{r+\ell t}{1-\ell}} (t+\sigma)^{-1}
\int_{|r-\sigma|<|x|<10\ell (t+\sigma)} \left| |x|^{-3} - (\ell (t+\sigma))^{-3}\right| K_\ell(t+\sigma) dx d\sigma \\
\lesssim \int_{\frac{r-\ell t}{1+\ell}}^{\frac{r+\ell t}{1-\ell}} (t+\sigma)^{-1}
\int_{|r-\sigma|<|x|<10\ell (t+\sigma)} \left( |x|^{-4}+(t+\sigma)^{-4}\right) \left(|x_1-\ell(t+\sigma)|+|\bar x| \right)^{-\frac {19}4}dx d\sigma\\
\lesssim \int_{\frac{r-\ell t}{1+\ell}}^{\frac{r+\ell t}{1-\ell}} (t+\sigma)^{-\frac {19}4}
\int_{ |y|<10 } \left( |y|^{-4}+1\right) \left(|y_1-\ell|+|\bar y| \right)^{-\frac {19}4}dyd\sigma
\lesssim r^{-\frac {15}4}.
\end{multline*}

Fourth, we observe that for $\frac{r-\ell t}{1+\ell}<\sigma<r$, we have
$\ell (t+\sigma)>r-\sigma$ and so for $|x|<r-\sigma$, 
$(|x_1-\ell (t+\sigma)|^2+|\bar x|^2)^{\frac 12} \geq \ell (t+\sigma)-|x|
\gtrsim \sigma-\frac{r-\ell t}{1+\ell}$. Thus, the following holds
\begin{multline*}
\int_{\frac{r-\ell t}{1+\ell}}^{r} (t+\sigma)^{-4} 
\int_{|x|<|r-\sigma|} K_\ell(t+\sigma) dx d\sigma \\
\lesssim 
r^{-4} \int_{\frac{r-\ell t}{1+\ell}}^{r} \left(\sigma-\frac{r-\ell t}{1+\ell}\right)^{-\frac 34} 
 \int \left( (x_1-\ell(t+\sigma))^2+|\bar x|^2+1\right)^{-3-\frac 18} dx d\sigma'\lesssim r^{-\frac {15}{4}},
\end{multline*}
and similarly,
\[
\int_r^{\frac{r+\ell t}{1-\ell}} (t+\sigma)^{-4} 
 \int_{|x|<|r-\sigma|} K_\ell(t+\sigma) dx d\sigma '\lesssim r^{-\frac {15}{4}}.
\]
Therefore,
\[
 \phi^{\rm II} =\frac {45}{16\pi^2 \ell^3}\int_{\frac{r-\ell t}{1+\ell}}^{\frac{r+\ell t}{1-\ell}} (t+\sigma)^{-4} 
 \int K_\ell(t+\sigma) dx d\sigma +O(r^{-1}t^{-\frac 94}).
\]
By change of variable, we see that
$
\int K_\ell(t+\sigma)=\frac{8\pi^2}{15}(1+3\ell^{-2})^{-1}\Theta(\ell),
$ 
where
\[
\Theta(\ell)=\frac {15(1+3\ell^{-2})}{16\pi^2} \int \left( \frac{|y_1|^2}{1-\ell^2} + |\bar y|^2+1\right)^{-\frac 72} dy.
\]
Moreover, for $1\ll t<r^{\frac{11}{12}}$,
\begin{multline*}
\ell^{-3} \int_{\frac{r-\ell t}{1+\ell}}^{\frac{r+\ell t}{1-\ell}} (t+\sigma)^{-4}
=\ell^{-3} \int_{\frac{r}{1+\ell}}^{\frac{r}{1-\ell}} \sigma^{-4} d\sigma + O(r^{-3-\frac 1{12}})
\\=
\frac {\ell^{-3}r^{-3}} 3 \left( (1+\ell)^3-(1-\ell)^3\right)+ O(r^{-3-\frac 1{12}})
=\frac 23 (1+3\ell^{-2}) r^{-3}+ O(r^{-3-\frac 1{12}}).
\end{multline*}
It follows that
$
\phi^{\rm II} = \Theta(\ell) r^{-3}+ O(r^{-1}t^{-\frac 94})$.
\smallbreak

\textbf{Estimate of $\phi^{\rm III}$.}
We observe that for $t$ large,
$
|f^{\rm III}(t,x)|\lesssim t^{-2} \langle x_\ell\rangle^{-5}$. For $\sigma>0$ fixed,
\begin{multline*}\left| \int_{|r-\sigma|}^{r+\sigma} a\left(\fint_{|x|=a}
f^{\rm III}(t+\sigma,x)d\omega(x)\right)da \right|
\lesssim (t+\sigma)^{-2}\int_{|r-\sigma|}^{+\infty} a^{-3} \int_{|x|=a} M_\ell(t+\sigma,x) d\omega(x)da\\
\lesssim (t+\sigma)^{-2}\int_{|x|>|r-\sigma|} |x|^{-3}\left(|x_1-\ell (t+\sigma)|^2+|\bar x|^2\right)^{-\frac 52+\frac 14} dx\\
\lesssim (t+\sigma)^{-5+\frac 12} \int_{|y|>\frac{|r-\sigma|}{t+\sigma}} |y|^{-3} \left(|y_1-\ell|^2+|\bar y|^2 \right)^{-\frac 52 +\frac 14} dy
\\\lesssim (t+\sigma)^{-2}(|r-\sigma|+t+\sigma)^{-3+\frac 12} \lesssim (t+\sigma)^{-2} r^{-3+\frac 12}.
\end{multline*}
Thus,
$
|\phi^{\rm III}(t,r)|\lesssim r^{-3+\frac 12} \int_0^{+\infty} (t+\sigma)^{-2} d\sigma \lesssim r^{-1} t^{-\frac 94}$.
\smallskip

\textbf{Conclusion.} For $1\ll t\leq r^{\frac {11}{12}}$, we have obtained
$
\phi_\ell(t,r) = (\Theta(\ell)-\Gamma(\ell)) r^{-3} + O(r^{-1}t^{-\frac 94})$.
First,
\begin{align*}
\Theta(\ell)& = \frac {15}{8\pi^2} (3\ell^{-2}+1)(1-\ell^2)^{\frac12} \int \left( |y_1|^2 + |\bar y|^2+1\right)^{-\frac 72} dy
\\&=5 (3\ell^{-2}+1)(1-\ell^2)^{\frac12}\int_0^{+\infty} r^4 (1+r^2)^{-\frac 72} dr=(3\ell^{-2}+1)(1-\ell^2)^{\frac12}.
\end{align*}
Second, we compute $\Gamma(\ell)$. Note that
\begin{align*}
&\frac {1}{2\pi^2} \int \left( \frac{\left(x_1-\ell |x|+1\right)^2}{{1-\ell^2}}+|\bar x|^2\right)^{-\frac 32}
\\&\quad= \int_0^{+\infty}\int_{-r}^{r} \left(\frac{(a-\ell r+1)^2}{1-\ell^2}+r^2-a^2\right)^{-\frac 32}
\left(1-\frac{a^2}{r^2}\right) da\frac{dr}r\\
&\quad=(1-\ell^2)^{\frac 32} \int_{-\infty}^0\int_{-1}^{1} \left( {(rb+\ell r+1)^2} +r^2(1-b^2)(1-\ell^2)\right)^{-\frac 32}\left(1-b^2\right) db {dr},
\end{align*}
and similarly
\begin{align*}
&\frac {1}{2\pi^2} \int \left( \frac{\left(x_1+\ell |x|+ 1\right)^2}{{1-\ell^2}}+|\bar x|^2\right)^{-\frac 32} \frac{dx}{|x|^4} 
\\&\qquad=(1-\ell^2)^{\frac 32} \int_0^{+\infty}\int_{-1}^{1} \left( {(rb+\ell r+1)^2} +r^2(1-b^2)(1-\ell^2)\right)^{-\frac 32}
\left(1-b^2\right) db {dr}.
\end{align*}
Thus, by direct computation
\begin{align*}
\Gamma(\ell)
& =\frac 34 \ell^{-2}(1-\ell^2)^{\frac 32} \int_{-\infty}^{+\infty}\int_{-1}^{1} \left( {(rb+\ell r+1)^2} +r^2(1-b^2)(1-\ell^2)\right)^{-\frac 32}\left(1-b^2\right) db {dr}\\
& = \frac 34 \ell^{-2}(1-\ell^2)^{\frac 32} \int_{-\infty}^{+\infty}\int_{-1}^{1} \left(\left( r(1+b\ell)+\frac {b+\ell}{1+b\ell} \right)^2 +\frac{(1-b^2)(1-\ell^2)}{(1+b\ell)^2}\right)^{-\frac 32}\left(1-b^2\right) db {dr}\\
& = \frac 34 \ell^{-2}(1-\ell^2)^{\frac 12}\left(\int_{-1}^{1} (1+b\ell) db\right)\left(\int_{-\infty}^{+\infty} \left(u^2+1\right)^{-\frac 32}du\right)=\ell^{-2}(1-\ell^2)^{\frac 12}.
\end{align*}
Therefore, $\phi_\ell(t,r) = (1-\ell^2)^{\frac 12} r^{-3} + O(r^{-1}t^{-\frac 94})$ and Lemma~\ref{le:B2} is proved. 

\end{document}